%% file: article_10072023.tex
\documentclass[final,hidelinks,onefignum,onetabnum]{siamart220329}

\input{ex_shared}

\ifpdf
\hypersetup{
  pdftitle={},
  pdfauthor={}
}
\fi

\usepackage{comment}

\newcommand{\ipK}[2]{\ip{#1}{#2}_K}
\newcommand{\nn}{{n_\star}}

\newcommand{\ff}{g}
\newcommand{\FF}{G}

\begin{document}

\maketitle

\begin{abstract}
Controlling spurious oscillations is crucial for designing reliable numerical schemes for hyperbolic conservation laws. This paper proposes a novel, robust, and efficient oscillation-eliminating discontinuous Galerkin (OEDG) method on general meshes, motivated by the damping technique in 
[Lu, Liu, and Shu, {\em SIAM J. Numer. Anal.}, 59:1299--1324, 2021].  
The OEDG method incorporates an OE procedure after each Runge--Kutta stage, and it is devised by alternately evolving the conventional semidiscrete DG scheme and a damping equation. 
A novel damping operator is carefully designed to possess both {\em scale-invariant} and {\em evolution-invariant} properties.  
We rigorously prove {\em optimal error estimates} of the {\em fully discrete} OEDG method for smooth solutions of linear scalar conservation laws. This might be the first generic fully-discrete error estimates for {\em nonlinear} DG schemes with an automatic oscillation control mechanism. 
The OEDG method exhibits many notable advantages. It effectively eliminates spurious oscillations for challenging problems spanning various scales and wave speeds, all without necessitating problem-specific parameters. It also obviates the need for characteristic decomposition in hyperbolic systems. Furthermore, it retains the key properties of the conventional DG method, such as conservation, optimal convergence rates, and superconvergence.  
Moreover, the OEDG method maintains stability under the normal CFL condition, even in the presence of strong shocks associated with highly stiff damping terms. 
The OE procedure is {\em non-intrusive}, facilitating seamless integration into existing DG codes as an independent module. Its implementation is straightforward and efficient, involving only simple multiplications of modal coefficients by scalars. 
The OEDG approach provides new insights into the damping mechanism for oscillation control. It reveals the role of the damping operator as a modal filter and establishes close relations between the damping technique and spectral viscosity techniques. 
Extensive numerical results confirm the theoretical analysis and validate the effectiveness and advantages of the OEDG method.

\end{abstract}



\section{Introduction}


Discontinuous Galerkin (DG) methods, originally introduced by Reed and Hill \cite{reed1973triangular} in 1973, are a class of finite element methods employing discontinuous piecewise polynomial spaces. 
Since the series of pioneering work by Cockburn and Shu \cite{rkdg1,rkdg2,rkdg3,rkdg4,rkdg5}, the DG methods coupled with Runge--Kutta (RK) time discretization have become a prominent approach for solving hyperbolic conservation laws of the form 
\begin{equation}\label{HCL}
	{\bf u}_t + \nabla \cdot {\bf f}({\bf u}) = {\bf 0}, 
\end{equation}
with applications in diverse fields including computational fluid dynamics. In this paper, we further develop the DG methods and focus on a novel numerical approach to control spurious oscillations near discontinuities, such as shocks.

The solution of nonlinear hyperbolic conservation laws may develop discontinuities within finite time, even starting from smooth initial conditions. This poses significant challenges in numerical simulations, as many schemes may generate spurious oscillations near these discontinuities. 
Such oscillations can give rise to nonphysical wave structures, numerical instability, or even result in blow-up solutions. Consequently, controlling spurious oscillations is paramount in the design of robust numerical schemes for hyperbolic conservation laws.
For DG methods,  two primary strategies exist to mitigate oscillations. The first strategy is to apply suitable limiters, such as the total variation diminishing limiter, total variation bounded limiter, and weighted essentially nonoscillatory type limiters; see, for example,   \cite{shu2009discontinuous,qiu2005runge,zhong2013simple}. These limiters serve to adjust the slopes or point values of the solution in some troubled cells, acting as postprocessors for the DG solution after each time step or RK stage. While effective, limiters may be problem-dependent and can potentially alter some desirable properties of the original DG scheme. 
Another strategy is to incorporate appropriate artificial viscosity terms with second or higher order spatial derivatives to diffuse off the oscillations. We refer to \cite{zingan2013implementation,hiltebrand2014entropy,yu2020study,huang2020adaptive} and the references therein for related works.  


In recent works \cite{lu2021oscillation,LiuLuShu_OFDG_system}, Lu, Liu, and Shu systematically developed the so-called oscillation-free DG (OFDG) method, which innovatively incorporates damping terms to suppress spurious oscillations. The semidiscrete OFDG method can be expressed as an ordinary differential equation (ODE) system:   
\begin{equation}\label{eq:OFDG}
	\frac{{\rm d} {\bf U} }{{\rm d}t} + {\bm L}_{\bf f}({\bf U}) + {\bf \Sigma}({\bf U}){\bf U} = {\bf 0}, 
\end{equation}
where ${\bf U}$ denotes the DG modal coefficients, ${\bm L}_{\bf f}({\bf U})$ corresponds to the conventional DG spatial discretization of $\nabla \cdot {\bf f}({\bf u})$, and ${\bf \Sigma}({\bf U}){\bf U}$ represents the added damping terms penalizing the deviation between the DG solution and its low-order polynomial projections.  
The OFDG method shares a similar flavor with local projection stabilization schemes \cite{becker2004two,braack2006local}. 
A notable feature of the OFDG scheme is the inclusion of predefined damping coefficients ${\bf \Sigma}({\bf U})$, which allows for its applicability to a wide range of common problems. 
In theory, Lu, Liu, and Shu \cite{lu2021oscillation} rigorously established   
the stability, optimal error estimates, and superconvergence of the OFDG method for scalar conservation laws. 
The extension to  systems of hyperbolic conservation laws was further studied in \cite{LiuLuShu_OFDG_system}, and  applications to other equations were explored in  \cite{liu2022oscillation,tao2023oscillation,du2023oscillation}.   
However, for discontinuous problems, the damping terms render the ODE system \eqref{eq:OFDG} highly stiff and lead to very restricted time step sizes for explicit schemes, as mentioned in \cite{LiuLuShu_OFDG_system}. 
To mitigate this challenge, 
the (modified) exponential RK methods \cite{huang2018bound} are typically required \cite{LiuLuShu_OFDG_system}.

This paper proposes and analyzes a new efficient DG approach, termed 
oscillation-eliminating DG (OEDG) method, designed on general meshes. 
While our focus is on hyperbolic conservation laws, the OEDG method is applicable to  general time-dependent partial differential equations (PDEs) with discontinuous solutions. 
The OEDG approach is inspired by the OFDG method \cite{lu2021oscillation,LiuLuShu_OFDG_system}, 
incorporating a similar oscillation control strategy but from a different novel perspective, which results in
significant enhancements in implementation and robustness. 
The new innovations, contributions,  and findings in this work are detailed in the following subsections.  

\subsection{Algorithmic innovations and contributions}  
The OEDG method is designed by alternately evolving the conventional semidiscrete DG scheme for \eqref{HCL} and a quasi-linear damping ODE: 
\begin{equation}\label{eq:OEDG1}
	\frac{{\rm d} {\bf U} }{{\rm d}t} = - {\bm L}_{\bf f}({\bf U}) \quad \mbox{and} \quad \frac{{\rm d} {\bf U}_\sigma }{{\rm d}t} = - \widetilde {\bf \Sigma}({\bf U}){\bf U}_\sigma. 
\end{equation}
The damping equation is devised for eliminating spurious oscillations. Here  
$\widetilde{\bf \Sigma}({\bf U})$ is a new damping operator, which differs from ${\bf \Sigma}({\bf U})$, and is carefully constructed to achieve the scale-invariant and evolution-invariant properties. 
With orthogonal DG basis functions, $\widetilde{\bf \Sigma}({\bf U})$ simplifies to a diagonal matrix.   
Notably, once $\widetilde{\bf \Sigma}({\bf U})$ is fixed, the quasi-linear damping ODE in \eqref{eq:OEDG1} can be {\em solved exactly} using a simple exponential operator. 
This splitting strategy is integrated into the RK stages or multi-steps. For example, with a second-order RK method, the fully discrete OEDG scheme is represented as  
\begin{equation}\label{eq:RK2-OEDG} 
	\begin{aligned}
		{\bf U}^{*} &= {\bf U}_\sig^n - \tau {\bm L}_{\bf f}({\bf U}_\sig^n),&\quad {\bf U}_\sig^{*} &= e^{-\tau \widetilde{\bf \Sigma}({\bf U}^{*})}{\bf U}^{*},\\
		{\bf U}^{n+1} &= \frac12 ( {\bf U}_\sig^n + {\bf U}_\sig^{*}) -\frac{ \tau }2 {\bm L}_{\bf f}({\bf U}_\sig^{*}),&\quad {\bf U}_\sig^{n+1} &= e^{-\tau \widetilde{\bf \Sigma}({\bf U}^{n+1})}{\bf U}^{n+1}. 
	\end{aligned}
\end{equation} 
Note that although the OEDG method employs a seemingly ``first-order'' splitting \eqref{eq:OEDG1} in the time evolution, this does {\em not} affect the order of accuracy, since the damping term $- \widetilde {\bf \Sigma}({\bf U}){\bf U}_\sigma$ is a high-order term for smooth solutions. Indeed, the optimal convergence rates of the OEDG method can be rigorously proven in theory; see \Cref{sec:err}.

The OEDG method offers many notable advantages. First, it 
effectively eliminates the spurious oscillations {\em for problems across different scales and wave speeds, without relying on problem-specific parameters}. It also 
retains many desirable properties of the original DG method, including {\em conservation}, {\em high-order accuracy}, {\em optimal fully-discrete error estimates} in theory, and numerical {\em superconvergence} for linear scalar conservation laws.  
Moreover, thanks to the exact solver for the damping equation in \eqref{eq:OEDG1} with the new coefficients $\widetilde{\bf \Sigma}({\bf U})$, the OEDG method exhibits some distinctive features not found in the OFDG method:   
	\begin{itemize}[leftmargin=*]
		\item 	{\em Non-intrusive}: The OE procedure is fully detached from the RK stage update and does not interfere with the DG spatial discretization. Hence it can be incorporated non-intrusively into existing DG codes as an independent module. This makes the OEDG approach applicable to general time-dependent PDEs with discontinuous solutions.  Note that such a non-intrusive feature is not possessed by the (modified) exponential RK method employed in the OFDG method \cite{LiuLuShu_OFDG_system}. 
		\item {\em Simple and efficient}: 
		Implementing the OE procedure is straightforward and efficient on general meshes, as it only involves the multiplication of modal coefficients by scalars. 
		\item {\em Stable with normal time step sizes}: 
		Thanks to the exact solver for the OE step, the OEDG method remains stable under the normal CFL condition, even in the presence of strong shocks associated with highly stiff damping terms. Unlike the OFDG method, 
		there is no need to use the (modified) exponential RK methods to avoid stringent time step restrictions. 
	\end{itemize}
	\begin{itemize}[leftmargin=*]
	\item {\em Scale-invariant and evolution-invariant}: 
	The OEDG method adopts the new damping operator $\widetilde{\bf \Sigma}({\bf U})$, which is carefully designed such that the damping effect remains scale-invariant and evolution-invariant. These key properties are crucial for guaranteeing the oscillation-free property, enabling the OEDG method to perform consistently well for problems across various scales and wave speeds.  
	It is observed that the damping operator ${\bf \Sigma}({\bf U})$ employed in the OFDG method, as described in \cite[Equation (2.12)]{lu2021oscillation}, lacks information about wave speed and scales linearly with respect to the solution magnitude: ${\bf \Sigma}(\lambda {\bf U}) = |\lambda| {\bf \Sigma}({\bf U})$. While this damping operator performs well for many common problems, it may exhibit 
	excessive smearing or persistent spurious oscillations for ultra-fast/ultra-slow  flows or solutions with large/small magnitudes; see  \Cref{fig:ex1_B}. 
	\item {\em Free of characteristic decomposition}: 
	The OEDG method does not require characteristic decomposition for hyperbolic systems, as the OE procedure applied directly to the conservative variables typically deliver
		satisfactory results. This feature simplifies the implementation and reduces the computational costs.
		It also distinguishes the OEDG method from the OFDG method with non-scale-invariant damping, in which characteristic decomposition 
		can be necessary for some challenging problems (as seen in \cite[Figure 3.5]{LiuLuShu_OFDG_system} and \Cref{tb} in the present paper).   
		Note that, without scale-invariant damping coefficients, it can be difficult to select the appropriate scale of eigenvectors in the characteristic
		decomposition, because eigenvectors are nonunique and can be scaled by a nonzero
		scalar. Different scaled eigenvectors will result in different characteristic variables, which lead to different damping strengths. Consequently, if the scheme lacks scale invariance, its 
		numerical performance heavily relies on the choice of eigenvectors, as improper selections can result in excessive smearing or persistent spurious oscillations. 
\end{itemize}
All these features highlight the differences between the OEDG and OFDG methods. 
	The last two features can be attributed to our novel damping operator. 
It is worth noting that this new operator, if integrated into the OFDG method in place of its original damping operator, may also offer these two features.

\subsection{New insights and findings}

By decoupling the OE step from the RK stage update, the OEDG method provides  
new insights into understanding the damping mechanism for oscillation control.     
Based on the exact solver for the damping equation in \eqref{eq:OEDG1}, we discover that the damping operator acts as a modal filter, revealing close relations between the oscillation-free damping technique and the spectral viscosity techniques  \cite{gottlieb2001spectral,hesthaven2007spectral,hesthaven2008filtering}. This perspective may not be obviously observed in the OFDG method  \cite{lu2021oscillation,LiuLuShu_OFDG_system}.  
In addition, the exact solver for the damping equation leads to insights into designing the scale-invariant and evolution-invariant damping operator. 
As shown in \eqref{eq:RK2-OEDG}, the coefficients $\tau \widetilde{\bf \Sigma}({\bf U})$  
as arguments of an exponential must be dimensionless. The scale-invariant property 
$\widetilde{\bf \Sigma}(\lambda {\bf U} + \mu) = \widetilde{\bf \Sigma}({\bf U})$ for any nonzero $\lambda$ 
and the inclusion of wave speed in the damping coefficients are crucial to ensure that $\tau \widetilde{\bf \Sigma}({\bf U})$ is dimensionless.

\subsection{Theoretical contributions and innovations}

We carry out a comprehensive fully-discrete error analysis of the proposed OEDG methods of arbitrarily high order accuracy in space and time. We rigorously establish the 
optimal error estimates of the fully-discrete OEDG methods for the linear advection equation with smooth solutions on Cartesian meshes. 
It is worth noting that, despite the linearity of the model equation considered in the analysis, the OE procedure transforms the DG scheme into a nonlinear one. Consequently, the error analysis is quite nontrivial and presents technical challenges akin to those encountered in fully discrete error estimates for nonlinear hyperbolic equations. 
A key innovation in our analysis lies in ingeniously transforming the nonlinear OEDG formulation into a linear RKDG method augmented by a nonlinear source term. 
This transformation allows us to leverage the techniques developed in  \cite{xu2020error,xu2020superconvergence,ai20222} for the original RKDG methods. 
One of the major difficulties in our analysis pertains to estimating the perturbation of the numerical solution during the OE step based on appropriate a priori assumptions, which is then integrated into the fully-discrete error analysis of the OEDG method. These estimates introduce  essential complexities not encountered in the semidiscrete error analysis of the OFDG method \cite{lu2021oscillation}. 
To the best of our knowledge, the presented optimal error analysis might be the first generic fully-discrete error estimates of nonlinear DG schemes with an automatic oscillation control mechanism for conservation laws.

\begin{figure}[!htb]
	\centering
	\begin{subfigure}[b]{0.48\textwidth}
		\begin{center}
			\includegraphics[width=1.0\linewidth]{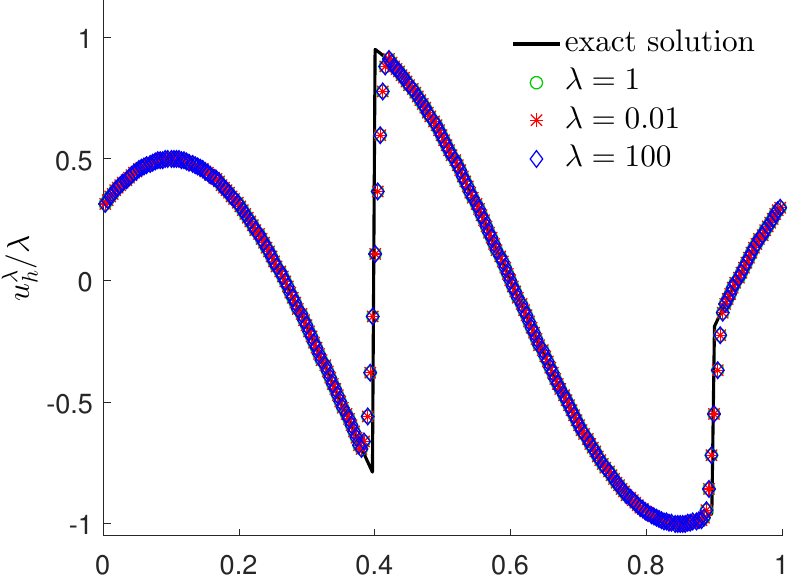}
			\caption{OEDG solution $u_h^\lambda/\lambda$}
			\label{fig1:a}
			\vspace{3.6mm}
		\end{center}
	\end{subfigure}
	\begin{subfigure}[b]{0.48\textwidth}
		\centering
		\includegraphics[width=1.0\linewidth]{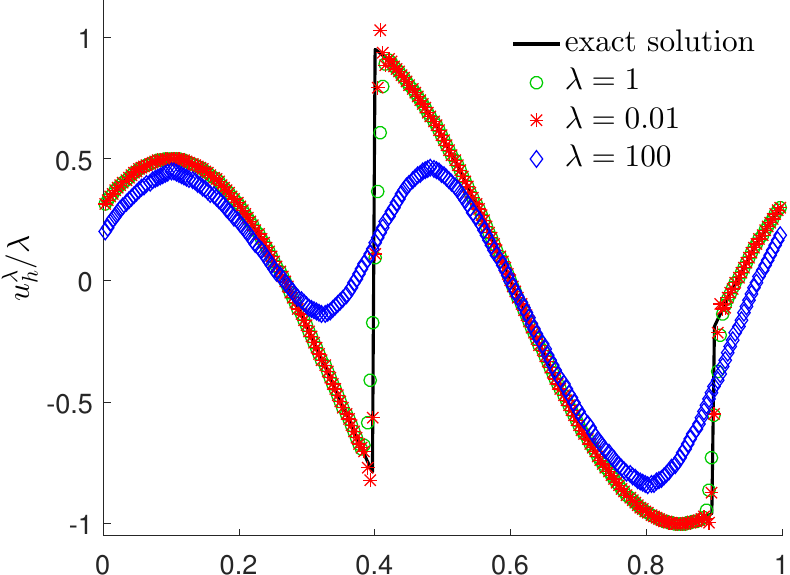}
		\caption{OFDG solution $u_h^\lambda/\lambda$}
		\label{fig1:b}
		\vspace{3.6mm}
	\end{subfigure}
	\centering
	\begin{subfigure}[b]{0.48\textwidth}
		\begin{center}
			\includegraphics[width=1.0\linewidth]{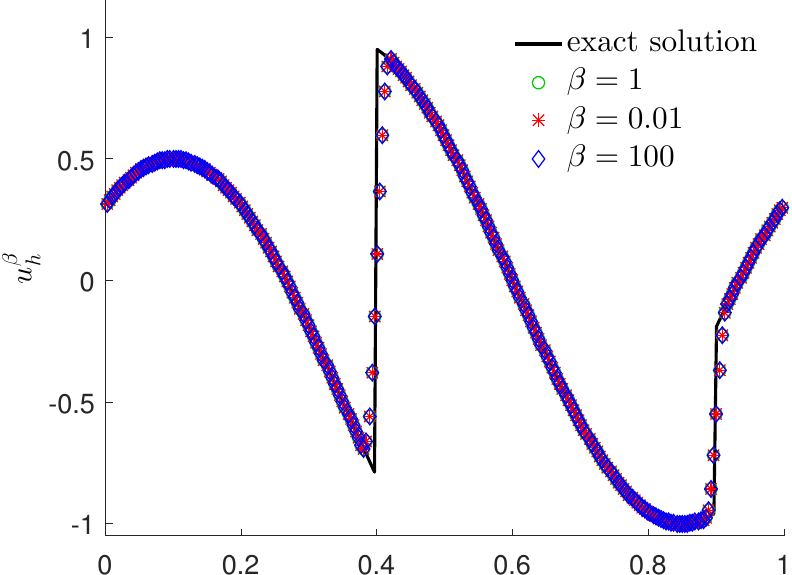}
			\caption{OEDG solution $u_h^\beta$}
			\label{fig1:c}
		\end{center}
	\end{subfigure}
	\begin{subfigure}[b]{0.48\textwidth}
		\centering
		\includegraphics[width=1.0\linewidth]{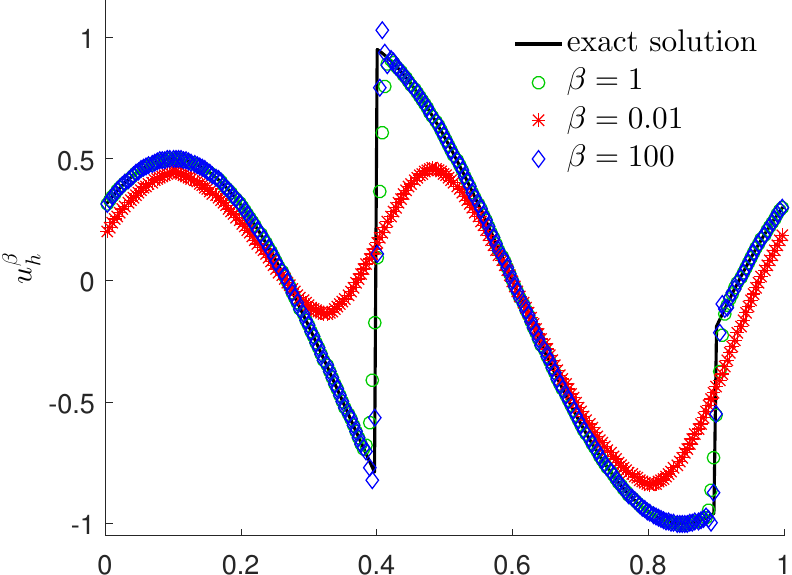}
		\caption{OFDG solution $u_h^\beta$}
		\label{fig1:d}
	\end{subfigure}
	\caption{Comparisons of third-order OEDG and OFDG methods for problems of different scales and wave speeds. The cell averages of the DG solutions are plotted:  $u_h^\lambda$ denotes numerical solutions at $t=1.1$ for $u_t + u_x =0$ with scaled initial data $u(x,0)=\lambda u_0(x)$;  $u_h^\beta$ denotes numerical solutions at $t=1.1/\beta$ for $u_t + \beta u_x =0$ with $u(x,0)= u_0(x)$, where $u_0(x)$ is defined in \eqref{eq:u0_exp2}. See \Cref{ex2} for the detailed setup. 
		The OFDG method exhibits  
		excessive smearing or persistent spurious oscillations (overshoots or undershoots) in large or small scale cases, while OEDG method performs consistently well thanks to scale invariance and evolution invariance.  The results of OFDG method can be improved if our new scale-invariant and evolution-invariant damping is used instead. 
	}
	\label{fig:ex1_B}
\end{figure}

This paper is structured as follows: \Cref{sec:1d-oedg} introduces the OEDG method for one-dimensional hyperbolic conservation laws, with extensions to multiple dimensions discussed in \Cref{sec:md}. Optimal error estimates of the fully discrete OEDG method for linear advection equation are established in \Cref{sec:err}.   \Cref{sec:num} provides extensive numerical examples to verify the effectiveness and advantages of the OEDG method, before conclusions in \Cref{sec:conclusion}.


\section{OEDG method for one-dimensional conservation laws} \label{sec:1d-oedg}
This section presents the OEDG scheme for 1D hyperbolic conservation laws
\begin{equation}\label{eq:1DHCL}
	u_t + f(u)_x = 0. 
\end{equation}
We first assume \eqref{eq:1DHCL} is a scalar conservation law, and will discuss the extension to hyperbolic systems later in \Cref{sec:1Dsystem}. 

\subsection{Formulation of OEDG method} 
Let $\Omega = \cup_j I_j$ with $I_j = [x_{j-\hf},x_{j+\hf}]$ be a partition of the 1D spatial domain $\Omega$. Denote $h_j = x_{j+\hf}-x_{j-\hf}$.  
The DG finite element space is defined as follows:
\begin{equation*}
	{\mathbb V}^k := \left \{ v \in L^2(\Omega):~ v |_{I_j} \in \mathbb P^k(I_j)~\forall j \right \},
\end{equation*}
where $\mathbb P^k(I_j)$ is the space of polynomials of degree less than or equal to $k$ on $I_j$. 

Let us first review the conventional DG method. 
The conventional semidiscrete DG method seeks the numerical solutions $u_h(\cdot,t)  \in {\mathbb V}^k$ such that for any test functions $v \in {\mathbb V}^k$, 
\begin{equation}\label{eq:1DDG}
	\int_{I_j} (u_h)_t v {\rm d}x = \int_{I_j} f(u_h) v_x {\rm d}x - 
	\hat{f}_{j+\hf} v_{j+\hf}^- + \hat{f}_{j-\hf} v_{j-\hf}^+,  
\end{equation}
where $v_{j+\hf}^\pm :=\lim_{\epsilon \to 0^+} v( x_{j+\hf} \pm \epsilon )$, and $\hat{f}_{j+\hf}$ is a suitable numerical flux, for example, a monotone flux in the scalar case. The semidiscrete scheme \eqref{eq:1DDG} can be rewritten in an ODE form $\frac{\rm d}{{\rm d}t}u_h = L_f(u_h)$, which can be further discretized in time with some high-order accurate RK or multi-step methods.  
For example, the 
DG method coupled with an $r$th-order $s$-stage RK method reads   
\begin{equation*}
	\begin{aligned}
		u_h^{n,0} &= u_h^{n},  \\
		u_h^{n,\ell+1} &= \sum_{0\leq \mykappa \leq \ell} \left(c_{\ell \mykappa}u_h^{n,\mykappa} + \tau d_{\ell \mykappa} L_f(u_h^{n,\mykappa}) \right),\quad \ell =0,1,\dots,s-1, \\
		u_h^{n+1} &= u_h^{n,s},
	\end{aligned} 
\end{equation*}
where $u_h^{n}$ is the numerical solution at $n$th time level, $\tau$ is the time step-size, and $\sum_{0\leq \mykappa\leq \ell} c_{\ell\mykappa} = 1$.

Define $u_\sig^{0}=u_h^{0}$. 
Now we present the OEDG method with an OE procedure applied after each RK stage: 
\begin{align}
	u_h^{n,\ell+1} &= \sum_{0\leq \mykappa \leq \ell} \left(c_{\ell \mykappa}u_\sigma^{n,\mykappa} + \tau d_{\ell \mykappa} L_f(u_\sigma^{n,\mykappa}) \right),\label{eq:rk}\\
	u_\sig^{n,\ell+1} &= \cF_\tau u_h^{n,\ell+1},\qquad \ell =0,1,\dots,s-1. \label{eq:OEstep}
\end{align}
For ease of notation, we define $u_\sig^{n,0} = u_\sig^n$ and $u_\sig^{n+1} = u_\sig^{n,s}$. 
Here \eqref{eq:OEstep} is the OE procedure, with 
$\cF_\tau : {\mathbb V}^k \to {\mathbb V}^k$ being the solution operator of a damping equation. 
More specifically, we define $(\cF_\tau u_h^{n,\ell+1})(x) = u_\sig(x,\tau)$ with $u_\sig(x,\hat t) \in {\mathbb V}^k$ being the solution to the following initial value problem: 
\begin{equation}\label{eq:filter}
	\left\{  
	\begin{aligned}
		&{\frac{\mathrm{d}}{\mathrm{d} \hat t} \int_{I_j}{u_\sig}{v}{\rm d} x} + \sum_{m = 0}^k 
		\beta_j \frac{\sig_j^{m} (u_h^{n,\ell+1}) }{h_j}{\int_{I_j}(u_\sig - P^{m-1}u_\sig)v{\rm d}x} =\; 0~~\forall v \in \mathbb P^k(I_j),\\
		&u_\sig(x,0) =\; u_h^{n,\ell+1}(x),
	\end{aligned} \right.
\end{equation}
where $\hat t$ is a pseudo-time different from $t$, and $\beta_j$ is a suitable estimate of the local maximum wave speed on $I_j$; for scalar conservation laws, we take 
$
	\beta_j = | f'( \overline{u}_j^{n,\ell+1} ) |,
	$ 
	where $\overline{u}_j^{n,\ell+1}$ denotes the cell average of $u_h^{n,\ell+1}(x)$ over $I_j$. 
The operator $P^m$ is the standard $L^2$ projection into ${\mathbb V}^m$ for $m\ge 0$, that is, for any function $w$, $P^m w \in {\mathbb V}^m$ satisfies 
\begin{equation*}
	\int_{I_j} (P^m w - w) v {\rm d}x =0 \qquad \forall v \in \mathbb P^m(I_j),~\forall j. 
\end{equation*}
We define $P^{-1}=P^0$. It is worth noting that $P^0 w (x) = \overline w_j$ for all $x \in I_j$, where $\overline w_j$ denotes the cell average of $w$ over $I_j$. 
The damping coefficient $\sig_j^{m}(u_h)$, as a function of $u_h$, is crucial and should be chosen carefully. 
It should be small in smooth regions and large near discontinuities. To make the resulting scheme scale-invariant, we propose a new damping coefficient
\begin{equation}\label{eq:1Dsigma}
	\sig_j^{m}(u_h) = 
	\begin{cases} 
		0, ~~ & {\rm if}~  u_h \equiv \mathrm{avg}(u_h), \\
			\displaystyle
	 \frac{ (2m+1)h_j^m}{(2k-1)m!} \frac{\big|\jump{\partial_x^m u_h}_{j-\frac{1}{2}}\big|+ \big|\jump{\partial_x^m u_h}_{j+\frac{1}{2}}\big|  } {2 \|  u_h - \mathrm{avg}(u_h)  \|_{L^\infty(\Omega)}  }, ~~ & {\rm otherwise},
	 \end{cases}
\end{equation}
where $\jump{v}_{j+\frac{1}{2}}=v_{j+\frac12}^+-v_{j+\frac12}^-$ is the jump of $v$ at $x=x_{j+\frac12}$, and 
$\mathrm{avg}(v)=\frac1{|\Omega|} \int_{\Omega} v(x) {\rm d} x$ denotes the average of $v$ over $\Omega$. 
Note that the common scalar quantity $\|  u_h - \mathrm{avg}(u_h)  \|_{L^\infty(\Omega)}$ is used for all mesh cells. At the beginning of the OE procedure in each stage, this quantity can be first calculated using global data, similarly to how we compute $\tau$ at every time step. 
After that, the computation of $\sig_j^{m}(u_h)$ will involve information only from $I_j$ and its adjacent cells, which maintains the compactness and local structure of the conventional DG methods. 
As it will be elaborated in \Cref{sec:invariance}, the OEDG method is 
scale-invariant and evolution-invariant, thanks to the carefully designed new damping coefficient \eqref{eq:1Dsigma} and the incorporation of wave speed $\beta_j$ into the damping terms in \eqref{eq:filter}. 


\subsection{Exact solver of OE procedure}

Note that $\sig_j^{m} (u_h^{n,\ell+1})$ in \eqref{eq:filter} only depends on the ``initial'' solution $u_\sig(x,0) = u_h^{n,\ell+1}(x)$. As a result, the damping ODE \eqref{eq:filter} is linear, and its exact solution can be {\em explicitly} formulated. Therefore, the OE procedure is computationally cheap and very easy to implement. 

Let $\{\phi_{j}^{(i)}(x)\}_{i=1}^k$ be a local orthogonal basis of $\mathbb P^k(I_j)$ over $I_j$. For example, we choose the scaled Legendre polynomials 
$$
\phi_{j}^{(0)}(x)=1,~~\phi_{j}^{(1)}(x)=\frac{x-x_j}{h_j},~~\phi_{j}^{(2)}(x)=12\left(\frac{x-x_j}{h_j}\right)^2-1,\quad \dots.
$$
Assume that the solution of \eqref{eq:filter} can be represented as 
\begin{equation}\label{eq:expansion}
	u_\sig (x,\hat t) = \sum_{i=0}^k  u_j^{(i)}(\hat t) \phi_j^{(i)}(x) \quad \mbox{for}~x\in I_j. 
\end{equation}
Note that
\begin{equation}\label{eq:expansion2}
	(u_\sig -P^{m-1} u_\sig) (x,\hat t) = \sum_{i=\max\{m,1\}}^k  u_j^{(i)}(\hat t) \phi_j^{(i)}(x).
\end{equation}
Substitute \eqref{eq:expansion}--\eqref{eq:expansion2} into \eqref{eq:filter} and take $v=\phi_{j}^{(i)}$.
For $i\geq 1$, one obtains 
$$
{\frac{\mathrm{d}}{\mathrm{d} \hat t}} u_j^{(i)} (\hat t) \int_{I_j} \left( \phi_j^{(i)}(x) \right)^2 {\rm d}x + \sum_{m=0}^i \beta_j \frac{ \sig_j^m}{h_j} u_j^{(i)} (\hat t) \int_{I_j} \left( \phi_j^{(i)}(x) \right)^2 {\rm d}x = 0,  \quad i = 1,2,\cdots, k,
$$
which can be simplified as 
\begin{equation}\label{eq:dampingODE}
	{\frac{\mathrm{d}}{\mathrm{d} \hat t}} u_j^{(i)} (\hat t) + \left(  \frac{\beta_j}{h_j} \sum_{m=0}^i  \sig_j^m  \right) u_j^{(i)} (\hat t) = 0, \quad i = 1, 2, \cdots, k. 
\end{equation}
Solving \eqref{eq:dampingODE} gives 
\begin{equation}\label{eq:dampingODEi}
	u_j^{(i)} (\tau) = \exp \left( -\frac{\beta_j \tau}{h_j} \sum_{m=0}^i \sig_j^m \right) u_j^{(i)} (0), \quad i = 1, 2, \cdots k.
\end{equation}
For $i = 0$, since $\int_{I_j} (u_\sig -P^{m-1} u_\sig)\phi_j^{(0)} {\rm d}x \equiv 0$, we have ${\frac{\mathrm{d}}{\mathrm{d} \hat t}} u_j^{(0)}(\hat{t})\int_{I_j}(\phi_j^{(0)}(x))^2 {\rm d}x = 0$ and thus
\begin{equation}\label{eq:dampingODE0}
	u_j^{(0)}(\tau) = u_j^{(0)}(0).
\end{equation} 
Hence, with \eqref{eq:dampingODE0} and \eqref{eq:dampingODEi}, one has 
\begin{equation}\label{usig}
	u_\sig^{n,\ell+1} =  \cF_\tau u_h^{n,\ell+1} = u_\sig (x,\tau) = u_j^{(0)} (0) \phi_j^{(0)}(x) + \sum_{i=1}^k {\rm e}^{-\frac{\beta_j \tau}{h_j} \sum_{m=0}^i \sig_j^m( u_h^{n,\ell+1} )  } u_j^{(i)} (0) \phi_j^{(i)}(x),
\end{equation}
where $
u_j^{(i)} (0) = \int_{I_j}u_h^{n,\ell+1} \phi_j^{(i)}{\rm d}x / \| \phi_j^{(i)} \|_{L^2(I_j)}^2
$ is the modal coefficient of $u_h^{n,\ell+1}(x)$.

\begin{remark}[dimensionless damping]
	As well-known, the argument of an exponential must be dimensionless, since it is raised to all powers in the corresponding series expansion and maintaining consistent dimensions across all terms in the series is imperative. For the exponential term in \eqref{usig}, 
	the scale-invariant property in \Cref{thm:scale-invariant} ensures that $\sig_j^m( u_h^{n,\ell+1} )$ is  inherently dimensionless, while the inclusion of the wave speed $\beta_j$ in the damping terms guarantees that $\frac{\beta_j \tau}{h_j}$ is also dimensionless. 
\end{remark}

\begin{remark}[conservation]
	From \eqref{eq:dampingODE0}, one can see that the OE procedure does not affect the modal coefficient associated with $\phi_j^{(0)}$. In other words, the OE procedure does not alter the cell average of the numerical solution, thereby preserving the local mass conservation. 
\end{remark}

\begin{remark}[stability]
	From \eqref{usig}, one can observe that the damping operator actually acts as a modal filter. This reveals the close relations between the damping technique and the spectral viscosity techniques  \cite{gottlieb2001spectral,hesthaven2007spectral,hesthaven2008filtering}. 
	The OE procedure damps the modal coefficients of the high-order moments of the numerical solution. Hence we have 
	\[
	\nm{u_\sig^{n,\ell+1}}_{L^2(I_j)}^2 = \sum_{i=0}^k|u_j^{(i)}(\tau)|\|\phi_j^{(i)}\|_{L^2(I_j)}^2\leq \sum_{i=0}^k|u_j^{(i)}(0)|\|\phi_j^{(i)}\|_{L^2(I_j)}^2 = \|u_h^{n,\ell+1}\|_{L^2(I_j)}^2. \]
	As a result, the OE procedure will reduce the element-wise $L^2$ norm of the numerical solution, thereby enhancing the stability of the numerical solution. 
	Thanks to the exact solver for the OE step, the OEDG method remains stable under the normal CFL condition, even in the presence of strong shocks associated with highly stiff damping terms. Unlike the OFDG approach \cite{LiuLuShu_OFDG_system}, 
	the OEDG method does not require the (modified) exponential RK methods to avoid stringent time step restrictions. 
\end{remark}

\begin{remark}[simplicity and efficiency] 
	The OE procedure is fully detached from the RK stage update and does not interfere with the DG spatial discretization. 
	In contrast to the exponential RK method in the OFDG approach, the OE procedure in the OEDG method is  non-intrusive: It can be incorporated easily and seamlessly into existing DG codes as an independent module, requiring only a slight modification to the existing structure, rather than an extensive overhaul.  
	Moreover, implementing the OE procedure is straightforward and efficient, as it involves only simple multiplication of modal coefficients by scalars. 
\end{remark}

\begin{remark}
	In the OE procedure \eqref{eq:OEstep}, it seems unnecessary to synchronize the time stepsize with the specific RK stage. We opt to evolve the damping equation over a full time step for ease, which also promotes uniformity in the implementation of the OE procedure. Certainly, one can choose to align the time stepsize with the specific RK stage, which may help reduce the damping strength. 
\end{remark}

\subsection{Scale invariance and evolution invariance}\label{sec:invariance} 
In this subsection, we elaborate the scale-invariant and evolution-invariant properties of the OEDG method, which are attained through the carefully designed new damping operator. 

\begin{theorem}[scale invariance]\label{thm:scale-invariant}
	For any $\lambda \neq 0$ and any $\mu \in \mathbb R$, ones has  
	\begin{equation*}
		\sig_j^{m}(\lambda u_h+\mu)=\sig_j^{m}(u_h), 
	\end{equation*}
	which implies 
	\begin{equation*}
		\cF_\tau(\lambda u_h+\mu)=\lambda\cF_\tau(u_h)+\mu.
	\end{equation*}
\end{theorem}

Such a scale-invariant property is crucial, as it ensures that the damping coefficient \eqref{eq:1Dsigma} is ``dimensionless'' and that the damping strength/effect remains unchanged regardless of the unit used for $u$. For instance, if $u$ represents mass, then the scale-invariant damping effect  is consistent, whether gram or kilogram is used for $u$.  
Additionally, we observe that scale invariance is vital for effectively guaranteeing the oscillation-free property, enabling the OEDG method to perform consistently well for problems across different scales; see \Cref{fig:ex1_B} and more numerical evidences in \Cref{sec:num}.

Let ${\mathcal S}_t$ be the exact solution operator of the equation \eqref{eq:1DHCL}, i.e., ${\mathcal S}_t(u(x,0))=u(x,t)$. Let ${\mathcal E}_{n}$ be the solution operator of the OEDG method, i.e., ${\mathcal E}_{n} (u_h^0) = u_\sigma^n$. 
When the flux in \eqref{eq:1DHCL} is homogeneous, namely, $f(\lambda u) = \lambda f(u)$ for all $\lambda$, then the exact solution operator ${\mathcal S}_t$ is also homogeneous: ${\mathcal S}_t ( \lambda u(x,0) ) = \lambda {\mathcal S}_t ( u(x,0) )$. Thanks to the scale-invariant property, the OEDG solution operator preserves such homogeneity.

\begin{theorem}[homogeneity]\label{thm:homogeneity}
	If the flux function $f(u)$ and the numerical flux $\hat{f}_{j+\hf}$ are both homogeneous, then the OEDG solution operator is homogeneous: 
	\begin{equation*}
		{\mathcal E}_{n} (\lambda u_h^0) = \lambda  {\mathcal E}_{n} (u_h^0) \qquad \forall \lambda \in \mathbb R. 
	\end{equation*}
\end{theorem} 

\begin{proof}
	If $f(u)$ and $\hat{f}_{j+\hf}$ are homogeneous, then $L_f(u_\sigma^{n,\mykappa})$ is homogeneous, implying that the local solution operator for each RK stage \eqref{eq:rk} is homogeneous. According to \Cref{thm:scale-invariant}, the operator $\cF_\tau(u_h)$ for each OE procedure \eqref{eq:OEstep} is  homogeneous. Therefore, ${\mathcal E}_{n}$ is homogeneous. 
\end{proof}

\begin{remark}\label{rem:LuLiuShuDamping}
	The following damping coefficient was proposed in \cite{lu2021oscillation}: 
\begin{equation}\label{eq:LuLiuShuDamping}
	\widehat{\sig}_j^{m}(u_h) = \frac{ 2(2m+1)h^m}{(2k-1)m!}  \left( \jump{\partial_x^m u_h}_{j-\frac{1}{2}}^2 + {\jump{\partial_x^m u_h}_{j+\frac{1}{2}}^2} \right)^{\frac12} ~~ \mbox{with}~~h:=\max_{j} h_j, 
\end{equation}
	which is not scale-invariant due to $\widehat{\sig}_j^{m}( \lambda u_h) = |\lambda| \widehat{\sig}_j^{m}(u_h)$. This indicates that if we scale $u$ by changing its unit, then the damping effect would be reduced ($|\lambda|<1$) or increased ($|\lambda|>1$). 
	As a result, oscillations may not be fully suppressed for small-scale problems, and much smearing or dissipation is produced for large-scale problems; see \Cref{fig1:b}, \Cref{ex2}, and \Cref{ex:tf}. 
\end{remark}

	\begin{remark}[local scale invariance]\label{rem:LSI}
	When simulating some multi-scale problems that simultaneously involve both small and large magnitude structures, one may consider a locally scale-invariant damping coefficient, for example, 
	\begin{equation}\label{eq:1Dsigma_localSI}
		\widetilde \sig_j^{m}(u_h) = 
		\begin{cases} 
			0, ~~ & {\rm if}~  u_h | _{\Omega_j} \equiv {\rm constant}, \\
			\displaystyle
			C_{m,k} \frac{ (2m+1)h_j^m \tilde h_j }{(2k-1)m!} \frac{\big|\jump{\partial_x^m u_h}_{j-\frac{1}{2}}\big|+ \big|\jump{\partial_x^m u_h}_{j+\frac{1}{2}}\big|  } {2 \left \|  u_h - \frac{1}{|\Omega_j|} \int_{\Omega_j}u_h {\rm d} x \right \|_{L^\infty(\Omega_j)}  }, ~~ & {\rm otherwise},
		\end{cases}
	\end{equation}
	where $\Omega_j$ is a local region containing $I_j$ and can be defined as $\Omega_j=I_{j-1} \cup I_j \cup I_{j+1}$, and $\tilde h_j := \max\{ h_{j-1}, h_j, h_{j+1} \}$ is introduced to ensure accuracy.  
	Yet, our numerical experiments indicate that it is difficult to determine a unified mesh-independent constant $C_{m,k}$. This will be explored in our future work. The importance of local scale invariance was also recognized in the design of finite volume and finite difference schemes  \cite{chen2022physical,don2022novel}. 
\end{remark}

Another distinctive feature in the OEDG method is that our damping terms in \eqref{eq:filter} incorporate the information of the local maximum wave speed  $\beta_j$, which was not considered in \cite{lu2021oscillation} but is critical for ensuring the appropriate damping strength for varying characteristic speeds (see \Cref{fig:ex1_B} and \Cref{ex2}). The insights of including $\beta_j$ arise from the evolution-invariant property, as explained in the following. 
For any given constant $\lambda >0$, let ${\mathcal S}_t^{\lambda}$ be the solution operator of the scalar conservation law 
$u_t + \lambda f(u)_x=0$. In particular, ${\mathcal S}_t$ denotes the solution operator of equation \eqref{eq:1DHCL}. The equation $u_t + \lambda f(u)_x=0$ can be regarded as the formulation of equation \eqref{eq:1DHCL} by changing the unit of $t$. Obviously, we have the following evolution-invariant property 
$$
{\mathcal S}_t^{\lambda} = {\mathcal S}_{\lambda t} \qquad \forall \lambda>0,~~ \forall t>0. 
$$
This property is also preserved by the OEDG method, as shown below. It states that for problems with slow-propagating waves (large time steps) and fast-propagating waves (small time steps), if their solutions are consistent, the OEDG method will yield identical results after a fixed number of steps.

  
\begin{theorem}[evolution invariance]\label{thm:evolution-invariant}
	Let ${\mathcal E}_{n}^{\lambda}$ be the solution operator of the OEDG method solving 
	$u_t + \lambda f(u)_x=0$ for $\lambda > 0$ on a fixed mesh with the same CFL number, namely, 
	${\mathcal E}_{n}^{\lambda}(u_h^{0}) = u_\sigma^{n}$ at time $t_n^\lambda=n \tau_\lambda$, where $\tau_\lambda = \tau/\lambda$ is the time stepsize. Then we have 
	\begin{equation}\label{eq:evolution-invariant}
		{\mathcal E}_{n}^{\lambda} = {\mathcal E}_{n}  \qquad \forall \lambda>0, 
	\end{equation}
where ${\mathcal E}_{n}$ is the solution operator of the OEDG method solving equation \eqref{eq:1DHCL}. 
\end{theorem}

\begin{proof}
	Note that the evolution-invariant property holds for the conventional DG scheme at each RK stage, because $L_{\lambda f}(u_h) = \lambda L_f(u_h)$ and  $\tau_\lambda L_{\lambda f}(u_h)  = \tau/\lambda\cdot \lambda L_{f}(u_h) = \tau L_f(u_h)$.  Thus the RK stage \eqref{eq:rk} will be identical regardless of the value of $\lambda$. 
		The OE step also preserves the evolution invariance. In fact, let $\cF_{\tau_\lambda}^\lambda$ be the OE operator for $u_t + \lambda f(u)_x=0$. Since $\beta_j = \lambda | f'( \overline{u}_j^{n,\ell+1} ) |$, we observe that 
		\begin{equation}\label{eq:Flambda}
			\cF_{\tau_\lambda}^\lambda = \cF_{\lambda \tau_\lambda} = \cF_{\tau} \qquad \forall \lambda > 0. 
		\end{equation}
		Consequently, both the RK stage and the OE step produce identical outputs regardless of the value of $\lambda$. Hence the OEDG method preserves the evolution invariance \eqref{eq:evolution-invariant}.
\end{proof}

\begin{remark}
	We would like to emphasize that the proposed damping terms offer a viable replacement for the original damping terms in the OFDG method \cite{lu2021oscillation,LiuLuShu_OFDG_system}. With this modification, the new OFDG method will achieve both scale invariance and evolution invariance. Consequently, it would demonstrate consistent performance across different scales and wave speeds. 
\end{remark}

\subsection{Extension to 1D hyperbolic systems}\label{sec:1Dsystem} 

The OE procedure can be naturally extended to 1D hyperbolic systems of conservation laws ${\bf u}_t + {\bf f}({\bf u})_x={\bf 0}$.  
More specifically, we propose 
\begin{equation}\label{eq:filterSYS}
	\left\{  
	\begin{aligned}
		&{\frac{\mathrm{d}}{\mathrm{d} \hat t}\int_{I_j}{\bf u}_\sig\cdot{\bf v}{\rm d}x} + \sum_{m = 0}^k 
		\beta_j \frac{\sig_j^{m} ({\bf u}_h^{n,\ell+1}) }{h_j}{\int_{I_j}({{\bf u}_\sig - P^{m-1}{\bf u}_\sig})\cdot{\bf v}{\rm d}x} =\; 0~~\forall {\bf v} \in \mathbb [\mathbb P^k(I_j)]^N,\\
		&{\bf u}_\sig(x,0) =\; {\bf u}_h^{n,\ell+1}(x),
	\end{aligned} \right.
\end{equation}
with $\beta_j$ being the spectral radius of the Jacobian matrix $\frac{\partial {\bf f}}{\partial {\bf u}} (\overline {\bf u}_j^{n,\ell+1})$. The damping coefficient $\sig_j^{m} ({\bf u}_h)$, as a function of ${\bf u}_h$, is defined as 
$$
\sig_j^{m} ({\bf u}_h) := \max_{1\le i\le N} \sig_j^{m} (u_h^{(i)}),
$$
where $u_h^{(i)}$ is the $i$th component of ${\bf u}_h$, and $\sig_j^{m} (u_h^{(i)})$ is computed by \eqref{eq:1Dsigma}. It can be verified that the scale-invariant and evolution-invariant properties are also satisfied by the OEDG method for  hyperbolic systems. 

\begin{remark}
	In contrast to the OFDG method \cite{LiuLuShu_OFDG_system}, characteristic decomposition appears unnecessary in the OEDG approach for hyperbolic systems. Although employing characteristic variables can aid in shock identification, determining suitable eigenvectors for the damping terms poses challenges, especially when the damping coefficients are not scale-invariant. It is crucial to note that eigenvectors are not unique; multiplying an eigenvector by a nonzero scalar yields another valid eigenvector. Consequently, distinct (scaled) eigenvectors correspond to different characteristic variables, resulting in significantly different damping strength. Hence the numerical performance may heavily hinge on the choice of eigenvectors: Improper choices could lead to excessive smearing or persistent spurious oscillations. Moreover, opting for characteristic decomposition may destroy the scale invariance, introduce complexity in implementation, and increase computational costs.
\end{remark}

\section{OEDG method for multidimensional conservation laws}\label{sec:md}

In this section, we present the OEDG method for multidimensional system of conservation laws:  
\begin{equation}\label{eq:2DHCL}
	{\bf u}_t + \nabla \cdot {\bf f}({\bf u}) = {\bf 0}, 
\end{equation}
where ${\bf u}=(u_1,\dots,u_N)$. 
Let ${\mathcal T}_h$ be a partition of the computational domain $\Omega$.  Our OEDG method seeks the approximate solution in the following finite element space of discontinuous functions: 
\begin{equation}\label{eq:2DDGspace}
	{\mathbb V}^k := \left \{ {\bf v} \in L^2(\Omega):~ {\bf v} |_{K} \in [\mathbb P^k(K)]^N~~\forall K \in {\mathcal T}_h \right \},
\end{equation}
where $\mathbb P^k(K)$ is the space of polynomials of total degree less than or equal to $k$ on the element $K$.

The conventional semidiscrete DG method reads  
\begin{equation}\label{eq:sDG}
	\int_K ({\bf u}_h)_t \cdot{\bf v} {\rm d} {\bm x} = \int_{K} {\bf f}({\bf u}_h) : \nabla {\bf v} {\rm d} {\bm x} - \sum_{e \in \partial K} \int_{e \in \partial K} \widehat {\bf f} ({\bf u}_h) : ({\bf n}_e \otimes {\bf v}) {\rm d} S \quad \forall {\bf v} \in {\mathbb V}^k,
\end{equation}
where ``$:$'' denotes the Frobenius inner product of two matrices, ``$\otimes$'' represents the Kronecker product of two vectors, 
$\widehat {\bf f} ({\bf u}_h)$ denotes a suitable numerical flux on the element interface $e \in \partial K$, and ${\bf n}_e$ is the outward unit normal vector at $e$ with respect to $K$. 
The semidiscrete scheme \eqref{eq:sDG} can be rewritten in an ODE form ${\frac{\mathrm{d}}{\mathrm{d} t}} {\bf u}_h = L_f({\bf u}_h)$, which can be further discretized in time by using some high-order accurate RK or multi-step methods. 

Similar to the 1D case, the OEDG method is based on applying an OE procedure after each stage in the RK methods or each step in the multi-step methods. 
For example, the OEDG coupled with an $r$th-order $s$-stage RK method reads   
\begin{align}
	{\bf u}_\sig^{n,0} &= {\bf u}_\sig^n, 
	\\
	{\bf u}_h^{n,\ell+1} &= \sum_{0\leq \mykappa \leq \ell} \left(c_{\ell \mykappa} {\bf u}_\sigma^{n,\mykappa} + \tau d_{\ell \mykappa} L_f( {\bf u}_\sigma^{n,\mykappa}) \right),\label{eq:rk2D}\\
	{\bf u}_\sig^{n,\ell+1} &= \cF_\tau {\bf u}_h^{n,\ell+1},\qquad \ell =0,1,\dots,s-1. \label{eq:OEstep2D}
	\\
	{\bf u}_\sig^{n+1} &= {\bf u}_\sig^{n,s}.
\end{align}
The OE operator $\cF_\tau$ is defined as 
 $(\cF_\tau {\bf u}_h^{n,\ell+1})(x) = {\bf u}_\sig(x,\tau)$ with ${\bf u}_\sig(x,\hat t) \in {\mathbb V}^k$ being the solution to the following initial value problem: 
\begin{equation}\label{eq:filter2D}
	\left\{  
	\begin{aligned}
		&{\frac{\mathrm{d}}{\mathrm{d} \hat t}  \int_K {\bf u}_\sig\cdot {\bf v} {\rm d}{\bm{x}}} + \sum_{m = 0}^k 
		\delta_K^{m} ({\bf u}_h^{n,\ell+1})  {\int_K( {{\bf u}_\sig - P^{m-1}{\bf u}_\sig})\cdot{\bf v}{\rm d}{\bm x}} =\; 0~~\forall {\bf v} \in [\mathbb P^k(K)]^N,\\
		&{\bf u}_\sig(x,0) =\; {\bf u}_h^{n,\ell+1}(x),
	\end{aligned} \right.
\end{equation}
where $\hat t$ is a pseudo-time different from $t$, $P^m$ is the standard $L^2$ projection operator into ${\mathbb V}^m$ for $m\ge 0$, and we define $P^{-1}=P^0$. 
The coefficient $\delta_K^{m} ({\bf u}_h)$ is defined as 
\begin{equation}\label{eq:delta}
	\delta_K^{m} ({\bf u}_h) = \sum_{e \in \partial K} \beta_e \frac{ \sigma_{e,K}^m({\bf u}_h) }{h_{e,K}}, 
\end{equation}
where  $\beta_e$ a suitable estimate of the local maximum wave speed in the direction ${\bf n}_e$ on the element interface $e$. In our computations, we take $\beta_e$ as the spectral radius of $\sum_{i=1}^d n_e^{(i)} \frac{\partial {\bf f}_i}{\partial {\bf u}} \big|_{{\bf u}=\overline{\bf u}_K}$, where $n_e^{(i)}$ is the $i$th  component of ${\bf n}_e$. 
We define  
\begin{equation*}\label{eq:2D_h_sig}
	h_{e,K} = \sup_{\bm x \in K} {\rm dist}({\bm x},e), \qquad \sigma_{e,K}^m({\bf u}_h) = \max_{1\le i\le N} \sigma_{e,K}^m(u_h^{(i)}) 
\end{equation*}
with $u_h^{(i)}$ being the $i$the component of ${\bf u}_h$, and 
\begin{equation}\label{eq:2Dsig}
	\sigma_{e,K}^m(u_h^{(i)}) = 
		\begin{cases} 
		0, ~~ & {\rm if}~  u_h^{(i)} \equiv \mathrm{avg}(u_h^{(i)}), \\
		\displaystyle
		\frac{ (2m+1)h_{e,K}^m}{2(2k-1)m!}  \sum_{|{\bm \alpha }|=m} \frac{ \frac{1}{|e|} \int_{e}  \big|\jump{\partial^{\bm  \alpha} u_h^{(i)}}_e \big| {\rm d}S  } { \|  u_h^{(i)} - \mathrm{avg}(u_h^{(i)})  \|_{L^\infty(\Omega)}  }, ~~ & {\rm otherwise}. 
	\end{cases}
\end{equation}
In \eqref{eq:2Dsig}, $|e|$ denotes the length/area of the edge/surface $e \in \partial K$, the vector ${\bm  \alpha}=( \alpha_1,\dots, \alpha_d)$ is the multi-index with $|{\bm  \alpha}|=\sum_{i=1}^d  \alpha_i$, and $\partial^{\bm  \alpha} u_h^{(i)}$ is defined as 
$$
\partial^{\bm  \alpha} u_h^{(i)}({\bm x}) = \frac{ \partial^{|{\bm  \alpha}|} }{\partial x_1^{ \alpha_1} \cdots \partial x_d^{ \alpha_d}} u_h^{(i)}({\bm x})
$$
with ${\bm x}=(x_1,\dots,x_d)$. 
In \eqref{eq:2Dsig}, $|\jump{\partial^{\bm  \alpha} u_h^{(i)}}_e|$ denotes the absolute value of the jump of $\partial^{\bm  \alpha} u_h^{(i)}$ across the element interface $e$, and its integration on $e$ should be approximated by some quadrature rules. In our computations for 2D problems, the trapezoidal rule is used. 
Similar to the 1D case, the multidimensional OE procedure can be easily implemented as 
\begin{equation}
	\cF_\tau {\bf u}_h^{n,\ell+1}  =  {\bf u}_K^{({\bm 0})} (0) \phi_K^{({\bm  0})}({\bm x}) + \sum_{j=1}^k {\rm e}^{-\tau \sum_{m=0}^j \delta_K^m( {\bf u}_h^{n,\ell+1} )  } \sum_{ |{\bm  \alpha}|=j } {\bf u}_K^{({\bm  \alpha})} (0) \phi_K^{({\bm  \alpha})}({\bm x}), 
\end{equation}
where 
${
	{\bf u}_K^{({\bm  \alpha})} (0) = \int_K {\bf u}_h^{n,\ell+1} \phi_K^{({\bm  \alpha})} {\rm d}\bm{x}} / \| \phi_K^{({\bm  \alpha})} \|_{L^2(K)}^2
$ is the expansion coefficient of ${\bf u}_h^{n,\ell+1}({\bm x})$, and $\{\phi_{K}^{(\bm  \alpha)}\}_{|\bm  \alpha|\le k}$ is a local orthogonal basis of $\mathbb P^k(K)$.

For example, consider a rectangular element $K=[x_{i-\hf},x_{i+\hf}]\times [y_{j-\hf},y_{j+\hf}]$. Define $h_{x,i}=x_{i+\hf}-x_{i-\hf}$ and  $h_{y,j}=y_{j+\hf}-y_{j-\hf}$. Let $\beta_{ij}^x$ and $\beta_{ij}^y$ be the estimates of the local maximum wave speeds in the $x$- and $y$-directions, respectively. 
Then the coefficient $\delta_K^{m} ({\bf u}_h)$ becomes 
\begin{equation}\label{eq:delta2Dcart}
	\delta_K^{m} ({\bf u}_h) = \max_{1\le q \le N}  \left(  \frac{ \beta_{ij}^x (\sigma_{i+\frac12,j}^m(u_h^{(q)}) + \sigma_{i-\frac12,j}^m(u_h^{(q)})) }{ h_{x,i}  } +  \frac{ \beta_{ij}^y (\sigma_{i,j+\frac12}^m(u_h^{(q)}) + \sigma_{i,j-\frac12}^m(u_h^{(q)})) }{ h_{y,j}  }   \right)
\end{equation}
with 
\begin{equation*}
	\sigma_{i+\frac12,j}^m(u_h^{(q)}) = 
	\begin{cases} 
		0, ~ & {\rm if}~  u_h^{(q)} \equiv \mathrm{avg}(u_h^{(q)}), \\
		\displaystyle
		\frac{ (2m+1)h_{x,i}^m}{2(2k-1)m!}  \sum_{|{\bm \alpha}|=m} \frac{ \frac{1}{|h_{y,ij}|} \int_{y_{j-\hf}}^{y_{j+\hf}}  \big|\jump{\partial^{\bm \alpha} u_h^{(q)}}_{i+\frac12,j} \big| {\rm d}y  } { \|  u_h^{(q)} - \mathrm{avg}(u_h^{(q)})  \|_{L^\infty(\Omega)}  }, ~ & {\rm otherwise},
	\end{cases}
\end{equation*}
\begin{equation*}
	\sigma_{i,j+\frac12}^m(u_h^{(q)}) = 
	\begin{cases} 
		0, ~ & {\rm if}~  u_h^{(q)} \equiv \mathrm{avg}(u_h^{(q)}), \\
		\displaystyle
		\frac{ (2m+1)h_{y,j}^m}{2(2k-1)m!}  \sum_{|{\bm \alpha}|=m} \frac{ \frac{1}{|h_{x,ij}|} \int_{x_{i-\hf}}^{x_{i+\hf}}  \big|\jump{\partial^{\bm \alpha} u_h^{(q)}}_{i,j+\frac12} \big| {\rm d}x  } { \|  u_h^{(q)} - \mathrm{avg}(u_h^{(q)})  \|_{L^\infty(\Omega)}  }, ~ & {\rm otherwise}.
	\end{cases}
\end{equation*}

\begin{remark}\label{rem:Consistency}
	It is worth noting that the aforementioned multidimensional OE procedure maintains consistency with its 1D counterpart. Consequently, the multidimensional OEDG method, when applied on rectangular meshes for quasi-1D problems, exactly reduces to the 1D OEDG method in \Cref{sec:1d-oedg}. Moreover, as the 1D case, the multidimensional OEDG method also satisfies the scale-invariant and evolution-invariant properties. 
\end{remark}


\section{Optimal error estimates for linear advection equations}\label{sec:err}


In this section, we derive error estimates of the OEDG method for the linear advection equation with constant coefficients in both 1D and 2D. To facilitate the analysis, we introduce the following nonhomogeneous equation
\begin{equation}\label{eq:nonhom}
	u_t + {\bm \beta}\cdot \nabla u = \ff({\bm x},t), \qquad \text{with periodic b.c.},
\end{equation}
while the error estimate will be carried out for the homogeneous case with $\ff \equiv 0$. Here ${\bm \beta}$ is a constant (vector). We assume that $\beta > 0$ in 1D and ${\bm \beta}=(\beta^x,\beta^y)$ with $\beta^x>0$ and $\beta^y>0$ in 2D.  Other cases can be proven along similar lines. 

We will denote by $C$, possibly with subscripts, a general constant which may depend on the order of the RK method $r$, the number of RK stages $s$, the RK coefficients $c_{\ell m}$ and $d_{\ell m}$, the DG polynomial degree $k$, the mesh regularity constant, the Sobolev norms of the exact solution, and the final time $T$, etc., but is independent of the time step size $\tau$ and the mesh size $h$. Here $h$ is the maximum diameter among all mesh cells. Without a subscript, $C$ along can take different values at different places.

\subsection{Main results} 
\begin{theorem}\label{thm:err}
	Consider the OEDG method with upwind flux on 1D meshes and 2D Cartesian meshes with explicit RK time stepping. Assume that the mesh partition is quasi-uniform, $\|{u_h^{n,\ell} - \mathrm{avg}({u_h^{n,\ell}})}\|_{L^\infty(\Omega)}\geq C>0$ is uniformly positive for all $n$ and $\ell$, and $k>\frac{d}{2}$ and $r > \frac{d}{2}+1$ where $d$ is the spatial dimension. Suppose that without the OE procedure, the fully discrete RKDG scheme is stable in the sense of \cref{assp:stab}. If $u(x,t)$ is sufficiently smooth, and $\nm{u_\sig^0-u(\cdot,0)}\leq Ch^{k+1}$, then the OEDG method for  \eqref{eq:nonhom} with $\ff\equiv 0$ (outlined in \eqref{eq:oedg1D2D} and \eqref{eq:oedg-uni3}) admits the optimal error estimate
	\begin{equation}
		\max_{0\leq n\leq\nn}\nm{u_\sig^n - u(\cdot,n\tau)} \leq C\left(h^{k+1} + \tau^r\right).
	\end{equation}
	Here $\nn$ is the final time step such that $T = \nn \tau$. 
\end{theorem}

The proof of \Cref{thm:err} will be given in \Cref{sec:proofmain}. 

\begin{remark}
	To avoid additional technicalities and heavy notations, we assume that $u$ is sufficiently smooth and disregard the dependence of $C$ on the Sobolev norms of $u$. If it becomes necessary to optimally characterize such dependence, a set of more sophisticated reference functions must be introduced in the analysis. We refer to \cref{rmk:reffunc} for further details.
\end{remark}

\begin{remark}
	By similar arguments, the error estimate can also be extended to upwind-biased flux on 1D meshes using the generalized Gauss--Radau projection \cite{meng2016optimal,cheng2017application} and on special simplex meshes in two and three dimensions with the projection in \cite[Lemma 3.4]{sun2023generalized}. 
\end{remark}

\begin{remark}
	The requirement $k> \frac{d}{2}$ and $r>\frac{d}{2}+1$ is a technical assumption specifically used in \cref{lem:ind1} for our proof to go through. Numerically, we still observe optimal convergence in multidimensions with low-order polynomials and RK schemes. 
	This requirement will be used to validate the a priori (or inductive) assumption \eqref{eq:apriori}. In the error estimates of nonlinear DG schemes for time-dependent problems, it is a standard technique to first introduce an a priori assumption as in \eqref{eq:apriori} and then validate it in the end. This technique is known to cause additional restrictions on $k$ and $r$ in multidimensions, unless introducing more sophisticated arguments. We refer to \cite[Lemma 4.2]{zhang2006error} and \cite[Lemma 3.2]{huang2017error} for discussions on similar issues. 
\end{remark}

\subsection{Notations and preliminaries} 
In the case that $\cF_\tau \equiv \mathcal{I}$ is the identity operator, the OEDG scheme retrieves the RKDG scheme. Hence, it is not surprising that the error estimate of the OEDG scheme will be built upon that of the original RKDG scheme \cite{xu2020error, xu2020superconvergence}. This process requires the following ingredients: a special spatial projection with certain approximation and superconvergence properties, the stability of the RKDG scheme, and reference functions to track the error at each RK stage. To avoid repetitive arguments, we will introduce unified notations for both 1D and 2D cases and prove both error estimates together.

\subsubsection{Projections and unified notations}\;

\emph{1D projection.} For the 1D problem, we adopt the notations in \Cref{sec:1d-oedg}, and denote by $\ip{w}{v} = \sum_j \ipI{w}{v} = \sum_j \int_{I_j} w v {\rm d}x$. We introduce the bilinear form for the DG operator
\begin{equation}\label{eq:H-1D}
	\dg{w}{v} = \beta\ip{w}{v_x} + \beta\sum_j w_{j+\hf}^-\jump{v}_{j+\hf}.
\end{equation}
The following Gauss--Radau projection $P^\star$ will be used in the error estimates. Given $w$, we define $P^\star$ such that for $\eta = w - P^\star w$, we have 
\begin{equation}\label{eq:gr}
	\ipI{\eta}{v} = 0\quad \forall v \in \mathbb{P}^{k-1}(I_j) \quand P^\star\eta_{j+\hf}^- = 0. 
\end{equation}
The projection has the following approximation and superconvergence properties 
\begin{alignat}{2}
	\|\eta\|_{H^{m}(I_j)}\leq & Ch^{k+1-m}{|w|_{H^{k+1}(I_j)}} \qquad &\forall 0\leq m\leq k+1,\\
	\dg{\eta}{v} = & 0 \qquad &\forall v\in \mathbb{V}^k.
\end{alignat}

\emph{2D projection.}
For the 2D problem on Cartesian meshes, we adopt the notations in \Cref{sec:md}. We denote by $\ip{w}{v} = \sum_{i,j}\ip{w}{v}_{K_{i,j}} = \sum_{i,j} \iint_{K_{i,j}} w v {\rm d}x{\rm d}y$ and introduce the bilinear form
\begin{equation}\label{eq:H-2D}
	\begin{aligned}
		\dg{w}{v} =& \ip{w}{\beta^xv_x+\beta^yv_y}- \sum_{i,j}\int_{x_{i-\hf}}^{x_{i+\hf}}\beta^yw(x,y_{j+\hf}^-)\left(v(x,y_{j+\hf}^-)-v(x,y_{j+\hf}^+)\right){\rm d}x\\
		&- \sum_{i,j}\int_{y_{j-\hf}}^{y_{j+\hf}}\beta^xw(x_{i+\hf}^-,y)\left(v(x_{i+\hf}^-,y)-v(x_{i+\hf}^+,y)\right){\rm d}y.
	\end{aligned}
\end{equation}
The following special projection from \cite{liu2020optimal} will be used for optimal error estimates. Given $w$, we define $P^\star$ such that for $\eta = w - P^\star w$, we have 
\begin{align}\label{eq:liushuzhang}
	&\ip{\eta}{1}_{K_{i,j}} = 0,\\
	&\ip{\eta}{\beta^xv_x+\beta^yv_y}_{K_{i,j}}- \int_{x_{i-\hf}}^{x_{i+\hf}}\beta^y\eta(x,y_{j+\hf}^-)\left(v(x,y_{j+\hf}^-)-v(x,y_{j-\hf}^+)\right) {\rm d}x\\
	&- \int_{y_{j-\hf}}^{y_{j+\hf}}\beta^x\eta(x_{i+\hf}^-,y)\left(v(x_{i+\hf}^-,y)-v(x_{i-\hf}^+,y)\right) {\rm d}y = 0\quad \forall v\in \mathbb{P}^k(K_{i,j}).\nonumber
\end{align} 
The projection has the following approximation and superconvergence properties
\begin{alignat}{2}
	\|\eta\|_{H^{m}(K_{i,j})}\leq& Ch^{k+1-m}{|w|_{H^{k+1}(K_{i,j})}} \qquad &\forall 0\leq m\leq k+1,\\
	|\dg{\eta}{v}| \leq& Ch^{k+1}\nm{v} \qquad &\forall v\in \mathbb{V}^k.
\end{alignat}
The second inequality is the Cauchy--Schwarz version of \cite[(2.30)]{liu2020optimal}. It can be obtained by first applying the Cauchy--Schwarz inequality to  \cite[(2.27)]{liu2020optimal}, then the inverse estimates  \cite[(2.20)]{liu2020optimal}, next the approximation estimates  \cite[(2.19)]{liu2020optimal}, and finally the Bramble--Hilbert lemma. 
%

\emph{Unified notations for both 1D and 2D cases.} The OEDG scheme for \eqref{eq:nonhom} in 1D and in 2D for Cartesian meshes can be written in the following unified form 
\begin{subequations}\label{eq:oedg1D2D}
	\begin{align}
		\ip{u_h^{n,\ell+1}}{v} &= \sum_{0\leq \mykappa \leq \ell} \left(c_{\ell \mykappa}\ip{u_\sig^{n,\mykappa}}{v} + \tau d_{\ell \mykappa}\left(\dg{u_\sig^{n,\mykappa}}{v} + \ip{\ff^{n,\mykappa}}{v}\right)\right),\label{eq:oedg-uni1}\\
		u_\sig^{n,\ell+1} &= \cF_\tau u_h^{n,\ell+1},\qquad 0\leq \ell \leq s-1, \label{eq:oedg-uni2}
	\end{align}
\end{subequations}
and the OE step corresponding to the solution of 
\begin{equation}\label{eq:oedg-uni3}
	\left\{
	\begin{aligned}		
		&{\frac{\mathrm{d}}{\mathrm{d} \hat t}} \ipK{u_\sig}{v} + \sum_{m = 0}^k 
		\delta_K^m\ipK{u_\sig - P^{m-1}u_\sig}{v} =\; 0~~\forall v \in \mathbb P^k(K),\\
		&		u_\sig(x,0) = u_h^{n,\ell+1}(x).
	\end{aligned}\right.	
\end{equation}
Here $K = I_j$ and $\delta_K^m = \delta_K^m(u_h^{n,\ell+1}) = 	\beta_j {\sig_j^{m} (u_h^{n,\ell+1}) }/h_j$ for the 1D case; $K = K_{i,j}$ and $\delta_K^m$ is defined in \eqref{eq:delta}, or more specifically in  \eqref{eq:delta2Dcart}, for the 2D case.

The projection $P^\star$ for the error estimates, is defined as the Gauss--Radau projection in \eqref{eq:gr} for the 1D case, and the Liu--Shu--Zhang projection in \eqref{eq:liushuzhang} for the 2D case. For both cases, with $\eta = w - P^\star w$, we have 
\begin{subequations}\label{eq:projprop}
	\begin{alignat}{2}
		\|\eta\|_{H^{m}(K)}\leq& Ch^{k+1-m}{|w|_{H^{k+1}(K)}} \qquad &\forall 0\leq m\leq k+1,\label{eq:projprop-1}\\
		|\dg{\eta}{v}| \leq& Ch^{k+1}\nm{v} \qquad &\forall v\in \mathbb{V}^k.\label{eq:projprop-2}
	\end{alignat}
\end{subequations}
Note that, in particular, \eqref{eq:projprop-1} implies $\nm{\eta}\leq Ch^{k+1}$. 

The following estimates, which hold for all $v\in \mathbb{P}^k(K)$, will be used in our analysis.
\begin{subequations} 
	\begin{align}
		\nm{v}_{H^{1}(K)}\leq& Ch^{-1}\nm{v}_{L^2(K)},\label{eq:normequiv-1}\\
		\nm{v}_{L^2(\partial K)}\leq& Ch^{-\hf}\nm{v}_{L^2(K)},\label{eq:normequiv-2}\\
		\nm{v}_{L^\infty(K)}\leq& Ch^{-\frac{d}{2}}\nm{v}_{L^2(K)}.\label{eq:normequiv}
	\end{align}
\end{subequations}
Note that applying the first two inequalities to \eqref{eq:H-1D} and \eqref{eq:H-2D} gives
\begin{equation}\label{eq:H}
	|\dg{w}{v}|\leq Ch^{-1} \nm{w}\nm{v}\qquad \forall w, v\in \mathbb{V}^k.
\end{equation}
\subsubsection{Stability of the RKDG method}\label{sec:stab}
The error estimate will be built upon the stability of the original RKDG method for the nonhomogeneous equation \eqref{eq:nonhom} without the OE step.  For ease of presentation, we first state the stability as a general assumption in \cref{assp:stab}, and then state in \cref{prop:stab} on stability of the $r$th-order $r$-stage RKDG methods for linear equations. 
\begin{assumption}[Stability of the RKDG method]\label{assp:stab}
	Under the time step constraint $\tau \leq C_{\mathrm{CFL}} h^\kappa$ with $\kappa\geq 1$, the RKDG method without the OE procedure, namely \eqref{eq:oedg1D2D} with $\cF_\tau \equiv \mathcal{I}$, satisfies
	\begin{equation}\label{eq:stability}
		\nm{u_h^{n+1}}^2\leq \left(1+C_{\mathrm{s}}\tau\right)\nm{u_h^n}^2 + C\tau\sum_{0\leq \ell<s}\nm{\ff^{n,\ell}}^2.
	\end{equation}
\end{assumption} 

\cref{prop:stab} is collected from \cite{xu2020superconvergence}. See also \cite{sun2017rk4,sun2019strong,xu2019l2,sun2022energy} for related results. Most of these works focus on the case $\ff\equiv 0$, but the extension to $\ff\not\equiv 0$ is straight forward. See for example, \cite{xu2020error}, for the case of the classical fourth-order RK method. General stability criteria of an arbitrary RK method can also be found in the aforementioned literature. 
\begin{proposition}\label{prop:stab}
	Consider an $r$th-order $r$-stage RK method for linear problems. 
	\begin{enumerate}
		\item If $r\equiv 1\pmod 4$, then \cref{assp:stab} holds with $\kappa = 1+1/r$. 
		\item If $r\equiv 2\pmod 4$, then \cref{assp:stab} holds with $\kappa = 1+1/(r+1)$.
		\item If $r\equiv 3\pmod 4$, then \cref{assp:stab} holds with $\kappa = 1$ and $C_{\mathrm{s}} = 0$.
		\item If $r\equiv 0\pmod 4$, then after combining every 2 or 3 time steps as a single step and renumbering the stages, \cref{assp:stab} holds with $\kappa = 1$ and $C_{\mathrm{s}} = 0$.
	\end{enumerate}
\end{proposition}

\subsubsection{Reference functions}\label{sec:refsol} 
We will only consider $\ff \equiv 0$ in \eqref{eq:rk} in our error estimates. We introduce the reference solution $U^{n,0} = u(x,t^n)$ and define recursively 
\begin{align}
	\ip{U^{n,\ell+1}}{v} =\;& \sum_{0\leq \mykappa \leq \ell} \left(c_{\ell \mykappa}\ip{U^{n,\mykappa}}{v} + \tau d_{\ell \mykappa}\dg{U^{n,\mykappa}}{v} \right) + \tau\ip{\rho^{n,\ell}}{v}\quad \text{for }0\leq \ell \leq s-1.\label{eq:ref}
\end{align}
Here $\rho^{n,\ell} \equiv 0$ for $\ell<s-1$ and $\rho^{n,s-1}$ is set so that $U^{n,s} = u(x,t^{n+1})=U^{n+1,0}$. As a convention, $U^{n}:=U^{n,0}$. With a Taylor expansion, it can be shown that 
\begin{equation}\label{eq:estrho}
	\nm{\rho^{n,s-1}}\leq C\tau^{r}.
\end{equation}
\begin{remark}\label{rmk:reffunc}
	Since $U^{n,m}$ is smooth, we have $\dg{U^{n,m}}{v} = -\ip{{\bm \beta}\cdot \nabla U^{n,m}}{v}$. Hence it can be seen from \eqref{eq:ref} that we need $u(\cdot,t)\in H^s(\Omega)$ to ensure all $U^{n,\ell}$ and $\rho^{n,\ell}$ are well defined. But typically, we expect the required regularity of $u$ to be close to $r$ instead of $s$. This could lead to very restrictive regularity condition on $u$ for some RK schemes with $s\gg r$, or in the case $s = r$ and $r\equiv 0\pmod 4$, for which we combine multiple time steps in the stability estimate. The dependence of the regularity condition on $s$ can be avoided by carefully truncating the high-order terms in the definition of $U^{n,m}$. This more sophisticated choice of $U^{n,\ell}$ will recover the minimum regularity requirement on $u$. We refer to \cite[Section 4.1.2]{xu2020superconvergence} and \cite[Section 5.2]{xu2020error} for details, and refrain from pursuing it here for simplicity. 
\end{remark}

\subsection{Proof of \cref{thm:err}}
\label{sec:proofmain}

For simplicity, we assume that $h\leq 1$ and $\tau \leq 1$ are sufficiently small. The route map of the fully discrete error analysis is outlined as follows. 
\begin{enumerate}
	\item We make an a priori assumption for all $0\leq n\leq n_\star$ and $1\leq \ell\leq s$, 
	\begin{equation}\label{eq:apriori}
		\nm{\xi^{n,\ell}}_{L^\infty(\Omega)}:=\nm{u_h^{n,\ell} - P^\star U^{n,\ell}}_{L^\infty(\Omega)}\leq h.
	\end{equation} 
	We prove \cref{thm:err} under this assumption through the following steps.
	\begin{enumerate}
		\item We rewrite the OEDG scheme and its corresponding error equation in the form of an RKDG scheme with a source term at each stage. 
		\item We show that these additional source terms are small and high-order terms.  
		\item The steps outline above lead to an error equation within a single time step. We then apply the discrete Gr\"onwall's inequality to prove the global-in-time error estimates.
	\end{enumerate} 
	\item We prove the a priori assumption \eqref{eq:apriori} inductively.
\end{enumerate}

Now, we will provide the proof of \cref{thm:err}. The proofs of the associated lemmas will be deferred to the next section. As previously stated, Steps 1 requires the assumption  \eqref{eq:apriori}, which is used in \cref{lem:Y} at Step 1(b).

\emph{Step 1(a).} Set $\ff^{n,\mykappa}\equiv 0$ and add $\ip{u_\sig^{n,\ell+1}-u_h^{n,\ell+1}}{v}$ on both sides of \eqref{eq:oedg-uni1}. It yields
\begin{equation}
	\ip{u_\sig^{n,\ell+1}}{v} = \sum_{0\leq \mykappa \leq \ell} \left(c_{\ell \mykappa}\ip{u_\sig^{n,\mykappa}}{v} + \tau d_{\ell \mykappa}\dg{u_\sig^{n,\mykappa}}{v}\right) + \ip{u_\sig^{n,\ell+1}-u_h^{n,\ell+1}}{v}, \label{eq:rksig}
\end{equation}
for $ 0\leq \ell\leq  s-1$. We define 
\begin{equation}
	\xi_\sig^{n,m} = u_\sig^{n,m} -  P^\star  U^{n,m}, \  \xi^{n,m} = u_h^{n,m} -  P^\star  U^{n,m} \quand\eta^{n,m} = U^{n,m}- P^\star  U^{n,m}.
\end{equation} 
Subtracting \eqref{eq:ref} from \eqref{eq:rksig} gives
\begin{equation}\label{eq:xiYZ}
	\begin{aligned}
		\ip{\xi_\sig^{n,\ell+1}}{v} = &\sum_{0\leq \mykappa\leq \ell} \left(c_{\ell\mykappa}\ip{\xi_\sig^{n,\mykappa}}{v} + \tau d_{\ell \mykappa} \dg{\xi_\sig^{n,\mykappa}}{v}\right) + \tau \cY^{n,\ell}(v) + \tau \cZ^{n,\ell}(v) ,
	\end{aligned}
\end{equation}
where
\begin{align}
	\cY^{n,\ell}(v) =\;& \tau^{-1}\ip{u_\sig^{n,\ell+1}-u_h^{n,\ell+1}}{v},\label{eq:defY}\\
	\cZ^{n,\ell}(v) =\;& \ip{\eta_{\mathrm{c}}^{n,\ell}}{v} - \dg{\eta_{\mathrm{d}}^{n,\ell}}{v} - \ip{\rho^{n,\ell}}{v},\label{eq:Z}
\end{align}
with 
\begin{equation}\label{eq:etas}
	{\eta_{\mathrm{c}}^{n,\ell}} = \tau^{-1}\left({\eta^{n,\ell+1}} -\sum_{0\leq\mykappa\leq \ell}c_{\ell\mykappa}{\eta^{n,\mykappa}}\right) \quand \eta_{\mathrm{d}}^{n,\ell} = \sum_{0\leq\mykappa\leq \ell} d_{\ell \mykappa} {\eta^{n,\mykappa}}.
\end{equation}
To apply the stability estimate in \cref{assp:stab}, we rewrite \eqref{eq:xiYZ} as
\begin{equation}\label{eq:xiF}
	\ip{\xi_\sig^{n,\ell+1}}{v} = \sum_{0\leq \mykappa\leq \ell} \left(c_{\ell\mykappa}\ip{\xi_\sig^{n,\mykappa}}{v} + \tau d_{\ell \mykappa} \left( \dg{\xi_\sig^{n,\mykappa}}{v} + \ip{\FF^{n,\mykappa}}{v}\right) \right), \quad  0\leq\ell \leq  s-1.
\end{equation}
Here $\FF^{n,\mykappa}$ is defined inductively as 
\begin{equation}\label{eq:F}
	d_{\ell \ell}\ip{\FF^{n,\ell}}{v} = \cY^{n,\ell}(v) + \cZ^{n,\ell}(v) - \sum_{0\leq \mykappa \leq \ell-1}d_{\ell\mykappa}(\FF^{n,\mykappa},v), \quad 0\leq \ell \leq s-1.
\end{equation}

\emph{Step 1(b).} From the definition of $\FF^{n,\ell}$ in \eqref{eq:F}, one can see that 
\begin{equation}\label{eq:estF}
	\nm{\FF^{n,\ell}} \leq C\sum_{0\leq \mykappa\leq \ell}\left(\nm{\cY^{n,\mykappa}} + \nm{\cZ^{n,\mykappa}}\right).
\end{equation} 
Here $\nm{\cY^{n,\mykappa}} = \sup_{v\in {\mathbb V}^k, v\neq 0} \frac{\nm{\cY^{n,\mykappa}(v)}}{\nm{v}}$ and $\nm{\cZ^{n,\mykappa}} = \sup_{v\in {\mathbb V}^k, v\neq 0} \frac{\nm{\cZ^{n,\mykappa}(v)}}{\nm{v}}$. Consequently, to estimate the norm of the source term $\FF^{n,\ell}$, we require estimates of $\nm{\cZ^{n,\ell}}$ and $\nm{\cY^{n,\ell}}$, which are provided below. Their proofs will be postponed to \Cref{sec:Z} and \Cref{sec:lem:Y}, respectively. 
\begin{lemma}\label{prop:Z}
	$\nm{\cZ^{n,\ell}} \leq C\left(h^{k+1} + \tau^{r}\right).$
\end{lemma}

\begin{lemma}\label{lem:Y}
	If $\nm{\xi^{n,m}}_{L^\infty(\Omega)}\leq h$ for all $1\leq m\leq \ell+1$, then
	\begin{equation}
		\nm{\cY^{n,\ell}} \leq C\left(\nm{\xi_\sigma^n} + h^{k+1} + \tau^{r}\right).
	\end{equation}
\end{lemma}

Applying \cref{prop:Z} and \cref{lem:Y} in \eqref{eq:estF} yields
\begin{equation}\label{eq:estFF}
	\nm{\FF^{n,\ell}} \leq C \left(\nm{\xi_\sig^{n}} + h^{k+1}+\tau^r\right).
\end{equation}

\emph{Step 1(c).} From \eqref{eq:xiF}, we see that $\xi_\sig^{n,\ell}$ satisfies the RKDG scheme accompanied by the source term $\FF^{n,\ell}$. Apply the stability result \eqref{eq:stability} to \eqref{eq:xiF} and use the estimate \eqref{eq:estFF}. It gives
\begin{equation}\label{eq:stabilityYZ}
		\nm{\xi_\sig^{n+1}}^2\leq \left(1+C_{\mathrm{s}}\tau\right)\nm{\xi_\sig^n}^2 + C\tau \sum_{0\leq \ell<s}\nm{\FF^{n,\ell}}^2\
		\leq \left(1+C\tau\right)\nm{\xi_\sig^n}^2 + C\tau\left(h^{2k+2}+\tau^{2r}\right).
\end{equation}
Since constants in \eqref{eq:stabilityYZ} are independent of $n$, one can apply the discrete Gr\"onwall's inequality to get
\begin{equation}\label{eq:gronwall}
	\nm{\xi_\sig^n}^2\leq C_\star^2\left(\nm{\xi_\sig^0}^2 + h^{2k+2} + \tau^{2r}\right)\quad \forall 0\leq n \leq \nn.
\end{equation}
Here we address that $C_\star$ is a constant that is independent of $h$, $\tau$, $n$, and $\nn$, but may depend on the final time $T = \nn\tau$. Then applying the inequality $\sqrt{a^2+b^2+c^2}\leq |a|+|b|+|c|$ gives
\begin{equation}\label{eq:xifinalest}
	\nm{\xi_\sig^n}\leq C_\star\left(\nm{\xi_\sig^0} + h^{k+1} + \tau^{r}\right)\quad \forall 0\leq n \leq \nn.
\end{equation}
Recall that $u_\sig^i - u(x,t^i) = \xi_\sig^i - \eta^i$ with $i = 0,n$. We can use the triangle inequality, the approximation property of $P^\star $, the fact $\nm{u_\sig^0 - u(\cdot,0)}\leq Ch^{k+1}$, to get
\begin{equation}
	\nm{u_\sig^n-u(\cdot,n\tau)}\leq C_\star\left(h^{k+1} + \tau^{r}\right)\quad \forall 0\leq n \leq \nn.
\end{equation}
The proof of \cref{thm:err}, under the assumption \eqref{eq:apriori}, can be completed by maximizing both sides over $n$ and observing the independence of $C_\star$ from $n$.

\emph{Step 2.} Note we have the following lemma, whose proof will be given in \Cref{sec:lem:ind1}.
\begin{lemma}\label{lem:ind1}
	Suppose $k>\frac{d}{2}$ and $r>\frac{d}{2}+1$. If $\nm{\xi_\sig^n}\leq C_{\star\star}(h^{k+1}+\tau^r)$, then  $\nm{\xi_\sig^{n,\ell}}\leq C(h^{k+1}+\tau^r)$ and $\nm{\xi^{n,\ell}}_{L^\infty(\Omega)}\leq h$ for all $1\leq \ell\leq s$ under the assumption $h\leq \overline{C}_{\star\star}$. Here $\overline{C}_{\star\star}$ is a fixed constant dependent on $C_{\star\star}$. 
\end{lemma}

Now we verify the a prior assumption \eqref{eq:apriori}. 

Let us fix the constant $C_\star$ in \eqref{eq:xifinalest}, which is obtained in Step 1 under the a priori assumption \eqref{eq:apriori} and appears in the bound of $\nm{\xi_\sig^n}$ for all $0\leq n \leq \nn$. By the choice of the initial data, we have $\nm{\xi_\sig^0}\leq C_* h^{k+1}$. Without loss of generality, we assume $C_\star>1$ and $C_*>1$ and define $C_{\star\star} = C_\star (C_*+1)$. According to \cref{lem:ind1}, since $\nm{\xi_\sig^0} \leq C_{\star\star}(h^{k+1}+\tau^r)$, the {a priori assumption \eqref{eq:apriori}} holds for $n = 0$ and all {$1\leq \ell\leq s$} when $h\leq \overline{C}_{\star\star}$. Thus one can follow the argument in Step 1 and apply the discrete Gr\"onwall's inequality to get \eqref{eq:xifinalest} with $n = 1$, which can be elaborated as
\begin{equation}
	\nm{\xi_\sig^1}\leq C_\star\left(\nm{\xi_\sig^0}+h^{k+1}+\tau^r\right)\leq C_\star (C_*+1)(h^{k+1}+\tau^r) =  C_{\star\star}(h^{k+1}+\tau^r).
\end{equation}
Now the assumption in \cref{lem:ind1} becomes valid for $n = 1$. Hence one can conclude that the a priori assumption \eqref{eq:apriori} holds for $n =1$ and all {$1\leq \ell\leq s$} under the same prescribed condition $h\leq \overline{C}_{\star\star}$. Then one can continue further as that in Step 1 and apply discrete Gr\"onwall's inequality to $n=2$ to obtain
\begin{equation}
	\nm{\xi_\sig^2}\leq C_\star\left(\nm{\xi_\sig^0}+h^{k+1}+\tau^r\right) \leq C_{\star\star}(h^{k+1}+\tau^r).
\end{equation}
Thus by \cref{lem:ind1}, the a priori assumption holds for $n = 2$. This argument can be further continued and formalized by induction to show that $\nm{\xi_\sig^n}\leq C_{\star\star}(h^{k+1}+\tau^r)$, and the a priori assumption $\nm{\xi^{n,\ell}}_{L^\infty(\Omega)}\leq h$ holds for all $0\leq n\leq \nn$ and $1\leq \ell \leq s$, if $h\leq \overline{C}_{\star\star}$. This validates the a priori assumption \eqref{eq:apriori} and thus completes the proof of \cref{thm:err}.

\subsection{Proofs of lemmas}
It now remains to prove the lemmas used in the proof of \cref{thm:err}.
\subsubsection{Proof of \cref{prop:Z}}\label{sec:Z}
This lemma is a special case of \cite[Lemma 4.6]{xu2020superconvergence} with $q = 0$. Recall the definition in \eqref{eq:Z}. One can use the definition of $\rho^{n,\ell}$ in \Cref{sec:refsol} to get $|\ip{\rho^{n,\ell}}{v}|\leq C\tau^r\nm{v}$, and the superconvergence property in  \eqref{eq:projprop-2} to show $|\dg{\eta_{\mathrm{d}}}{v}| \leq C h^{k+1} \nm{v}$. Let $W_{\mathrm{c}}^{n,\ell} = U^{n,\ell+1}-\sum_{0\leq m\leq \ell}c_{\ell m} U^{n,m}$. Note that, from \eqref{eq:ref},  
\begin{equation}
	\ip{W_{\mathrm{c}}^{n,\ell}}{v} = \tau \ip{-\sum_{0\leq m\leq \ell}d_{\ell,m}{\bm \beta}\cdot \nabla U^{n,\ell} +\rho^{n,\ell}}{v}:=\tau \ip{w^{n,\ell}}{v}.
\end{equation} 
Hence $\ip{\eta_{\mathrm{c}}^{n,\ell}}{v} = \tau^{-1} \ip{W_{\mathrm{c}}^{n,\ell}-P^\star W_{\mathrm{c}}^{n,\ell}}{v} = \ip{w^{n,\ell} - P^\star w^{n,\ell}}{v}$. Then by the approximation property in \eqref{eq:projprop-1}, we have $|\ip{\eta_{\mathrm{c}}^{n,\ell}}{v}|\leq Ch^{k+1}\nm{v}$. The proposition can be proven after combining the three inequalities of $|\ip{\rho^{n,\ell}}{v}|$, $|\dg{\eta_{\mathrm{d}}}{v}|$, and $|\ip{\eta_{\mathrm{c}}^{n,\ell}}{v}|$. 
\subsubsection{Preparations for \cref{lem:Y} and \cref{lem:ind1}}
Before proving the remaining lemmas in \Cref{sec:proofmain}, we need to introduce two preparatory lemmas, \cref{lem:xi-xis} and \cref{lem:est-us-uh}. \cref{lem:delta} and \cref{lem:perpP} are used solely for proving  \cref{lem:est-us-uh} and will not be used elsewhere.  
\begin{lemma}\label{lem:xi-xis}
	$\nm{\xi^{n,\ell+1}}\leq C\left(\sum_{0\leq\mykappa\leq \ell}\nm{\xi_\sig^{n,\mykappa}}\right)+C\tau\left(h^{k+1}+\tau^r\right)$ for $0\leq \ell\leq s-1.$
\end{lemma}
\begin{proof}
	Combining \eqref{eq:xiYZ} and \eqref{eq:defY}, it yields
	\begin{equation}
		\begin{aligned}
			\ip{\xi^{n,\ell+1}}{v} =& \sum_{0\leq \mykappa\leq \ell} \left(c_{\ell\mykappa}\ip{\xi_\sig^{n,\mykappa}}{v} + \tau d_{\ell \mykappa} \dg{\xi_\sig^{n,\mykappa}}{v}\right) +\tau\cZ^{n,\ell}(v) \\
			\leq& \sum_{0\leq\mykappa\leq \ell}\left(|c_{\ell\mykappa}|\nm{\xi_\sig^{n,\mykappa}}\nm{v} + C\tau h^{-1}\nm{\xi_\sig^{n,\mykappa}}\nm{v}\right) +\tau\nm{\cZ^{n,\ell}}\nm{v}\\
			\leq& C\left(\sum_{0\leq\mykappa\leq \ell}\nm{\xi_\sig^{n,\mykappa}}+\tau\nm{\cZ^{n,\ell}}\right)\nm{v}.	
		\end{aligned}
	\end{equation}
	Here we have used the estimate \eqref{eq:H} and the time step constraint $\tau h^{-1}\leq C_{\mathrm{CFL}}h^{\kappa-1}\leq C_{\mathrm{CFL}}$. After taking $v = \xi^{n,\ell+1}$ and applying the estimate of $\nm{\cZ^{n,\ell}}$ in \cref{prop:Z}, we complete the proof of \cref{lem:xi-xis}.
\end{proof}

\begin{lemma}\label{lem:delta}
	$|\delta_K^m(u_h^{n,\ell+1})|\leq C\left(h^{-1-\frac{d}{2}}\nm{\xi^{n,\ell+1}}_{L^2(K)}+ h^{k-\frac{d}{2}}\nm{U^{n,\ell+1}}_{H^{k+1}(K)}\right).$
\end{lemma}
\begin{proof}
	We omit the superscript $(n,\ell+1)$ and denote by $w:=w^{n,\ell+1}$ with $w = u_h, \xi,\eta, U$ in this proof. We adopt the general notations in \eqref{eq:delta} and \eqref{eq:2Dsig} to prove the general case. The proof of the 1D case is similar and is hence omitted. Recall that $\|  u_h - \mathrm{avg}(u_h)  \|_{L^\infty(\Omega)}\geq C>0$ and $|\beta_{e}|\leq C$ for our linear advection equation. Hence with \eqref{eq:delta} and \eqref{eq:2Dsig}, we have 
	\begin{equation}\label{eq:delta-est-1}
		\begin{aligned}
			(\delta_K^m)^2 =& \left( \sum_{e\in \partial K}\frac{\beta_{e}}{h_{e,K}}\frac{ (2m+1)h_{e,K}^{m}}{2(2k-1)m!}  \sum_{|{\bm \alpha}|=m} \frac{ \frac{1}{|e|} \int_{e}  \big|\jump{\partial^{\bm \alpha} u_h}_e \big| {\rm d}S  } { \|  u_h - \mathrm{avg}(u_h)  \|_{L^\infty(\Omega)}  }\right)^2\\
			\leq& C\sum_{e\in\partial K}\sum_{|{\bm \alpha}|=m}h^{2m-2}\left(\frac{1}{|e|}\int_e \big|\jump{\partial^{\bm \alpha} u_h}_e \big| {\rm d}S\right)^2.
		\end{aligned}
	\end{equation}
	With the Cauchy--Schwarz inequality, the fact $\jump{u_h} = \jump{\xi - \eta +U} = \jump{\xi}-\jump{\eta}$, and the inequality $(a+b)^2 \leq 2a^2+2b^2$, we have 
	\begin{equation}
		\left(\frac{1}{|e|}\int_e \big|\jump{\partial^{\bm \alpha} u_h}_e \big| {\rm d}S\right)^2\leq \frac{1}{|e|}\int_e\jump{\partial^{\bm \alpha} u_h}_e^2 {\rm d}S \leq Ch^{-(d-1)} \left(\int_e \jump{\partial^{\bm \alpha} \xi}_e^2 {\rm d}S + \int_e \jump{\partial^{\bm \alpha} \eta}_e^2 {\rm d}S\right).
	\end{equation}
	Substituting it into \eqref{eq:delta-est-1} yields
	\begin{equation}\label{eq:deltaest2-2}
		(\delta_K^m)^2\leq C\sum_{|\bm \alpha| = m}\sum_{e\in \partial K} h^{2m-1-d}\left(\int_e \jump{\partial^{\bm \alpha} \xi}_e^2 {\rm d}S + \int_e \jump{\partial^{\bm \alpha} \eta}_e^2 {\rm d}S\right).
	\end{equation}
	With the multiplicative trace inequality (see, e.g., \cite[Lemma 3.1]{dolejvsi2002finite}) and the inverse estimates, it yields
	\begin{equation}
		\begin{aligned}
			\sum_{e\in \partial K}\int_e \jump{\partial^{\bm \alpha} \xi}_e^2 {\rm d}S \leq& {C}	\nm{\partial^{\bm \alpha} \xi}_{L^2(K)}\left(\nm{\partial^{\bm \alpha} \xi}_{H^1(K)} + h^{-1}\nm{\partial^{\bm \alpha} \xi}_{L^2(K)}\right)
			\leq Ch^{-2m-1}\nm{\xi}_{L^2(K)}^2.\label{eq:xiest2}
		\end{aligned}
	\end{equation}
	Similarly, using the approximation property of $P^\star$, it yields
	\begin{equation}
		\begin{aligned}
			\sum_{e\in \partial K}\int_e \jump{\partial^{\bm \alpha} \eta}_e^2 {\rm d}S\leq& C\nm{\partial^{\bm \alpha} \eta}_{L^2(K)}\left(\nm{\partial^{\bm \alpha} \eta}_{H^1(K)}+h^{-1}\nm{\partial^{\bm \alpha} \eta}_{L^2(K)}\right)
			\leq Ch^{2k+1-2m}\nm{U}_{H^{k+1}(K)}^2.\label{eq:etaest2}
		\end{aligned}
	\end{equation}
	Substituting \eqref{eq:xiest2} and \eqref{eq:etaest2} into \eqref{eq:deltaest2-2}, it yields
	\begin{equation}
		(\delta_K^m)^2\leq C\left(h^{-d-2}\nm{\xi}_{L^2(K)}^2+ h^{2k-d}\nm{U}_{H^{k+1}(K)}^2\right).
	\end{equation}
	The proof is completed after taking a square root on both sides and noting that $\sqrt{a^2+b^2}\leq |a|+|b|$.
\end{proof}

\begin{lemma}\label{lem:perpP}
	If $\nm{\xi^{n,\ell+1}}_{L^\infty(\Omega)}\leq h$, then 
	$\nm{u_h^{n,\ell+1}-P^{m-1}u_h^{n,\ell+1}}_{L^2(K)} \leq Ch^{1+\frac{d}{2}}$, for $0\leq m\leq k$.
\end{lemma}
\begin{proof}
	We omit the superscript $(n,\ell+1)$ and denote by $w:=w^{n,\ell+1}$ with $w = u_h, \xi,\eta, U$ in this proof. Recall $u_h = \xi - \eta +U$. With $\nm{I-P^{m-1}}_{L^2(K)}\leq 1$ and the triangle inequality, we have
	\begin{equation}
		\begin{aligned}
			&\nm{u_h-P^{m-1}u_h}_{L^2(K)} = \nm{(I-P^{m-1})\left(\xi-\eta+U\right)}_{L^2(K)}\\
			\leq& \nm{\xi}_{L^2(K)}+\nm{\eta}_{L^2(K)}+ \nm{(I-P^{m-1})U}_{L^2(K)}\\
			\leq& \nm{\xi}_{L^2(K)}+Ch^{\max{(1,m)}+\frac{d}{2}}
			\leq \nm{\xi}_{L^\infty(K)}h^{\frac{d}{2}}+C h^{1+\frac{d}{2}}.
		\end{aligned}
	\end{equation}
	Here we have used the approximation property of $P^\star$ and $P^{m-1}$ in $L^2(K)$ in the second last inequality. The proof is then completed after applying the a priori assumption \eqref{eq:apriori}. 
\end{proof}

\begin{lemma}\label{lem:est-us-uh}
	If $\nm{\xi^{n,\ell+1}}_{L^\infty(\Omega)}\leq h$, then 
	$\nm{u_\sig^{n,\ell+1}-u_h^{n,\ell+1}}  \leq C\tau \left(\nm{\xi^{n,\ell+1}} + h^{k+1}\right).$
\end{lemma}

\begin{proof}	We omit the superscripts $(n,\ell+1)$ and denote by $w:=w^{n,\ell+1}$ with $w = u_h, u_\sigma, \xi$ in this section. 	Note that $u_h$ is independent of $\hat{t}$. The OE procedure \eqref{eq:oedg-uni3}, with the unknown $u_\sig = u_\sig(x,\hat{t})$, can be written as 
	\begin{equation}\label{eq:filter2}
		\ipK{(u_\sig-u_h)_{\hat{t}}}{v} + \sum_{m = 0}^k\delta_K^m\ipK{u_\sig -u_h- P^{m-1}\left(u_\sig-u_h\right)}{v} = -\sum_{m = 0}^k \delta_K^m\ipK{u_h- P^{m-1}u_h}{v}.
	\end{equation}
	Since $\ipK{v-P^{m-1}v}{P^{m-1}v} = 0$, we can take $v = u_\sig -u_h$ to obtain
	\begin{equation}\label{eq:filter3}
		\begin{aligned}
			\hf \dth \nm{u_\sig-u_h}_{L^2(K)}^2 + \sum_{m = 0}^k \delta_K^m\nm{u_\sig -u_h- P^{m-1}\left(u_\sig-u_h\right)}_{L^2(K)}^2 \\
			= -\sum_{m = 0}^k \delta_K^m\ipK{u_h- P^{m-1}u_h}{u_\sig-u_h}.
		\end{aligned}
	\end{equation}
	Drop the second term on the left and apply the Cauchy--Schwarz inequality on the right. After canceling $\nm{u_\sig-u_h}_{L^2(K)}$ on both sides, it gives
	\begin{equation}\label{eq:us-uh}
		\begin{aligned}
			\dth \nm{u_\sig-u_h}_{L^2(K)} \leq 
			\sum_{m = 0}^k |\delta_K^m|\nm{u_h- P^{m-1}u_h}_{L^2(K)}.
		\end{aligned}
	\end{equation}

	After applying \cref{lem:perpP} and \cref{lem:delta} in \eqref{eq:us-uh}, we get
	\begin{equation}
		\dth \nm{u_\sig-u_h}_{L^2(K)} \leq C\left(\nm{\xi}_{L^2(K)}+h^{k+1}\nm{U}_{H^{k+1}(K)}\right).
	\end{equation}
	Note the right hand is independent of $\hat{t}$. After integrating the inequality from $\hat{t} = 0$ to $\hat{t} = \tau$ and using the initial condition $u_\sig(\cdot,0) = u_h(\cdot)$, we get
	\begin{equation}\label{eq:lemYK}
		\nm{u_\sig-u_h}_{L^2(K)}  \leq C\tau \left(\nm{\xi}_{L^2(K)}+ h^{k+1}\nm{U}_{H^{k+1}(K)}\right).
	\end{equation}  
	The proof of \cref{lem:est-us-uh} can be completed after taking the square on both sides of \eqref{eq:lemYK}, summing over all $K$, taking a square root, and finally applying the inequality $\sqrt{a^2+b^2}\leq |a|+|b|$. 
\end{proof}

\subsubsection{Proof of \cref{lem:Y}}\label{sec:lem:Y}
Note that $\xi_\sig^{n,\ell+1} = \xi^{n,\ell+1}+(u_\sig^{n,\ell+1}-u_h^{n,\ell+1})$. Apply the triangle inequality, \cref{lem:est-us-uh}, and then \cref{lem:xi-xis}. It gives
\begin{equation}
	\begin{aligned}
		\nm{\xi_\sig^{n,\ell+1}}\leq& \nm{\xi^{n,\ell+1}}+\nm{u_\sig^{n,\ell+1}-u_h^{n,\ell+1}}\leq C\nm{\xi^{n,\ell+1}}+C\tau h^{k+1}\\
		\leq& C\sum_{0\leq m\leq \ell}\nm{\xi_\sig^{n,m}}+C\tau\left(h^{k+1}+\tau^r\right).
	\end{aligned}
\end{equation}
Hence by induction, we have 
\begin{equation}\label{eq:estxis}
	\nm{\xi_\sig^{n,m}}\leq C\nm{\xi_\sig^n}+C\tau \left(h^{k+1}+\tau^r\right)\quad \forall 0\leq m\leq \ell+1.
\end{equation} Substituting \eqref{eq:estxis} into \cref{lem:xi-xis} yields
\begin{equation}
	\nm{\xi^{n,\ell+1}}\leq C\nm{\xi_\sig^n}+C\tau \left(h^{k+1}+\tau^r\right).
\end{equation}
Together with \cref{lem:est-us-uh}, one can get
\begin{equation}
	\nm{\cY^{n,\ell}} = \tau^{-1}\nm{u_\sig^{n,\ell+1}-u_h^{n,\ell+1}} \leq C \left(\nm{\xi_\sig^{n}} + h^{k+1}+\tau^r\right).
\end{equation}

\subsubsection{Proof of \cref{lem:ind1}}\label{sec:lem:ind1}
For convenience, we define $\xi^{n,0} = \xi_\sig^n$. For $\ell = 0$, note we automatically get $\nm{\xi_\sig^{n,0}} = \nm{\xi_\sig^{n}}\leq C_{\star\star} \left(h^{k+1}+\tau^r\right)$. Since $\tau \leq C_{\mathrm{CFL}}h^{\kappa}$, by \eqref{eq:normequiv}, we have
\begin{equation}
	\nm{\xi^{n,0}}_{L^\infty(\Omega)}\leq Ch^{-\frac{d}{2}}\nm{\xi^{n,0}} = Ch^{-\frac{d}{2}}\nm{\xi_\sig^{n}} \leq C_0h^{\min(k+1-\frac{d}{2},r\kappa-\frac{d}{2})}\leq h
\end{equation}
if $h\leq C_0^{-1/(\min(k,r-1)-d/2)}$. Hence the lemma holds for $\ell = 0$. 

Now assume the lemma is true for all $m\leq \ell$ under the assumption $h\leq \min_{0\leq m\leq \ell} C_{m}^{-1/(\min(k,r-1)-d/2)}$. By \cref{lem:xi-xis}  and the induction hypothesis, we have 
\begin{equation}\label{eq:estxiind}
	\nm{\xi^{n,\ell+1}}\leq C\left(h^{k+1}+\tau^r\right).
\end{equation}
Thus using \eqref{eq:normequiv}, we have 
\begin{equation}\label{eq:estxi-uhind}
	\nm{\xi^{n,\ell+1}}_{L^\infty(\Omega)}\leq Ch^{-\frac{d}{2}}\nm{\xi^{n,\ell+1}}\leq C_{\ell+1}h^{\min(k+1-\frac{d}{2},r\kappa-\frac{d}{2})}\leq h
\end{equation}
if $h\leq C_{\ell+1}^{-1/(\min(k,r-1)-d/2)}$. Hence with this extra constraint, the assumption in \cref{lem:est-us-uh} becomes valid and the lemma gives
\begin{equation}\label{eq:estus-uhind}
	\nm{u_\sig^{n,\ell+1}-u_h^{n,\ell+1}}  \leq C\tau \left( h^{k+1} +\tau^r\right).
\end{equation}  
Applying the triangle inequality, \eqref{eq:estxiind}, and \eqref{eq:estus-uhind}, it yields
\begin{equation}\label{eq:estxisind}
	\nm{\xi_\sig^{n,\ell+1}}\leq \nm{\xi^{n,\ell+1}} +\nm{u_\sig^{n,\ell+1}-u_h^{n,\ell+1}}\leq C(h^{k+1}+\tau^r).
\end{equation}
With \eqref{eq:estxi-uhind} and \eqref{eq:estxisind}, the lemma is proven for $m = \ell+1$ with $h\leq \min_{0\leq m\leq \ell+1} C_{m}^{-1/(\min(k,r-1)-d/2)}$. The proof is completed by induction.

So far, all lemmas in \Cref{sec:proofmain} are proven, and the proof of \Cref{thm:err} is now complete.

\section{Numerical tests}\label{sec:num}

This section presents extensive 1D and 2D benchmark 
numerical examples to validate the theoretical analysis and to 
 demonstrate the accuracy, effectiveness, and robustness of the proposed OEDG method. 
The examples encompass four models of hyperbolic conservation laws: the linear advection equation, inviscid Burgers' equation, Lighthill–-Whitham-–Richards traffic flow model, and compressible Euler equations.   
Unless stated otherwise,   
for smooth problems, we employ a $(k+1)$th-order explicit RK time discretization for the $\mathbb P^k$-based OEDG method to validate the optimal convergence rates; for problems involving discontinuities, we typically use the classic third-order strong-stability-preserving explicit RK time discretization.   
The OEDG method remains stable under the normal CFL condition, thanks to the exact solver for the OE step. 
Consequently, we set the time stepsize as $\Delta t= C_{\mathrm{CFL}} h/\beta$ for 1D problems, where $\beta$ represents the maximum wave speed, and the CFL number $C_{\mathrm{CFL}} = \frac{1}{2k+1}$ for the $\mathbb P^k$-based OEDG method. 
Analogously, for 2D problems on uniform rectangular meshes with spatial stepsizes $h_x$ and $h_y$, we adopt $\Delta t= C_{\mathrm{CFL}} / (\beta_x/h_x+\beta_y/h_y)$.  
To demonstrate the importance of scale-invariant and evolution-invariant properties, we will compare the proposed OEDG method with the OFDG method \cite{lu2021oscillation,LiuLuShu_OFDG_system}, which is coupled with a third-order exponential RK time discretization \cite{huang2018bound} to mitigate restrictive time step constraints for discontinuous problems.  
We adopt the upwind flux for linear equations and the local {Lax--Friedrichs} flux for the others. Both OEDG and OFDG methods are implemented in C/C++ with double precision.  

\subsection{1D linear advection equation} 
We consider two examples of the advection equation on the spatial domain $\Omega = [0,1]$ with periodic boundary conditions. 

\begin{exmp}[smooth problem]\label{ex:1Dadvec1}
	The first example is used to validate the optimal convergence rates and 
	explore the superconvergence by 
	solving the equation $u_t + u_x=0$ with the initial condition $u_0(x) = \sin^2 (2\pi x)$ up to time $t=1.1$. 
	We measure the following five types of numerical errors: 
	\begin{align*}
	&e_1 := \| u-u_h \|_{L^1(\Omega)}, \qquad e_2 := \| u-u_h \|_{L^2(\Omega)}, \qquad e_3 := \| u-u_h \|_{L^\infty(\Omega)},
	\\
	&e_4 : =\left( h \sum_{j} \left| \frac{1}{h} \int_{I_j} (u - u_h) {\rm d} x  \right|^2 \right)^{\frac12},\quad e_5:= \left( h\sum_j \left| u(x_{j+\frac12}) -{u_{h,j+\frac{1}{2}}^-} \right|^2 \right)^\frac12. 
\end{align*}
	To verify the optimal convergence rates of $e_1$, $e_2$, and $e_3$, 
	we employ a $(k+1)$th-order explicit RK time discretization for the $\mathbb P^k$-based OEDG method. 
	For the investigation of superconvergence phenomena related to $e_4$ and $e_5$, we alternatively use a seventh-order RK method and follow 
	\cite{lu2021oscillation} to 
	  initialize $u_h(x,0) = P^\star u_0(x)$ by using the  
	Gauss--Radau projection $P^\star$ in \eqref{eq:gr}. 
	Table \ref{linear_accuracy} lists the numerical errors at $t=1.1$ for the OEDG method with $N_x$ uniform cells. 
	Note that on coarse meshes, the high-order damping effect dominates the numerical errors, resulting in convergence rates higher than $k+1$, a phenomenon also observed in the OFDG method \cite{lu2021oscillation}.  
	For comparison, \Cref{1Dlinear_stdDG} gives the error table of the conventional DG method under the same settings but without the OE procedure. 
	Our results demonstrate the optimal convergence rates in three different norms ($e_1$, $e_2$, and $e_3$), thereby confirming the fully discrete error analysis in \Cref{sec:err}. 
	We also observe the $(k+2)$th-order superconvergence in $e_4$ and $e_5$. 
	 The same superconvergence phenomena were also observed and proven for the semidiscrete OFDG method \cite{lu2021oscillation}, but the rigorous proof is currently lacking for the fully discrete OEDG method and will be explored in the future work. 
\end{exmp}

%
%
%

	\begin{table}[!htb]
		\centering
		\caption{Errors and convergence rates for $\mathbb P^k$-based OEDG method at different mesh resolutions.}
		
		\begin{tabular}{c|c|c|c|c|c|c|c|c|c|c|c} 
			\bottomrule[1.0pt]
			$k$& 	$N_x$ &$e_1$ &  rate & $e_2$  &  rate & $e_3$  & rate &  $e_4$  &  rate &  $e_5$  &  rate   \\
			\hline
%
		
\multirow{5}*{$1$}					
&128&1.69e-3&	-&	1.96e-3&	-&	3.62e-3&	-&	2.06e-3& -&	2.07e-3&	-\\ 
&256&2.81e-4&	2.59&	3.38e-4&	2.54&	6.14e-4&	2.56&	2.85e-4&	2.85&	2.86e-4&	2.85\\ 
&512&5.96e-5&	2.24&	6.78e-5&	2.32&	1.17e-4&	2.39&	3.72e-5&	2.94&	3.73e-5&	2.94\\ 
&1024&1.52e-5&	1.97&	1.70e-5&	1.99&	2.98e-5&	1.98&	4.73e-6&	2.97&	4.74e-6&	2.97\\ 
&2048&3.69e-6&	2.05&	4.10e-6&	2.05&	7.02e-6&	2.08&	5.97e-7&	2.99&	5.98e-7&	2.99\\ 

\hline

\multirow{5}*{$2$}				
&128&9.85e-6&	-&	1.08e-5&	-&	2.12e-5&	-&	9.93e-6&	-&	9.97e-6&	-\\ 
&256&6.53e-7&	3.91&	7.18e-7&	3.92&	1.68e-6&	3.66&	5.54e-7&	4.16&	5.56e-7&	4.17\\ 
&512&5.25e-8&	3.64&	5.85e-8&	3.62&	1.61e-7&	3.38&	3.28e-8&	4.08&	3.29e-8&	4.08\\ 
&1024&5.01e-9&	3.39&	5.68e-9&	3.36&	1.75e-8&	3.20&	2.00e-9&	4.04&	2.01e-9&	4.04\\ 
&2048&5.41e-10&	3.21&	6.23e-10&	3.19&	2.05e-9&	3.09&	1.23e-10&	4.02&	1.24e-10&	4.02\\ 
\hline
\multirow{4}*{$3$}								
&128&9.10e-8&	-&	1.02e-7&	-&	2.00e-7&	-&	1.04e-7&	-&	1.05e-7&	-\\ 
&256&2.91e-9&	4.97&	3.27e-9&	4.97&	7.63e-9&	4.71&	3.30e-9&	4.98&	3.31e-9&	4.98\\ 
&512&9.45e-11&	4.94&	1.09e-10&	4.91&	3.45e-10&	4.47&	1.04e-10&	4.99&	1.04e-10&	4.99\\ 
&1024&3.33e-12&	4.83&	4.02e-12&	4.76&	1.73e-11&	4.31&	3.26e-12&	5.00&	3.28e-12&	4.99\\ 		
		
				\toprule[1.0pt]
			
		\end{tabular}
		\label{linear_accuracy}
	\end{table}

		\begin{table}[!htb]
			\centering
	\caption{Errors and convergence rates for $\mathbb P^k$-based conventional RKDG method without OE step.}
	\begin{center}
		\begin{tabular}{c|c|c|c|c|c|c|c} 
			\bottomrule[1.0pt]
			$k$ &	$N_x$ &$e_1$  & rate & $e_2$  & rate & $e_3$  & rate   \\
			\hline
			
			\multirow{5}*{$1$}					
&128&8.12e-4&	-&	9.02e-4&	-&	1.39e-3&	-\\ 
&256&2.04e-4&	1.99&	2.27e-4&	1.99&	3.57e-4&	1.97\\ 
&512&5.10e-5&	2.00&	5.67e-5&	2.00&	9.03e-5&	1.98\\ 
&1024&1.29e-5&	1.99&	1.43e-5&	1.99&	2.28e-5&	1.99\\ 
&2048&3.20e-6&	2.01&	3.56e-6&	2.01&	5.71e-6&	2.00\\ 

\hline

\multirow{5}*{$2$}
&128&1.93e-6&	-&	2.25e-6&	-&	7.86e-6&	-\\ 
&256&2.40e-7&	3.00&	2.81e-7&	3.00&	9.88e-7&	2.99\\ 
&512&3.00e-8&	3.00&	3.51e-8&	3.00&	1.24e-7&	3.00\\ 
&1024&3.74e-9&	3.00&	4.39e-9&	3.00&	1.55e-8&	3.00\\ 
&2048&4.68e-10&	3.00&	5.48e-10&	3.00&	1.94e-9&	3.00\\ 

\hline 
\multirow{4}*{$3$}
&128&7.36e-9&	-&	9.98e-9&	-&	5.22e-8&	-\\ 
&256&4.60e-10&	4.00&	6.24e-10&	4.00&	3.27e-9&	4.00\\ 
&512&2.88e-11&	4.00&	3.90e-11&	4.00&	2.04e-10&	4.00\\ 
&1024&1.74e-12&	4.05&	2.41e-12&	4.02&	1.29e-11&	3.99\\ 				
			
			\toprule[1.0pt]
		\end{tabular}
		\label{1Dlinear_stdDG}
	\end{center}
\end{table}

\begin{exmp}[discontinuous problems in different scales]\label{ex2} 
This example illustrates that our OEDG method is devoid of spurious oscillations and consistently displays scale invariance and evolution invariance for problems spanning different scales and wave speeds, in contrast to the OFDG method with non-scale-invariant damping \eqref{eq:LuLiuShuDamping}. 
	We examine the linear advection equation $u_t + \beta u_x=0$ 
with discontinuous initial data described by 
\begin{equation}\label{eq:u0_exp2}
		u(x,0) = \lambda u_0(x), \qquad u_0(x):=\begin{cases}
			\sin (2\pi x), \quad &  x\in[0.3,0.8],\\
			\cos (2\pi x)-0.5, \quad & {\rm{otherwise}}.\\
		\end{cases} 	
\end{equation}
\end{exmp}
The exact solution is given by $u(x,t)= \lambda u_0(x-\beta t)$. 
To assess scale invariance, we set $\beta=1$ and vary $\lambda$ within $\{1,0.01,100\}$, 
thus scaling $u$ to represent the 
 use of different units for $u$. The cell averages of the numerical solutions $u_h^\lambda/\lambda$ at $t=1.1$, obtained by using the third-order OEDG and OFDG schemes with $256$ uniform cells, are shown in \Cref{fig1:a,fig1:b}, respectively. 
As we can see, the OEDG solution consistently upholds scale invariance: $u_h^\lambda/\lambda = u_h^1$ for all $\lambda$. 
Such consistency allows the OEDG method to eliminate spurious oscillations effectively, irrespective of the units assigned to $u$. 
Conversely, the OFDG method with damping \eqref{eq:LuLiuShuDamping} lacks scale invariance, yielding varied numerical results for different 
 $\lambda$ values, as shown \Cref{fig1:b}. More specifically, one can observe 
that the OFDG method exhibits persistent spurious oscillations (overshoots or undershoots near the discontinuity) in  
the small-scale case ($\lambda = 0.01$) and displays excessive smearing in the large-scale case ($\lambda = 100$). These observations are consistent with the analyses in \Cref{thm:scale-invariant}, \Cref{thm:homogeneity}, and \Cref{rem:LuLiuShuDamping}. 

Subsequently, we investigate the evolution-invariant attribute by fixing $\lambda=1$ and varying wave speed  $\beta \in \{1,0.01,100\}$, which corresponds to the use of different units for $t$. 
The numerical solutions, $u_h^\beta$, at $t=\frac{1.1}{\beta}$, obtained using the third-order OEDG and OFDG schemes with $256$ uniform cells, are depicted in \Cref{fig1:c,fig1:d}, respectively. 
As expected, we observe that the OEDG solution exactly satisfies the evolution invariance: $u_h^\beta (t=\frac{1.1}{\beta}) = u_h^1 (t=1.1)$ for all $\beta$. 
Hence, the OEDG method performs consistently well for different wave speeds. 
In contrast, the OFDG method is not evolution-invariant,  resulting in very different numerical results for 
 different $\beta$, as shown in \Cref{fig1:d}. Specifically,  the OFDG method does not fully suppress spurious oscillations (overshoots or undershoots) for $\beta = 100$ and causes serious smearing for $\beta = 0.01$. These findings, asserted by our analyses in \Cref{sec:invariance}, 
demonstrate the pivotal roles of scale-invariant and evolution-invariant properties in
 effectively eliminating spurious oscillations, 
thereby accentuating the superiority of the proposed OEDG method. 
The results of OFDG method will be improved if our new scale-invariant and evolution-invariant damping is used instead.

\subsection{1D inviscid Burgers' equation}
This subsection considers the nonlinear Burgers' equation $u_t+(\frac{u^2}{2})_x = 0$ on the domain $\Omega = [0,2\pi]$ with periodic boundary conditions. 

\begin{exmp}[smooth problem]
	We take the initial solution as $u_0(x)= \sin (x)+0.5$ and 
	conduct the simulation up to $t=0.6$, during which the exact solution remains smooth. 
	The smoothness allows us to study the convergence order of the OEDG method for such a nonlinear equation. 
	\Cref{burgers_con} presents the numerical errors and corresponding convergence rates 
	for the  $\mathbb P^k$-based OEDG method with a $(k+1)$th-order explicit RK time discretization at different mesh resolutions. 
	We clearly observe the optimal $(k+1)$th-order convergence order for the $\mathbb P^k$-based OEDG method. 
	This finding suggests that 
	the optimal convergence rates of the OEDG method are applicable to nonlinear equations as long as the solution remains smooth, 
even though our theoretical error estimates are solely provided for linear equations.
		\begin{table}[!htb]
			\centering
			\caption{Errors and convergence rates for $\mathbb P^k$-based OEDG method for 1D Burgers' equation.}
			\begin{center}
				\begin{tabular}{c|c|c|c|c|c|c|c} 
					\bottomrule[1.0pt]
					$k$ &	$N_x$ &$L^1$ error & rate & $L^2$ error  & rate & $L^\infty$ error  & rate   \\
					\hline
%
					
\multirow{8}*{$1$} 
&64&3.09e-3&	-&	2.09e-3&	-&	4.38e-3&	-\\ 
&128&6.88e-4&	2.17&	5.00e-4&	2.06&	1.11e-3&	1.98\\ 
&256&1.66e-4&	2.05&	1.36e-4&	1.87&	3.36e-4&	1.73\\
&512&4.15e-5&	2.00&	3.54e-5&	1.95&	8.76e-5&	1.94\\ 
&1024&1.05e-5&	1.98&	9.07e-6&	1.96&	2.24e-5&	1.97\\ 
&2048&2.67e-6&	1.98&	2.30e-6&	1.98&	5.67e-6&	1.98\\ 
&4096&6.72e-7&	1.99&	5.81e-7&	1.99&	1.42e-6&	1.99\\ 
&8192&1.69e-7&	1.99&	1.46e-7&	1.99&	3.57e-7&	2.00\\ 		
\hline
\multirow{8}*{$2$} 
&64&8.76e-5&	-&	9.75e-5&	-&	3.07e-4&	-\\ 
&128&9.44e-6&	3.21&	1.06e-5&	3.20&	3.70e-5&	3.05\\ 
&256&1.12e-6&	3.08&	1.28e-6&	3.05&	4.50e-6&	3.04\\ 
&512&1.37e-7&	3.03&	1.58e-7&	3.02&	5.52e-7&	3.03\\ 
&1024&1.69e-8&	3.01&	1.97e-8&	3.01&	6.90e-8&	3.00\\ 
&2048&2.11e-9&	3.00&	2.46e-9&	3.00&	8.59e-9&	3.00\\ 	
&4096&2.64e-10&	3.00&	3.07e-10&	3.00&	1.07e-9&	3.00\\ 
&8192&3.29e-11&	3.00&	3.84e-11&	3.00&	1.34e-10&	3.00\\ 		
\hline
\multirow{5}*{$3$} 
&64&2.61e-6&	-&	3.41e-6&	-&	1.14e-5&	-\\ 
&128&1.34e-7&	4.28&	1.80e-7&	4.24&	6.02e-7&	4.24\\ 	
&256&8.10e-9&	4.05&	1.04e-8&	4.12&	3.02e-8&	4.32\\ 
&512&5.16e-10&	3.97&	6.44e-10&	4.01&	1.46e-9&	4.37\\ 
&1024&3.31e-11&	3.96&	4.08e-11&	3.98&	8.89e-11&	4.03\\ 					
					
				\toprule[1.0pt]
				\end{tabular}
				\label{burgers_con}
			\end{center}
		\end{table}
	
\end{exmp}

\subsection{Lighthill–-Whitham-–Richards traffic flow model}
This model \cite{lu2008explicit} is governed by a scalar conservation law 
\begin{equation}\label{eq:trafficflow}
	u_t + f(u)_x=0, \quad f(u) = \begin{cases}
	-0.4 u^2 + 100 u, \quad & 0\le u \le 50,
	\\
	-0.1u^2 + 15 u + 3500, \quad & 50 \le u \le 100,
	\\
	-0.024 u^2 -5.2 u + 4760, \quad  & 100 \le u \le 350,   
\end{cases}
\end{equation}
where $t$ represents time in hours, $x$ stands for distance in kilometers, and $u$ denotes the traffic density in vehicles per kilometer (veh/km). 
Given the nonhomogeneous nature of $f(u)$, the exact solution operator for equation \eqref{eq:trafficflow} is not homogeneous. However, the application of a scale-invariant damping operator remains pivotal. 
To show the importance of scale invariance and evolution invariance, we reformulate equation \eqref{eq:trafficflow} in different units, resulting in the following two equivalent forms \eqref{eq:tf2} and \eqref{eq:tf3}. 
 Using a new time variable, $\tilde t = 60 t$, measured in minutes, the governing equation \eqref{eq:trafficflow} becomes 
	\begin{equation}\label{eq:tf2}
		u_{\tilde t} + \frac{1}{60} f(u)_x=0.
	\end{equation}
 By redefining the distance variable as $\tilde x=1000x$ 
(measured in meters) and reinterpreting the density as $\tilde u = u/1000$ 
(measured in vehicles per meter, veh/m), we obtain  
\begin{equation}\label{eq:tf3}
	{\tilde u}_{t} + f(1000 \tilde u)_{\tilde x}=0.
\end{equation}



\begin{exmp}\label{ex:tf}
	Following \cite{lu2008explicit}, we simulate a traffic flow on a homogeneous freeway stretching over 20 km. 
	Initially, the density at the entrance ($x=0$) is set as 50 veh/km, 
	and an accident on the freeway creates a piecewise linear traffic density profile, as depicted in \cite[Figure 6]{lu2008explicit}. This profile features a congestion spanning 5 km, specifically from the 10 km to 15 km measured from the entrance. 
	To alleviate the congestion, the entrance is temporarily closed for 10 minutes. Following this closure, traffic resumes from the entrance at a heightened capacity density of 75 veh/km. However, 20 minutes later, the entrance flow reverts to its original density of 50 veh/km. At the freeway's exit, there is a traffic signal that operates on a cyclic pattern: 2 minutes of green light (indicating zero density) and 1 minute of red light (indicating jam density 350 veh/km). 
	We perform the simulation for a duration of one hour, by solving the three 
	equivalent governing equations \eqref{eq:trafficflow}, \eqref{eq:tf2}, and \eqref{eq:tf3}, respectively. 
	Figure \ref{trafficflow} presents the cell averages of the numerical solutions computed by the third-order  
	OEDG and OFDG methods with $800$ uniform cells in the domain $[0, 20~{\rm km}]$. The reference solution is obtained by the local Lax--Friedrichs scheme on a very fine mesh of 80,000 cells. 
	Due to the scale-invariant and evolution-invariant attributes, the OEDG schemes for all three equations yield congruent results, devoid of any nonphysical oscillations. 
	This stands in contrast to the OFDG method, which when integrated with damping \eqref{eq:LuLiuShuDamping}, produces inconsistent numerical results for these three equivalent equations in varying units. 
	Specifically,  
	the OFDG solutions result in notable smearing for equation \eqref{eq:tf2} and 
	exhibit pronounced spurious oscillations in the case of equation \eqref{eq:tf3}.



	\begin{figure}[!htp]
	\centering
	\begin{subfigure}[h]{1\linewidth}
		\centering
		\includegraphics[width=0.32\textwidth]{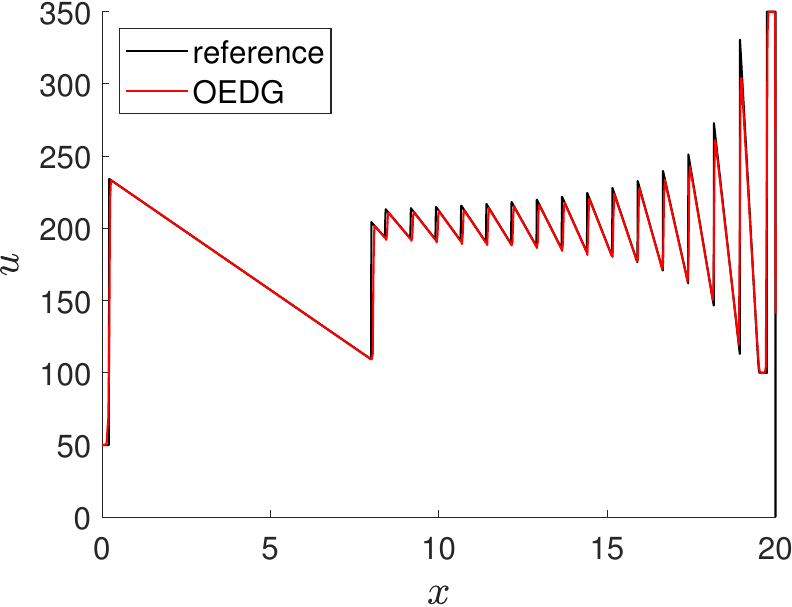} \hfill 
		\includegraphics[width=0.32\textwidth]{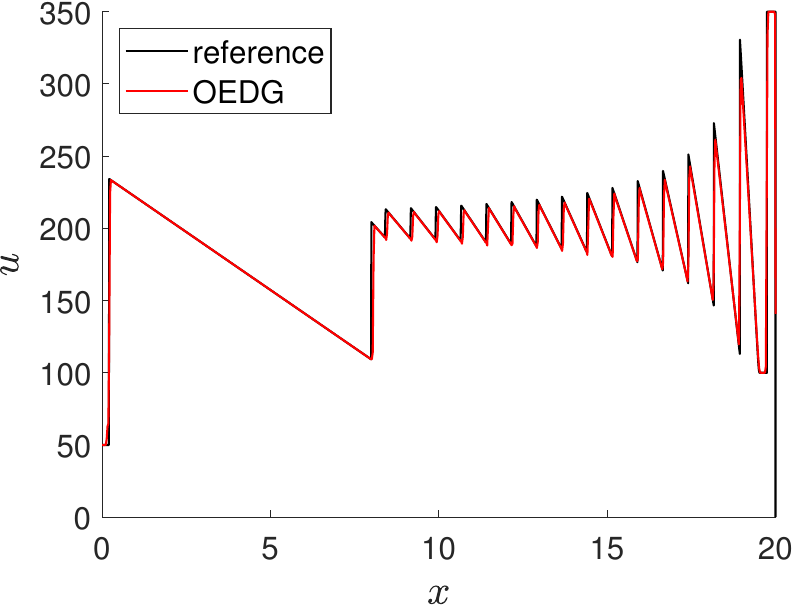} \hfill 
		\includegraphics[width=0.32\textwidth]{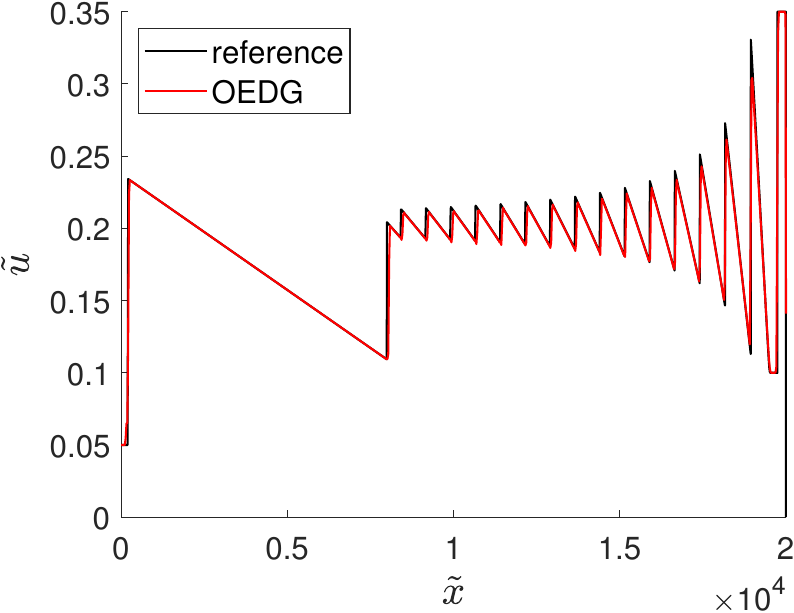}
		\subcaption{Proposed OEDG for equations \eqref{eq:trafficflow}, \eqref{eq:tf2}, and \eqref{eq:tf3}, from left to right}	
		\vspace{3mm}
	\end{subfigure}
	\begin{subfigure}[h]{1\linewidth}
		\centering
		\includegraphics[width=0.32\textwidth]{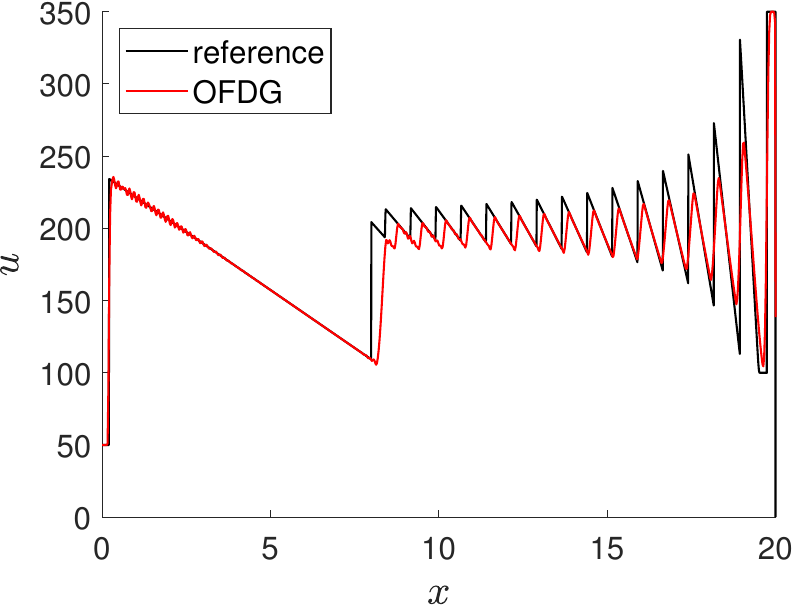}  \hfill 
		\includegraphics[width=0.32\textwidth]{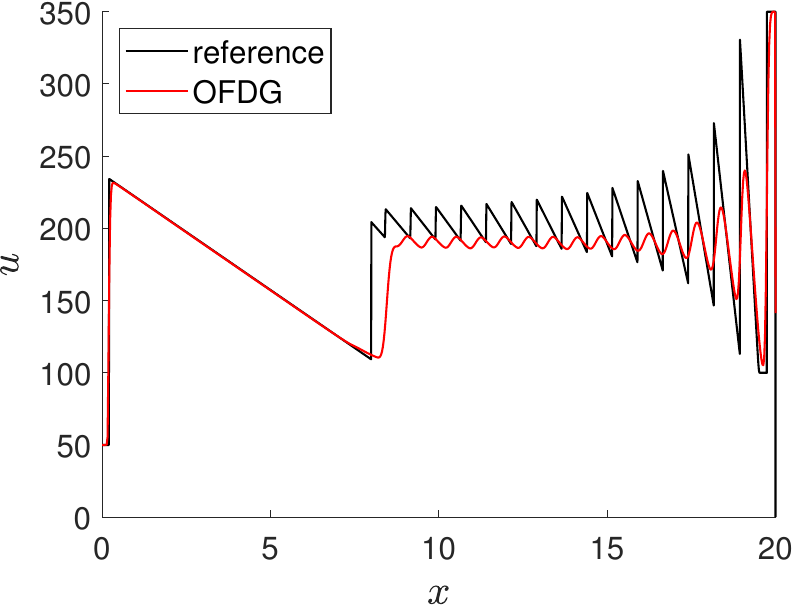} \hfill 
		\includegraphics[width=0.32\textwidth]{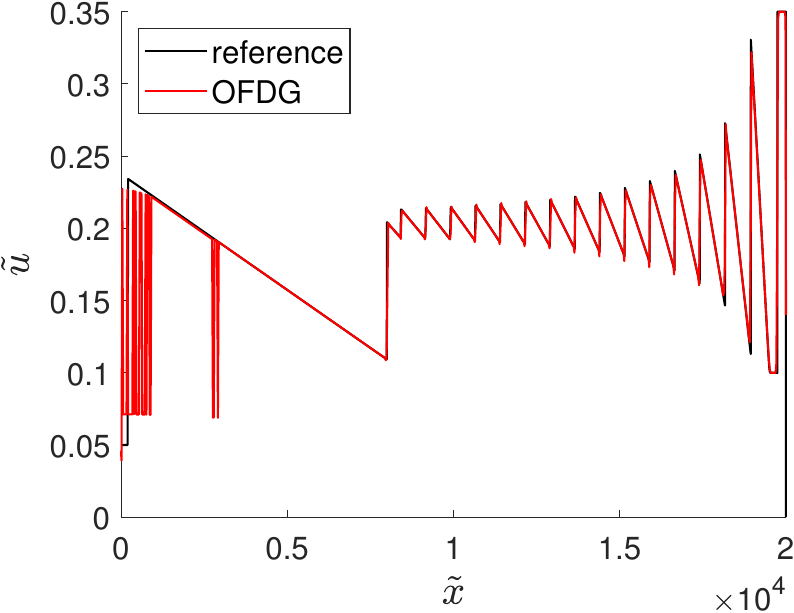}
		\subcaption{OFDG for equations \eqref{eq:trafficflow}, \eqref{eq:tf2}, and \eqref{eq:tf3}, from left to right}	
	\end{subfigure}
	\caption{Density at one hour computed by OEDG and OFDG methods solving three equivalent equations in different units for \Cref{ex:tf}. Cell averages are plotted.}
	\label{trafficflow}
\end{figure}

\end{exmp}

\subsection{1D compressible Euler equations}

This subsection considers several examples of the compressible Euler equations, 
which can be written in the form of \eqref{HCL} with ${\bf u}=(\rho, \rho v, E)^\top$ and 
${\bf f}({\bf u}) = (\rho v, \rho v^2 +p, (E+p)v)^\top $
Here $\rho$ is the density, $v$ is the velocity, $p$ is the pressure,  
$E=\frac{1}{2} \rho v^2+\frac{p}{\gamma -1}$ denotes the total energy, and 
the adiabatic index $\gamma$ is taken as $1.4$ unless otherwise stated.

\begin{exmp}[smooth problem]
	This example tests a smooth problem with the exact solution
	$$
	\rho(x,t) = 2+2 \sin^2(x-t), \qquad v(x,t) = 1, \qquad p(x,t)=2.  
	$$
	The computational domain is $\Omega = [0,2\pi]$ with periodic boundary conditions. The simulations are conducted up to $t=1.1$. 
	\Cref{1Deuler_acc} gives the numerical errors and corresponding convergence rates 
	for  $\mathbb P^k$-based OEDG method with $(k+1)$th-order explicit RK time discretization with $C_{\mathrm{CFL}} = \frac{0.95}{2k+1}$. 
	Again, 
	we observe the optimal $(k+1)$th-order convergence order for the $\mathbb P^k$-based OEDG method, as expected.  
	
	\begin{table}[!htb]
		\centering
		\caption{Errors and convergence rates for $\mathbb P^k$-based OEDG method for 1D Euler equations.} 
		\begin{center}
			
			\begin{tabular}{c|c|c|c|c|c|c|c} 
				\bottomrule[1.0pt]
			$k$	&	$N_x$ &  $L^1$ error & rate & $L^2$ error & rate & $L^\infty$ error & rate   \\
				\hline
					\multirow{7}*{$1$}

&256&8.72e-4&  2.78&  4.01e-4&  2.77&  3.29e-4&  2.71\\ 
&512&1.43e-4& 2.61&  6.60e-5&  2.61&  5.33e-5&  2.63\\ 
&1024&2.68e-5&    2.41&  1.33e-5&  2.31&  1.07e-5&  2.32\\ 
&2048&5.86e-6&    2.19&  3.10e-6&  2.11&  3.03e-6&  1.81\\ 
&4096&1.42e-6&    2.05&  7.58e-7&  2.03&  8.04e-7&  1.91\\ 
&8192&3.50e-7&    2.02&  1.89e-7&  2.01&  2.07e-7&  1.96\\ 
&16384&8.70e-8&   2.01&  4.71e-8&  2.00&  5.25e-8&  1.98\\

\hline
\multirow{4}*{$2$}

&256&5.12e-6&  4.02&  2.60e-6&  3.93&  2.39e-6&  3.81\\ 
&512&4.84e-7& 3.40&  2.46e-7&  3.40&  2.25e-7&  3.40\\ 
&1024&5.49e-8&    3.14&  2.78e-8&  3.14&  2.51e-8&  3.17\\ 
&2048&6.63e-9&    3.05&  3.36e-9&  3.05&  3.00e-9&  3.06\\ 
&4096&8.20e-10&    3.01&  4.15e-10&  3.02&  3.69e-10&  3.03\\

\hline

\multirow{4}*{$3$}& 256 & 1.40e-8 & - & 6.31e-9 & - & 4.45e-9 & - \\ 
& 512 & 4.51e-10 & 4.95 & 2.07e-10 & 4.93 & 1.60e-10 & 4.80 \\ 
& 1024 & 1.58e-11 & 4.83 & 7.59e-12 & 4.77 & 6.32e-12 & 4.66 \\ 
& 2048 & 7.57e-13 & 4.39 & 3.57e-13 & 4.41 & 4.11e-13 & 3.94 \\ 
				\toprule[1.0pt]
			\end{tabular}
		\end{center}
		\label{1Deuler_acc}
	\end{table}
\end{exmp}

\begin{exmp}[Lax problem]
	In this test, we investigate a classical Riemann problem, the Lax problem, with the scaled initial data ${\bf u}^\lambda(x,0)=\lambda {\bf u}_0(x)$, where ${\bf u}_0(x)=(\rho_0, \rho_0 v_0, E_0)^\top$ is defined by  
	\begin{equation*}
		(\rho_0,v_0,p_0)=
		\begin{cases}
			 (0.445,~0.698,~3.528),\quad & x<0,\\
			 (0.5,~0,~0.571),\quad & x>0.
		\end{cases}
	\end{equation*}
	We choose the computational domain $\Omega=[-5,5]$ and apply outflow boundary conditions. It is noteworthy that the exact solution conforms to 
	$\frac{1}{\lambda}{\bf u}^\lambda(x,t) = {\bf u}^1(x,t)$ for all $\lambda > 0$. 
	To assess the scale-invariant property, we choose three different $\lambda$ within $\{1,0.01,100\}$. 
    \Cref{Lax} presents the numerical solutions at $t=1.3$ obtained using the third-order OEDG and OFDG methods, respectively, both with $256$ uniform cells. For clarity, we plot the OFDG solution polynomials and cell averages; we also  
     zoom in some regions where spurious oscillations are produced.  
   	It is seen that the OEDG solution is scale-invariant and effectively captures the shock and contact discontinuity, devoid of noticeable spurious oscillations across all values of $\lambda$. 
   	Conversely, the OFDG method with non-scale-invariant damping \eqref{eq:LuLiuShuDamping} yields inconsistent numerical outputs across different  
   	$\lambda$ values. Specifically, \Cref{fig:Lax:OFDG:poly} clearly shows that the OFDG solution exhibits 
   	excessive smearing for $\lambda =100$ and presents spurious oscillations proximate to the 
   	shock and contact discontinuity when $\lambda = 0.01$. 
   	Some undershoots are also observed in the OFDG solution polynomials near the shock even in the case of $\lambda=1$. We would like to point out that the results of OFDG method will be improved if the scale-invariant damping \eqref{eq:1Dsigma} is employed to replace its original damping \eqref{eq:LuLiuShuDamping}.

	\begin{figure}[!htp]
		\centering
			\begin{subfigure}[h]{.99\linewidth}
		\centering
		\includegraphics[width=0.48\textwidth]{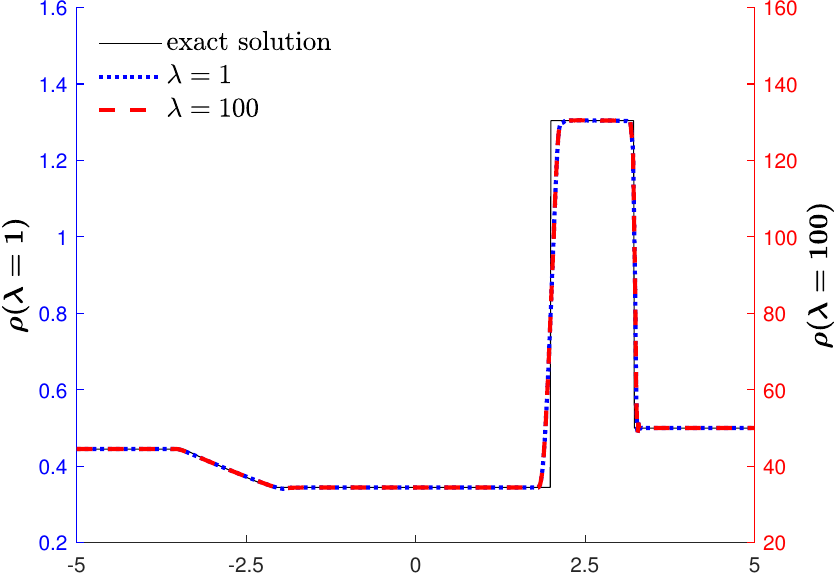} \hfill 
		\includegraphics[width=0.48\textwidth]{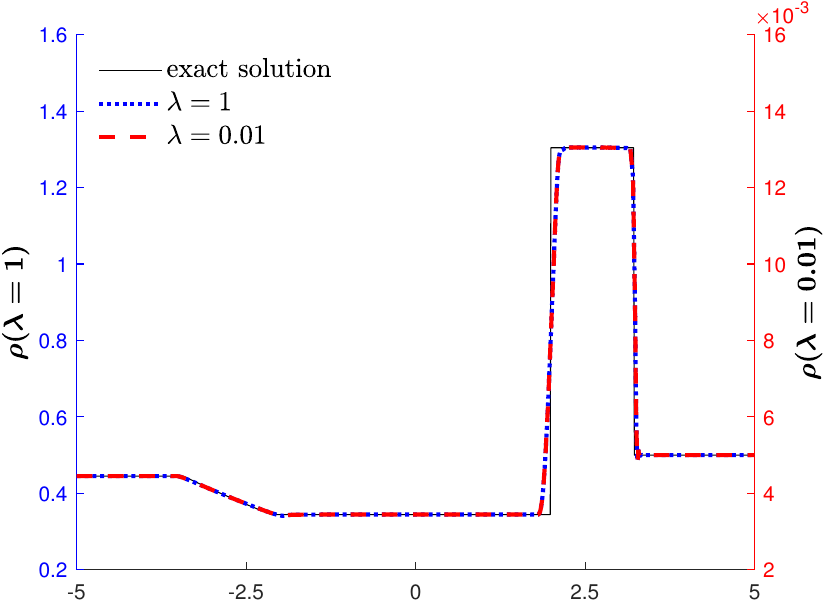}
		\subcaption{Proposed OEDG: solution polynomials are plotted}	
	\end{subfigure}
		\begin{subfigure}[h]{.99\linewidth}
			\centering
			\includegraphics[width=0.48\textwidth]{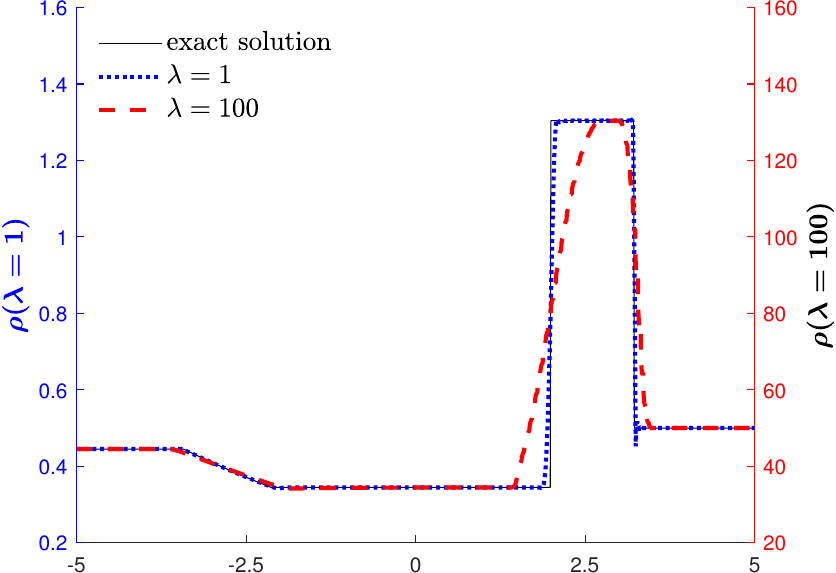}  \hfill 
			\includegraphics[width=0.48\textwidth]{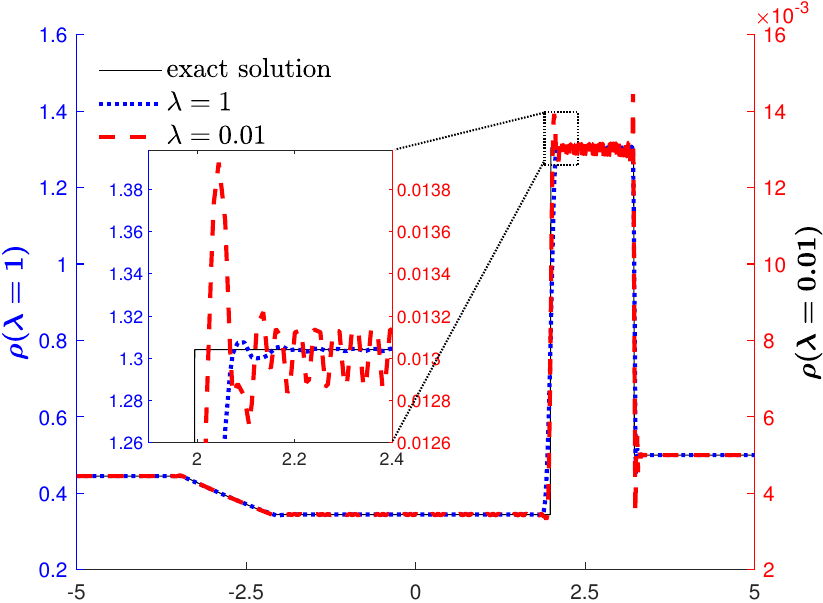}
			\subcaption{OFDG: solution polynomials are plotted}	
			\label{fig:Lax:OFDG:poly}
		\end{subfigure}
			\begin{subfigure}[h]{.99\linewidth}
		\centering
		\includegraphics[width=0.48\textwidth]{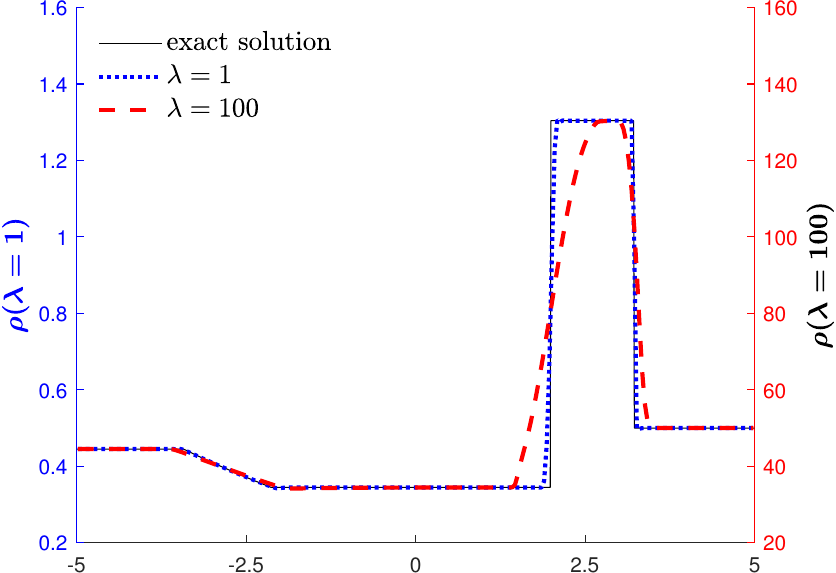}  \hfill 
		\includegraphics[width=0.48\textwidth]{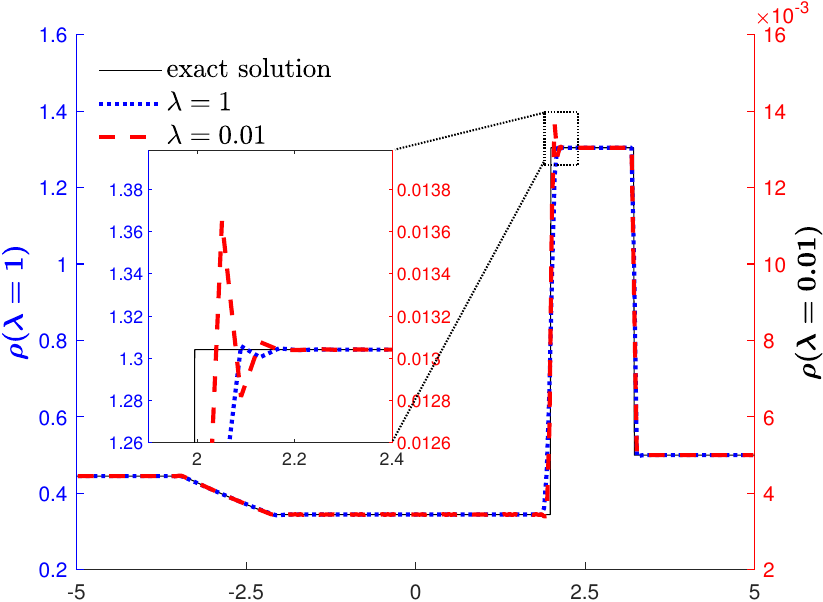}
		\subcaption{OFDG: cell averages are plotted}	
	\end{subfigure}
		\caption{Density of Lax problem at $t = 1.3$ computed by OEDG and OFDG methods. }
		\label{Lax}
	\end{figure}
	
\end{exmp}

\begin{exmp}[Woodward--Colella blast wave]
	This example simulates the interaction of two blast waves in the domain $[0,1]$ with reflective boundary conditions.  The initial conditions are defined as 
	\begin{equation*}
		(\rho_0,v_0,p_0)=
			\begin{cases}
				(1,0, 10^3),\quad & 0<x<0.1,\\
				(1,0, 10^{-2}),\quad & 0.1<x<0.9,\\
				(1,0, 10^2),\quad & 0.9<x<1.\\
			\end{cases}
	\end{equation*}
	Figure \ref{tb} displays the numerical results at $t=0.038$ obtained by using the proposed OEDG method, in comparison with the OFDG method, both on the uniform mesh of 640 cells. Here the DG solution polynomials are plotted. 
	 The reference solution is computed by using the local Lax--Friedrichs scheme on a fine mesh of 100,000 uniform cells. 
	For comparison, we also present the numerical solutions computed with or without characteristic decomposition. 
	 As seen from Figure \ref{tb}, 
	 the OFDG method exhibits much smearing without using characteristic decomposition, which is  
	 necessitated to achieve satisfactory results---a finding that aligns with observations in \cite[Figure 3.5]{LiuLuShu_OFDG_system}. 
	 In contrast, the proposed OEDG method works well either with or without the use of characteristic decomposition, and notably, it offers superior resolution compared to the OFDG method.



	\begin{figure}[!htb]
	\centering
	\begin{subfigure}[h]{.48\linewidth}
		\centering
		\includegraphics[width=.99\textwidth]{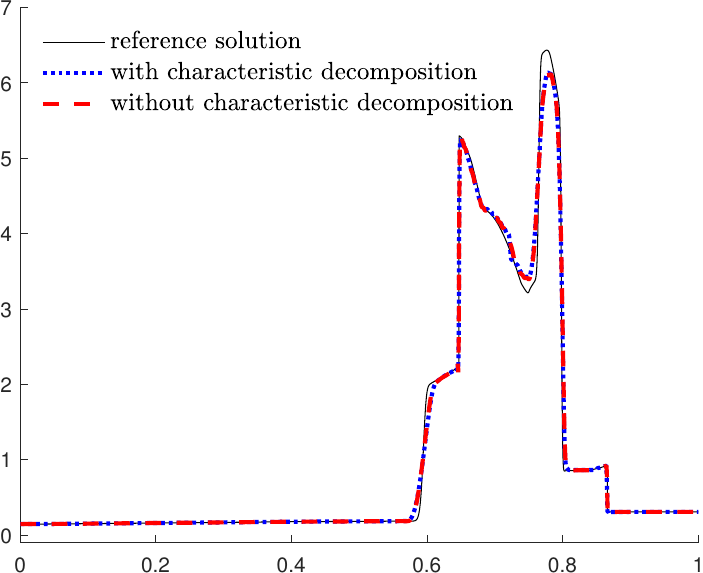}
		\subcaption{Proposed OEDG}	
	\end{subfigure}	
	\hfill 
	\begin{subfigure}[h]{.48\linewidth}
		\centering
		\includegraphics[width=.99\textwidth]{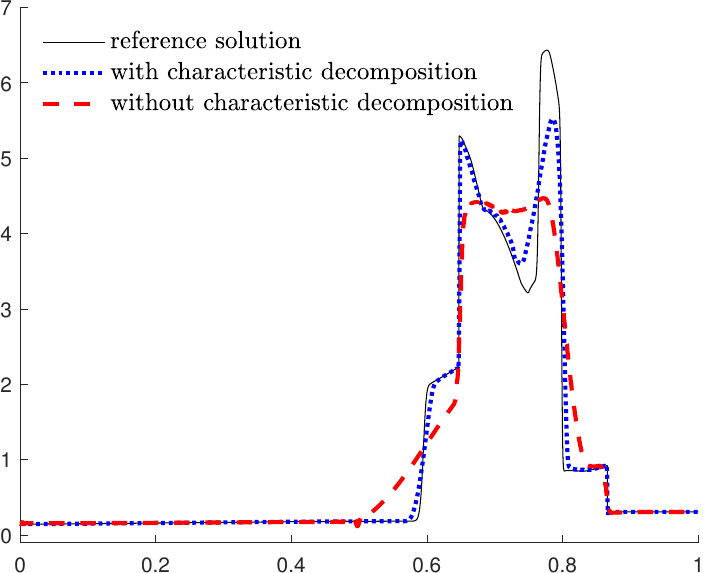}
		\subcaption{OFDG}	
	\end{subfigure}
	\caption{Density of Woodward--Colella blast wave problem at $t=0.038$ computed by OEDG and OFDG methods with or without characteristic decomposition. Solution polynomials are plotted. }
	\label{tb}
\end{figure}

\end{exmp}

\begin{exmp}[Shu--Osher problem]
	In this example, we consider the Shu--Osher problem, with the scaled initial data ${\bf u}^\lambda(x,0)=\lambda {\bf u}_0(x)$, where ${\bf u}_0(x)=(\rho_0, \rho_0 v_0, E_0)^\top$ defined by  
		\begin{equation*}
		(\rho_0,v_0,p_0)=
			\begin{cases}
				(3.857143, 2.629369, 10.33333),\quad &x<-4,\\
				(1 + 0.2{\rm sin}(5x), 0, 1),\qquad & x>-4.
			\end{cases}
	\end{equation*}
	This problem describes the interaction of sine waves and a right-moving shock. It is typically used to 
	examine the capability of high-order numerical schemes. 
	We vary $\lambda$ within $\{1,0.01,100\}$ to demonstrate the importance of ensuring scale invariance. 
	\Cref{Shu-Osher} reports the numerical results at $t=1.8$ simulated by the OEDG and OFDG methods, respectively, both using the $\mathbb P^3$ DG element with 400 uniform cells. Here the DG solution polynomials are plotted.  
	The reference solution is computed by using the local Lax--Friedrichs scheme on a very fine mesh of 300,000 uniform cells. 
	  We see that the proposed OEDG method exactly maintains the scale invariance and provides consistently good results without spurious oscillations for all $\lambda$.   
  The OFDG method with non-scale-invariant damping \eqref{eq:LuLiuShuDamping}, although giving satisfactory solution for $\lambda =1$,   
  yields some small spurious oscillations when $\lambda = 0.01$ and produces 
  too much dissipation in the case of $\lambda =100$. 
  This, again, underscores the importance of ensuring scale invariance when simulating problems across different scales.

	\begin{figure}[!htb]
		\centering
		\begin{subfigure}[h]{.99\linewidth}
			\centering
			\includegraphics[width=0.48\textwidth]{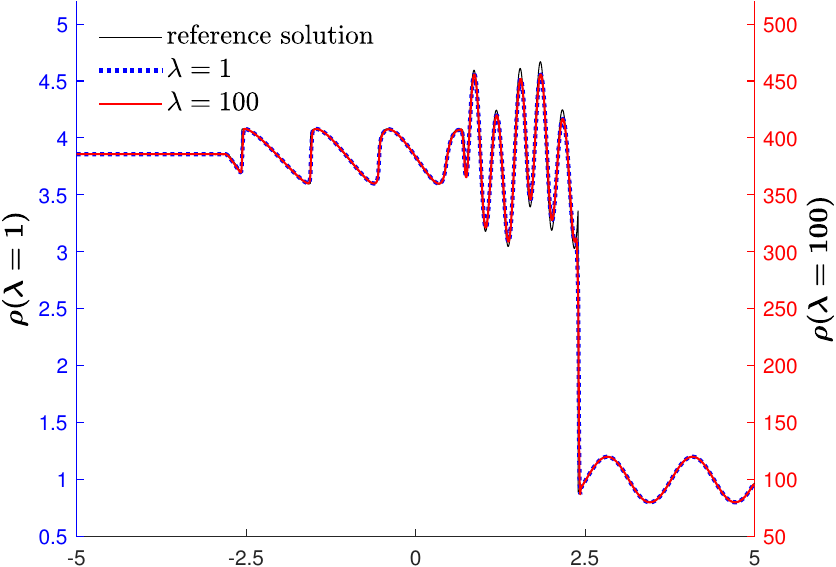} \hfill 
			\includegraphics[width=0.48\textwidth]{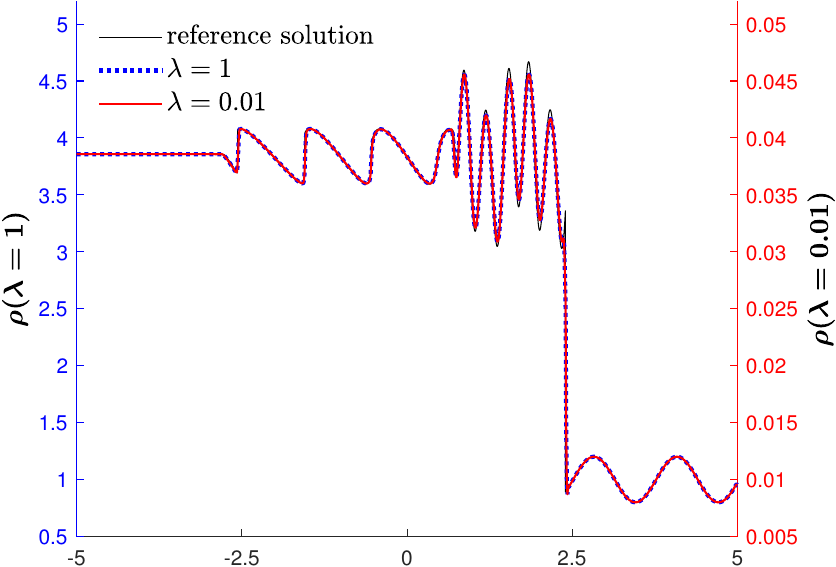}
			\subcaption{Proposed OEDG with scale-invariant damping \eqref{eq:1Dsigma}}
			\vspace{3.6mm}
		\end{subfigure}
		\begin{subfigure}[h]{.99\linewidth}
			\centering
			\includegraphics[width=0.48\textwidth]{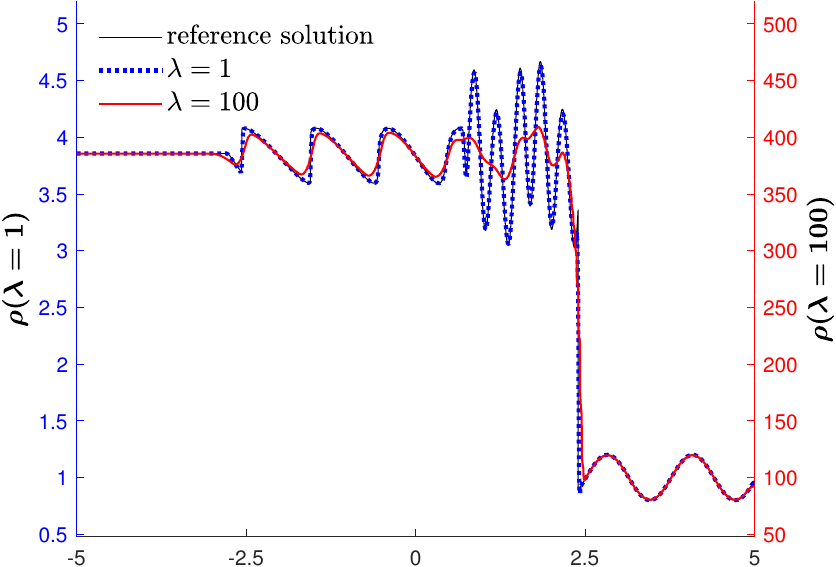} \hfill 
			\includegraphics[width=0.48\textwidth]{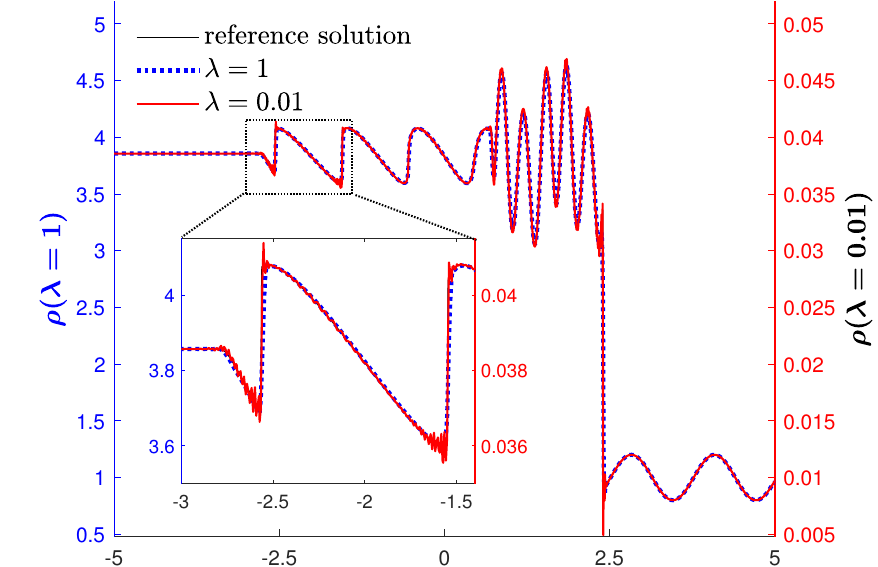}
			\subcaption{OFDG with non-scale-invariant damping \eqref{eq:LuLiuShuDamping}}
		\end{subfigure}	
		\caption{Density of Shu--Osher problem at $t = 1.8$ computed by OEDG and OFDG methods. Solution polynomials are plotted.}
		\label{Shu-Osher}
	\end{figure}
\end{exmp}

\subsection{2D linear advection equation}

In this subsection, we present two 2D examples of the advection equation $u_t+u_x+u_y = 0$ in the domain $\Omega = [-1,1]^2$ with periodic boundary conditions. 

\begin{exmp}[smooth problem] \label{ex:2DsmoothLinear}
	The first example serves as a smooth test case to validate the accuracy of the proposed 2D OEDG schemes. 
	The initial solution is $u_0(x,y) = \sin^2(\pi(x+y))$. 
	\Cref{table:test5} lists the numerical errors at $t=1.1$ and corresponding convergence orders in three different norms for the 2D $\mathbb P^k$-based OEDG method with $(k+1)$th-order RK time discretization. 
	We observe that the convergence rates surpass the theoretical rate of $k+1$. 
	This is because the high-order damping effect dominates the errors on coarse meshes. 
 However, as the mesh is refined, the convergence rates will progressively approach the anticipated theoretical rate.
 Similar phenomena were also observed for 
 the 2D OFDG method in \cite[Table 4.3]{lu2021oscillation}. 
	

		\begin{table}[!htb]
			\centering
			\begin{center}
				\caption{Errors and convergence rates for 2D $\mathbb P^k$-based OEDG method for \Cref{ex:2DsmoothLinear}.}\label{table:test5}
				\begin{tabular}{c|c|c|c|c|c|c|c} 
					\bottomrule[1.0pt]
				$k$	&	$N_x \times N_y$ &$L^1$ error& rate & $L^2$ error & rate & $L^\infty$ error &  rate   \\ \hline \multirow{6}*{$1$}
					
					& $80\times64$ &1.46e-2&- &1.80e-2&- &3.25e-2&-  \\
					& $160\times128$ &2.56e-3&2.51 &2.95e-3&2.61 &5.23e-3&2.64  \\
					& $320\times256$ &3.85e-4&2.74 &4.45e-4&2.73 &7.57e-4&2.79 \\
					& $640\times512$ &5.89e-5&2.71 &7.07e-5&2.65 &1.36e-4&2.48 \\
					& $1280\times1024$ &1.17e-5&2.33 &1.35e-5&2.39 &2.55e-5&2.41 \\
					& $2560\times2048$ &2.73e-6&2.10 &3.06e-6&2.15 &5.38e-6&2.25 	\\

					\hline \multirow{6}*{$2$}
					
					&$80\times64$&2.15e-3&- &2.42e-3&- &3.95e-3&-  \\
					&$160\times128$&8.69e-5&4.63 &9.63e-5&4.65 &1.53e-4&4.69 \\
					&$320\times256$&5.48e-6&3.99 &6.09e-6&3.98 &1.05e-5&3.87 \\
					&$640\times512$&3.53e-7&3.96 &3.94e-7&3.95 &7.80e-7&3.75 \\
					&$1280\times1024$&2.37e-8&3.90 &2.70e-8&3.87 &6.49e-8&3.59 \\
					&$2560\times2048$&1.76e-9&3.75 &2.10e-9&3.68 &6.24e-9&3.38 \\
					
				\hline \multirow{6}*{$3$}		
				&$80\times64$&1.90e-4&6.11 &2.22e-4&6.00 &3.91e-4&5.79 \\
				&$160\times128$&4.99e-6&5.25 &5.78e-6&5.27 &1.02e-5&5.26 \\
				&$320\times256$&1.60e-7&4.96 &1.87e-7&4.95 &3.83e-7&4.74 \\
				&$640\times512$&5.22e-9&4.94 &6.25e-9&4.91 &1.56e-8&4.62 \\
				&$1280\times1024$&1.84e-10&4.83 &2.36e-10&4.73 &7.65e-10&4.35 \\			
					
					\toprule[1.0pt]
				\end{tabular}
			\end{center}
		\end{table}
 
\end{exmp}

\begin{exmp}[pentagram discontinuities] \label{ex:pentagram}
	The initial solution of this example is discontinuous: 
 \begin{equation*}
		u_0(x,y)=
			\begin{cases}
				1 ,\quad& r\leq \frac{1}{8}(3+3^{\sin(5\theta)}),\\
				0,\quad & \text{otherwise},\\
			\end{cases}
		\qquad \theta =
			\begin{cases}
				{\rm arccos}(\frac{x}{r}) ,\quad & y\geq 0,\\
				2\pi-{\rm arccos}(\frac{x}{r}),\quad &y<0,\\
			\end{cases}
	\end{equation*}	
where $r = \sqrt{x^2+y^2}$. The computational domain $\Omega = [-1,1]^2$ is divided into $320 \times 320$ uniform rectangular cells. The DG solutions at $t=1.8$ obtained by using the OEDG schemes are visualized in \Cref{2D_lineardis0}. 
	We see that the numerical solutions agree well with the exact one, and the discontinuities in a 
	pentagram shape are correctly captured with high resolution and without spurious oscillations. 
 
	\begin{figure}[!htb]
	\centering

	\includegraphics[width=.32\textwidth]{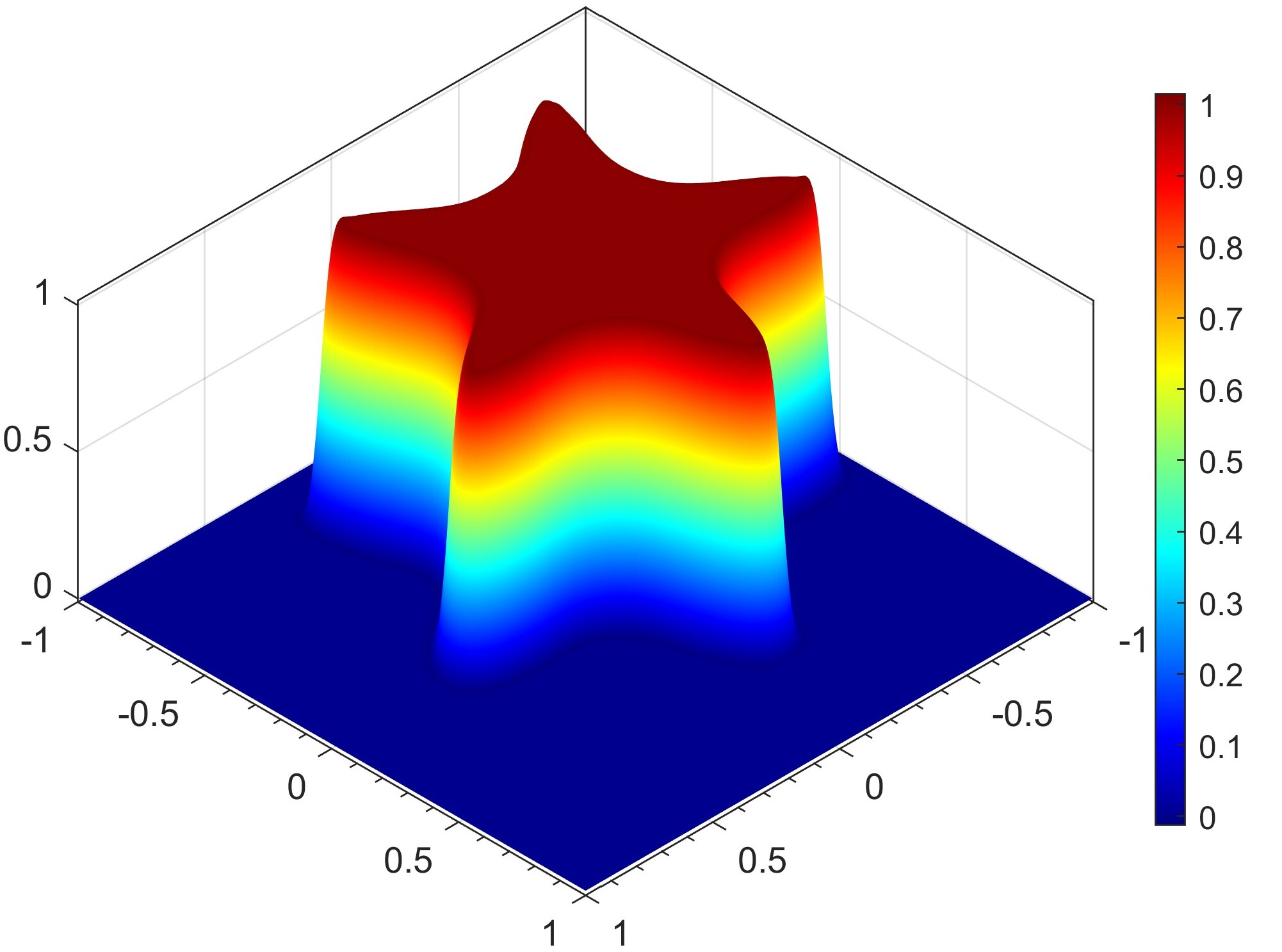}\hfill 
	\includegraphics[width=.32\textwidth]{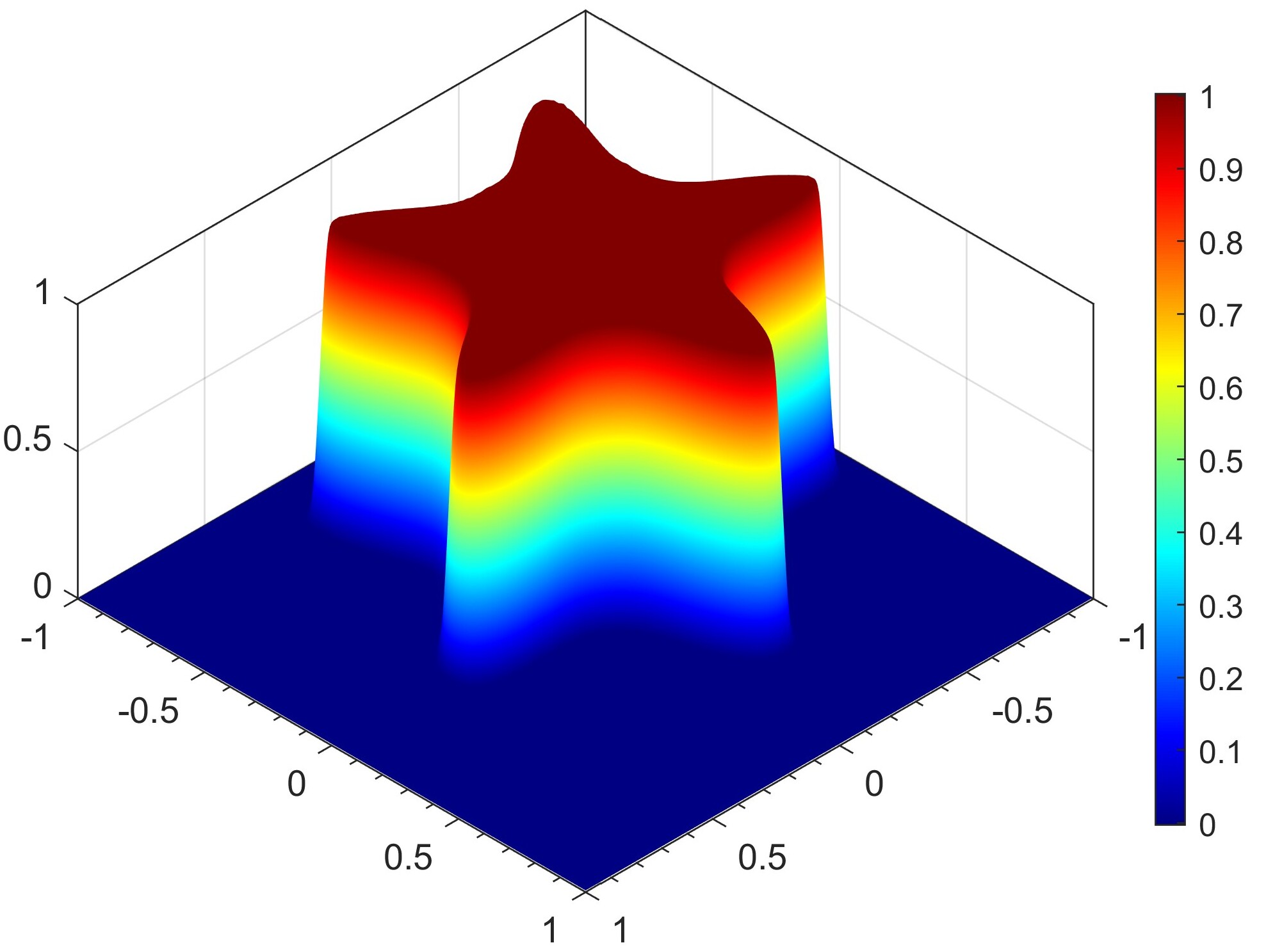}\hfill 
	\includegraphics[width=.32\textwidth]{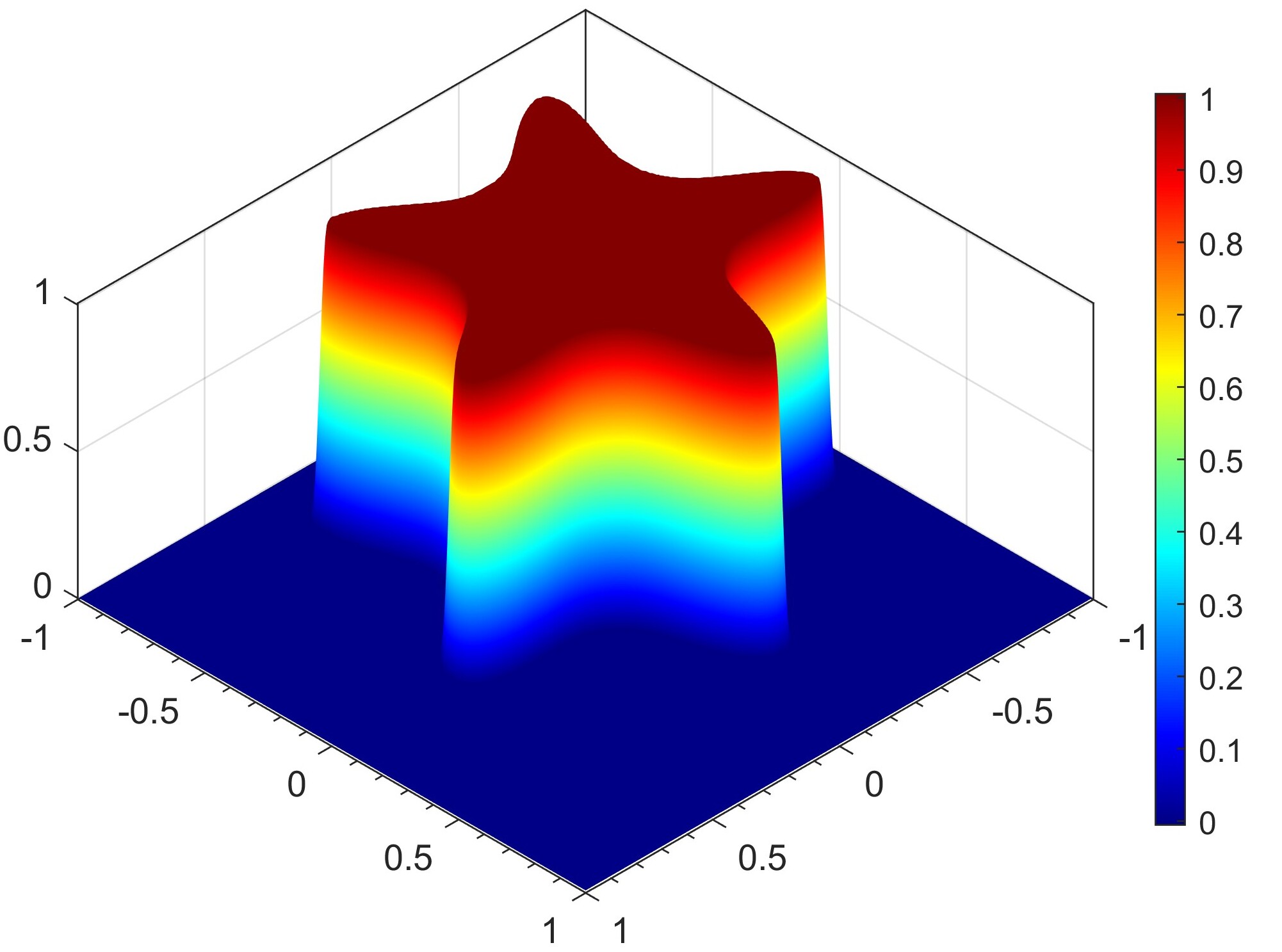}

	\includegraphics[width=.32\textwidth]{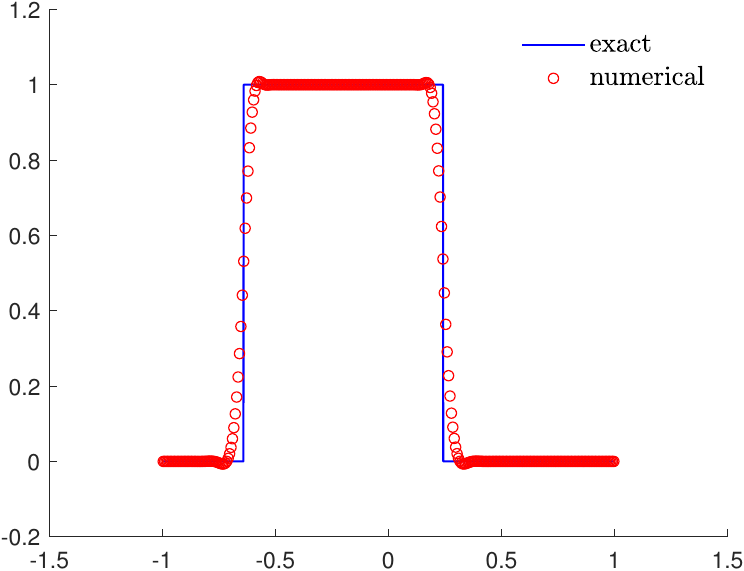}\hfill 
	\includegraphics[width=.32\textwidth]{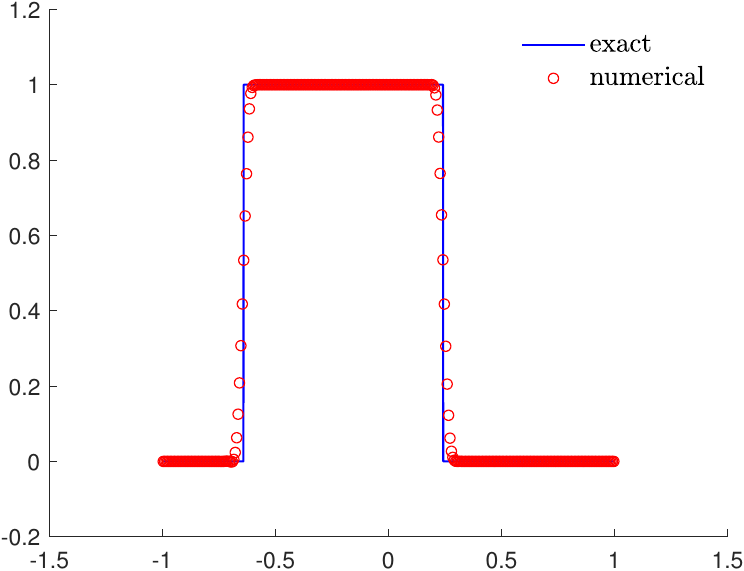}\hfill 
	\includegraphics[width=.32\textwidth]{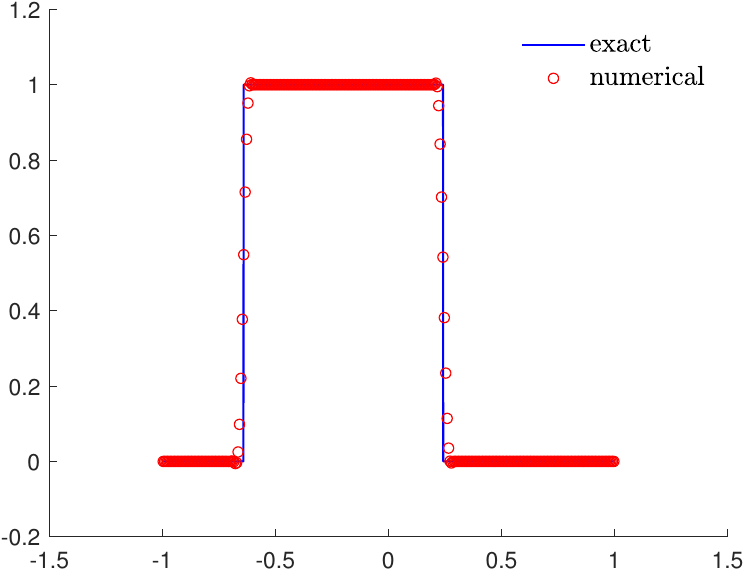} 
	
	\caption{OEDG solutions (top) at $t=1.8$ and their cut along $y=-0.25$ (bottom) for \Cref{ex:pentagram}. From left to right: $\mathbb P^1$, $\mathbb P^2$, and $\mathbb P^3$ elements. 
		}
	\label{2D_lineardis0}
\end{figure}


\end{exmp}
\subsection{2D inviscid Burgers' equation}
In this subsection, we consider the 2D Burgers' equation $u_t+(\frac{u^2}{2})_x+(\frac{u^2}{2})_y=0$ in the domain $\Omega =[0,1]^2$. 

\begin{exmp}[2D Riemann problem]\label{ex:2DburgersRP}
	This is a 2D Riemann problem with the initial data given by 
		\begin{equation*}
		u_0(x,y)=
			\begin{cases}
				0.5 ,\quad& x<0.5,y<0.5,\\
				0.8,\quad &x>0.5,y<0.5,\\
				-1,\quad &x>0.5,y>0.5,\\
				-0.2,\quad &x<0.5,y>0.5.\\
			\end{cases}
	\end{equation*}
	The computational domain $\Omega$ is partitioned into $256 \times 256$ uniform rectangular cells. Outflow boundary conditions are imposed on all the edges of $\Omega$. 
	\Cref{2DburgersRie} presents the DG solutions at $t=0.5$ computed using the OEDG schemes with $\mathbb P^1$, $\mathbb P^2$, and $\mathbb P^3$ elements, respectively. 
	It is worth noting that the numerical solutions exhibit no spurious oscillations. 
 

	\begin{figure}[!htb]
		\centering
		\includegraphics[width=.32\textwidth]{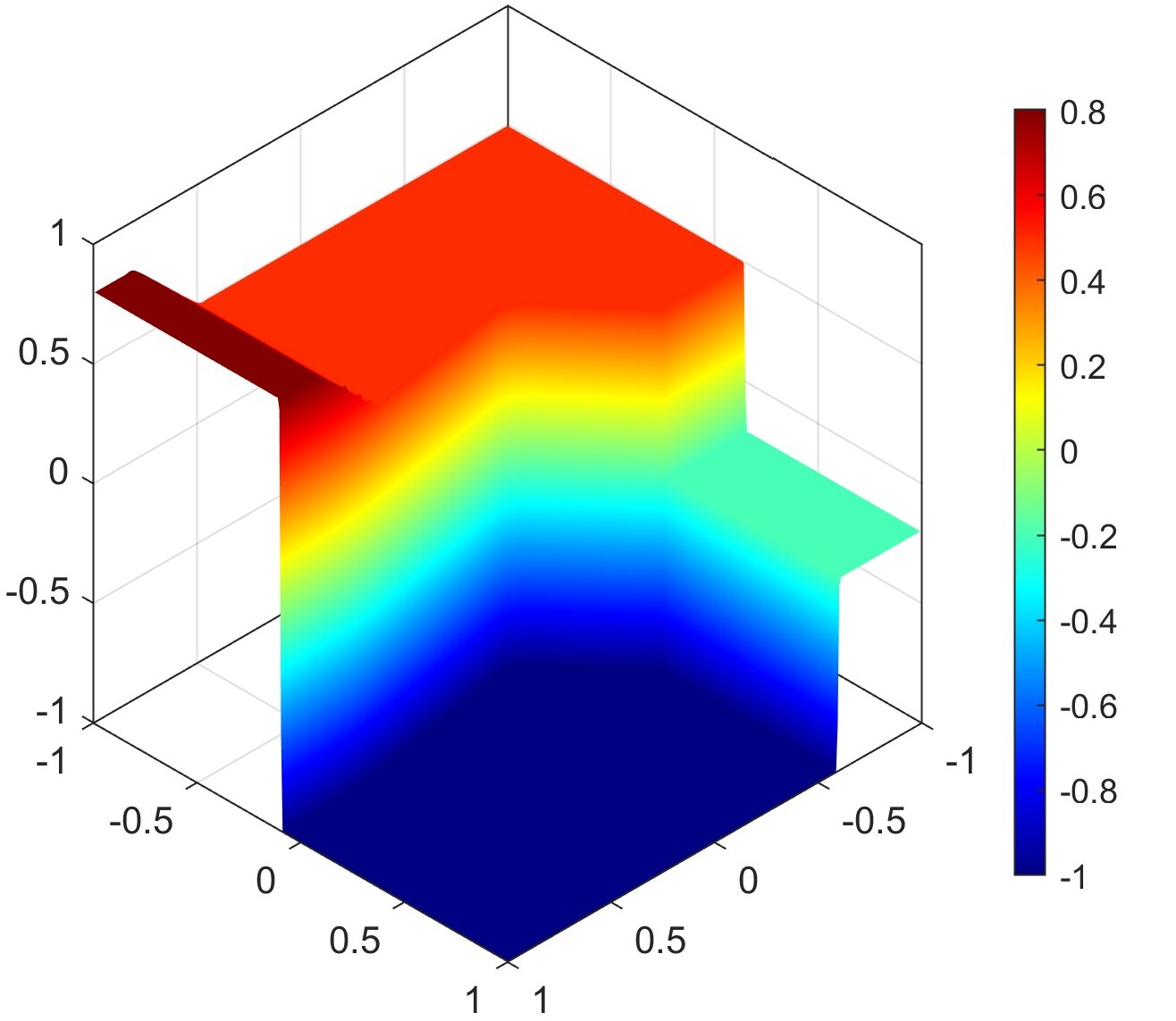}\hfill 
		\includegraphics[width=.32\textwidth]{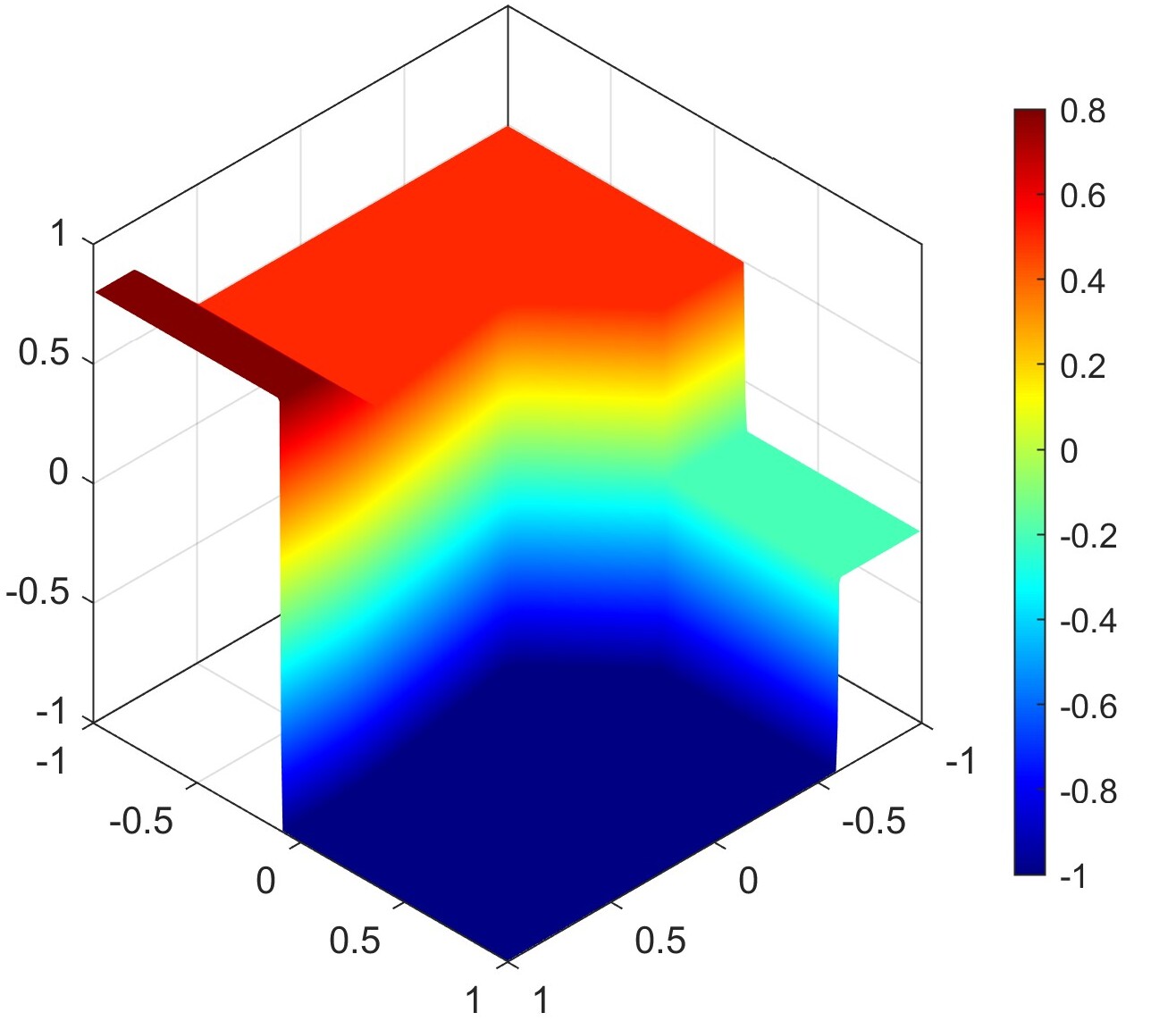}\hfill 
		\includegraphics[width=.32\textwidth]{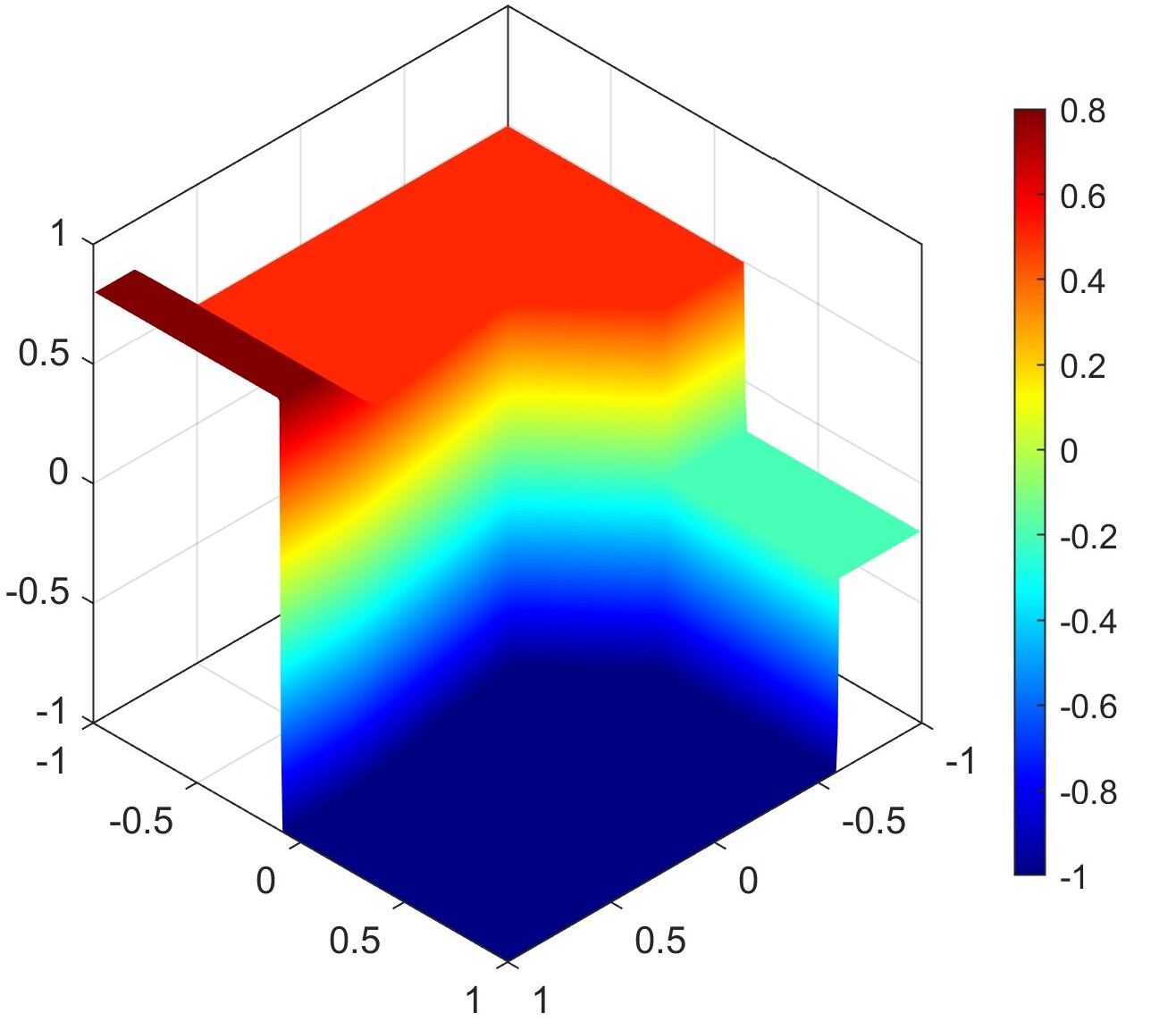} 

		\caption{OEDG solutions at $t=0.5$ for \Cref{ex:2DburgersRP}. From left to right: $\mathbb P^1$, $\mathbb P^2$, and $\mathbb P^3$ elements. 
		}
		\label{2DburgersRie}
	\end{figure}
\end{exmp}
\subsection{2D compressible Euler equations}
This subsection presents several benchmark numerical tests for the 2D compressible Euler equations. These equations can be expressed in the form of \eqref{eq:2DHCL} with ${\bf u}=(\rho, \rho {\bf v}, E)^\top$ and 
${\bf f} ({\bf u}) = (\rho {\bf v}, \rho {\bf v} \otimes {\bf v} + p {\bf I}, (E+p) {\bf v} )^\top$. 
Here $\rho$ represents the density, ${\bf v}$ denotes the velocity field, and $p$ is the pressure. The total energy is given by $E=\frac{1}{2} \rho |{\bf v}|^2+\frac{p}{\gamma -1}$, and the adiabatic index $\gamma$ is set to $1.4$ unless specified otherwise.

\begin{exmp}[smooth problem]\label{ex:EulerSmooth}
	The exact solution of this problem is given by 
	$$(\rho,{\bf v},p) = (1+0.2\sin(\pi(x+y-t)),0.7,0.3,1),$$ 
	which describes a sine wave that propagates periodically within the domain $[0,2]^2$.  
	Table \ref{2Deuler} shows the numerical errors and corresponding convergence orders for the density at $t=2$ obtained by the 2D $\mathbb P^k$-based OEDG method with $(k+1)$th-order RK time discretization. 
	One can observe that the convergence orders are above the theoretical order of $k+1$,  
	 incrementally nearing $k+1$ as the mesh is refined.  

%
%

		\begin{table}[!htb] 
			\caption{Errors and convergence rates for 2D $\mathbb P^k$-based OEDG method for \Cref{ex:EulerSmooth}.}\label{2Deuler}
			
			\begin{center}
				\centering
				\begin{tabular}{c|c|c|c|c|c|c|c} 
					\bottomrule[1.0pt]
					&	$N_x\times N_y$ &$L^1$ error& rate & $L^2$ error & rate & $L^\infty$ error & rate   \\ \hline \multirow{5}*{$1$}
					
					&80$\times$80&1.83e-4&- &2.08e-4&- &4.39e-4&- \\ 
					&160$\times$160&2.52e-5&2.86 &2.94e-5&2.82 &6.81e-5&2.69 \\
					&320$\times$320&3.85e-6&2.71 &4.60e-6&2.68 &1.17e-5&2.54 \\ 
					&640$\times$640&7.08e-7&2.44 &8.83e-7&2.38 &2.27e-6&2.37 \\ 
					&1280$\times$1280&1.53e-7&2.21 &2.00e-7&2.14 &4.85e-7&2.23 \\ 
					
					\hline \multirow{5}*{$2$}
					
					&80$\times$80&6.10e-4&-&7.85e-4&-&1.54e-3&-\\ 
					&160$\times$160&3.82e-6&7.32&4.37e-6&7.49&7.94e-6&7.60\\ 
					&320$\times$320&1.56e-7&4.61&1.77e-7&4.63&3.36e-7&4.56\\ 
					&640$\times$640&9.14e-9&4.09&1.06e-8&4.06&2.36e-8&3.83\\ 
					&1280$\times$1280&6.13e-10&3.90&7.57e-10&3.81&2.00e-9&3.56\\

					\hline \multirow{5}*{$3$}
					&80$\times$80&7.66e-7&- &1.01e-6&- &2.86e-6&- \\ 
					&160$\times$160&1.35e-8&5.83 &1.57e-8&6.01 &3.32e-8&6.43 \\ 
					&320$\times$320&3.98e-10&5.08 &4.50e-10&5.12 &8.73e-10&5.25 \\ 
					&640$\times$640&1.25e-11&4.99 &1.42e-11&4.99 &3.26e-11&4.74 \\

					\toprule[1.0pt]
				\end{tabular}
			\end{center}
		\end{table}
\end{exmp}

\begin{exmp}[shock reflection problem] 
	In this example, we simulate the shock reflection problem (cf.~\cite{zhu2017numerical}) in the domain $\Omega = [0,4]\times[0,1]$. 
	The initial conditions are defined as   
	$(\rho_0,{\bf v}_0,p_0) = (1,2.9,0,\frac{5}{7}),$ 
	which are also imposed as inflow condition on the left boundary of $\Omega$. 
	The upper boundary is also specified as inflow condition but with a different state 
	$$(\rho,{\bf v},p)=(1.69997, 2.61934, -0.50632, 1.52819).$$  
	The reflective wall boundary condition is applied to the lower boundary, while the outflow boundary condition is used on the right boundary. 
	The density contour plot at $t=20$, computed using the third-order OEDG scheme on a uniform rectangular mesh of $200 \times 50$ cells, are presented in Figure \ref{2D_shockreflectionwaveAA}. 
			As time evolves, the solution converges to a 
		steady 
		state. 
		To study the convergence behavior, we follow \cite{zhu2017numerical,LiuLuShu_OFDG_system} and compute the average residue as 
		\begin{equation}\label{eq:Res}
			{\rm Res} := \frac1{4N_xN_y} \sum_{i=1}^{N_x}\sum_{j=1}^{N_y}\sum_{q=1}^4\left|R^{(q)}_{ij} \right|,
		\end{equation}
		where $R^{(q)}_{ij}:= \frac{1}{\tau} ( u_h^{n+1,(q)} -u_h^{n,(q)}  ) $ denotes the local residue on the cell $[x_{i-\frac12},x_{i+\frac12}]\times [y_{j-\frac12},y_{j+\frac12}]$ for the $q$th component of the conservative vector ${\bf u}_h$. 
		   The convergence history of the average residue over time is plotted 
		Figure \ref{2D_shockreflectionwaveBB}. It is seen that the average residue reaches tiny values close to the machine epsilon in double precision and remains at that level after about $t = 9$.

	\begin{figure}[!htb]
		\centering
		\begin{subfigure}[h]{.66\linewidth}
			\includegraphics[width=0.96\textwidth]{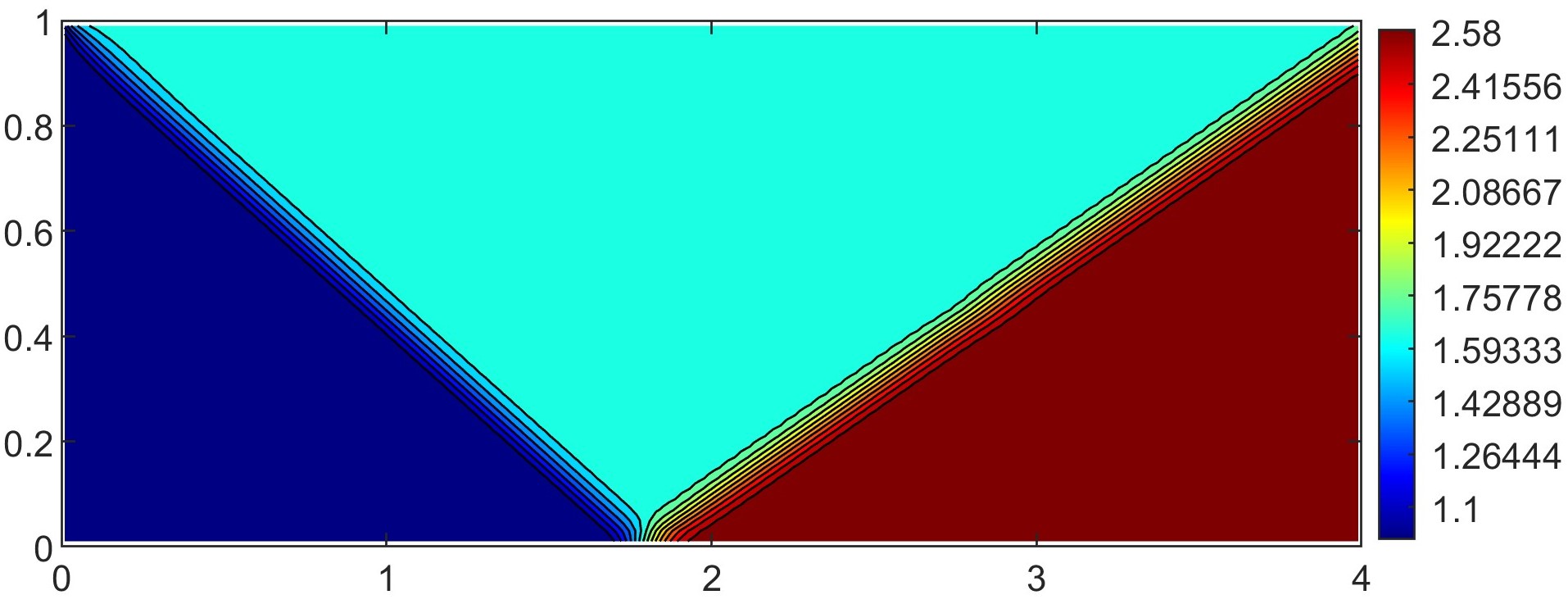}
			\subcaption{Density contour at $t=20$}
			\label{2D_shockreflectionwaveAA}
				\end{subfigure}
			\hfill 
				\begin{subfigure}[h]{.33\linewidth}
			\includegraphics[width=0.96\textwidth]{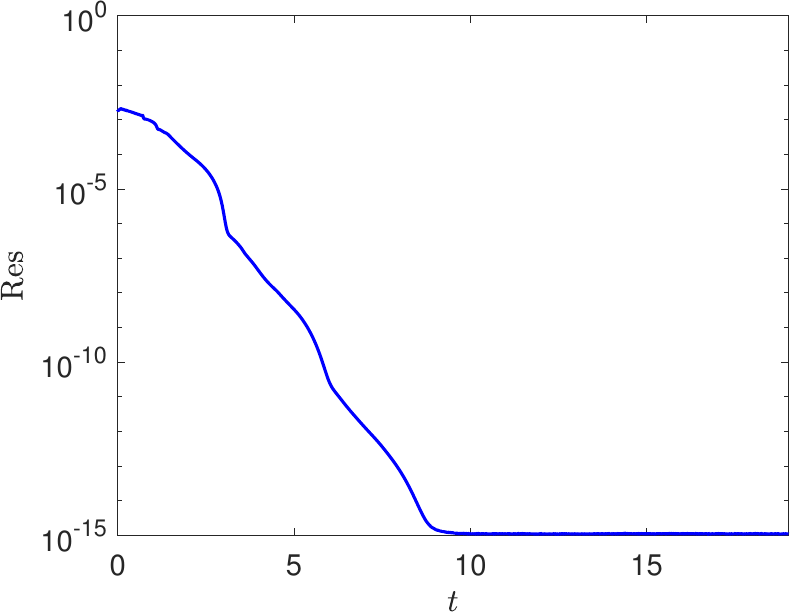}
			\subcaption{Average residue}
			\label{2D_shockreflectionwaveBB}
		\end{subfigure}
		\caption{Shock reflection problem simulated by third-order OEDG scheme.}
		\label{2D_shockreflectionwave}
	\end{figure}

\end{exmp}

\begin{exmp}[supersonic flow past two plates] 
	This example simulates a supersonic flow past two plates, with an attack angle of $15^{\circ}$, in the domain $[0,10] \times [-5,5]$ as detailed in \cite{zhu2017numerical}.
	The initial condition is given by  
	$$(\rho_0, {\bf v}_0,p_0) = \left(1,\cos(\frac{\pi}{12}),\sin(\frac{\pi}{12}),\frac{1}{\gamma M_{\infty}^2} \right),$$
	where $M_{\infty}=3$ represents the Mach number of the free stream. 
	The plates are situated at  $y=\pm2$ with $x\in(2,3)$, where slip boundary conditions are imposed. 
	The inflow boundary conditions are set on the left and lower boundaries, whereas the outflow conditions are specified on the upper and right boundaries. The simulation is conducted until 
	$t=100$ using the third-order OEDG scheme on a mesh of $200 \times 200$ uniform cells. 
	The results are illustrated in \Cref{2D_supersonicAA}, showing the numerical solution at the steady state at $t=100$. 
	We see that the flow structures are correctly captured by the OEDG method without nonphysical oscillations. 
	\Cref{2D_supersonicBB} presents 
	the convergence history of the average residue over time, which is computed by 
	\eqref{eq:Res}. 
	It is observed that the average residue settles down to tiny values around the machine epsilon in double precision. 
	\begin{figure}[!h]
		\centering
		\begin{subfigure}[h]{.52\linewidth}
			\centering 
			\includegraphics[width=0.9\textwidth]{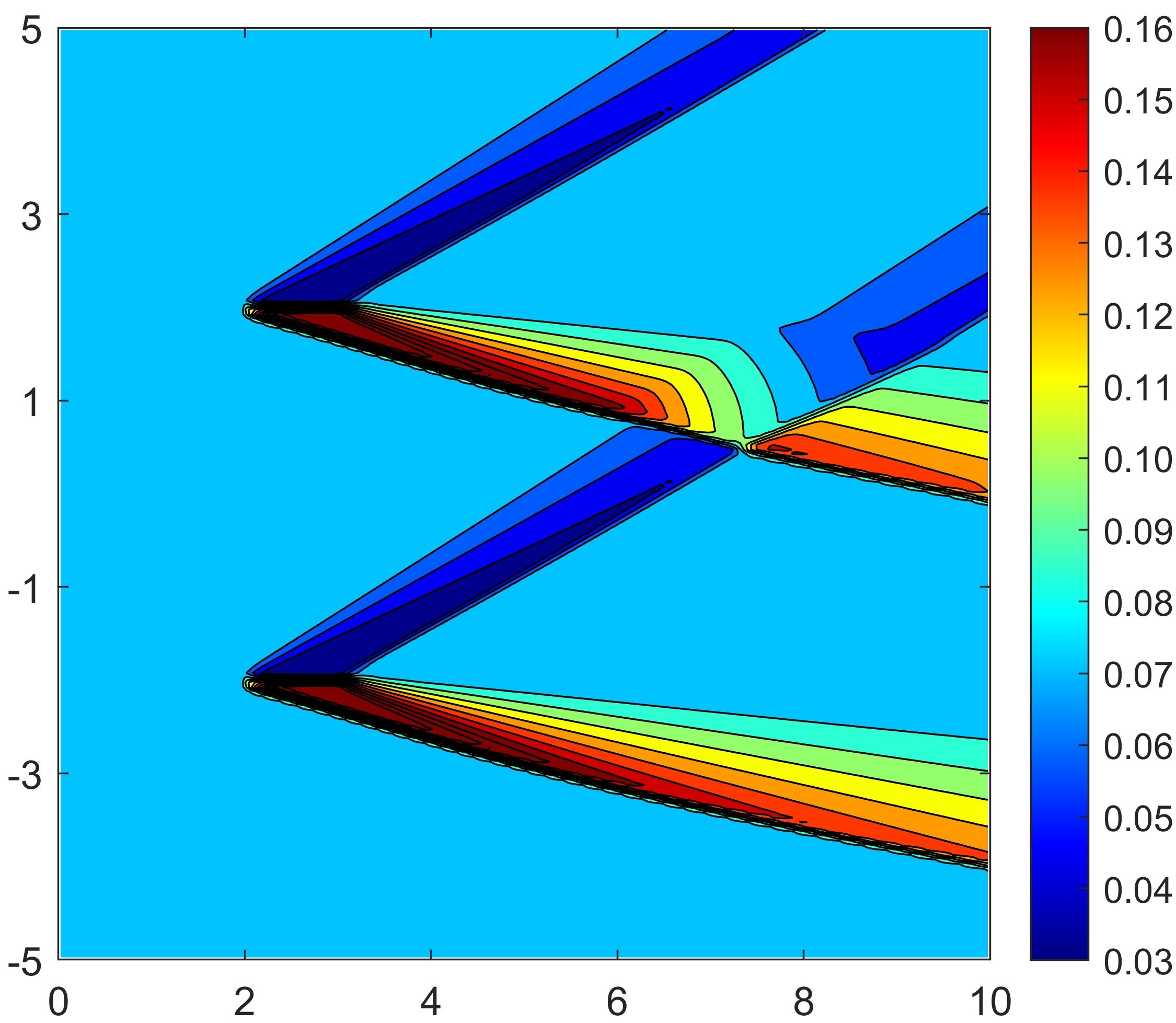}
			\subcaption{Density contour at $t=100$}
			\label{2D_supersonicAA}
		\end{subfigure}
		\hfill 
		\begin{subfigure}[h]{.46\linewidth}
			\centering 
			\includegraphics[width=0.9\textwidth]{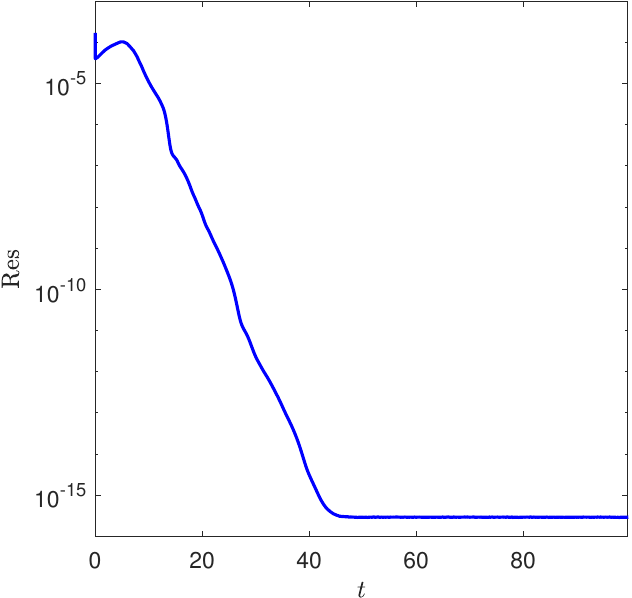}
			\subcaption{Average residue}
			\label{2D_supersonicBB}
		\end{subfigure}
		\caption{Supersonic flow problem simulated by third-order OEDG scheme.}
		\label{2D_supersonic}
	\end{figure}
\end{exmp}

\begin{figure}[!htb]
	\centering
	\begin{subfigure}[h]{.99\linewidth}
		\includegraphics[width=0.48\textwidth]{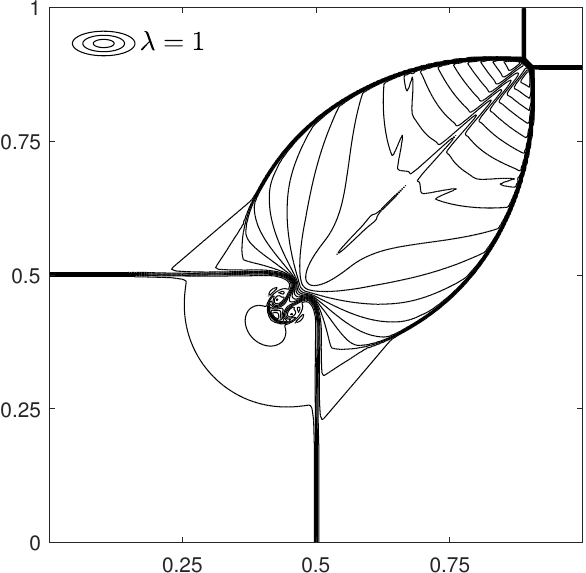} \hfill 
		\includegraphics[width=0.48\textwidth]{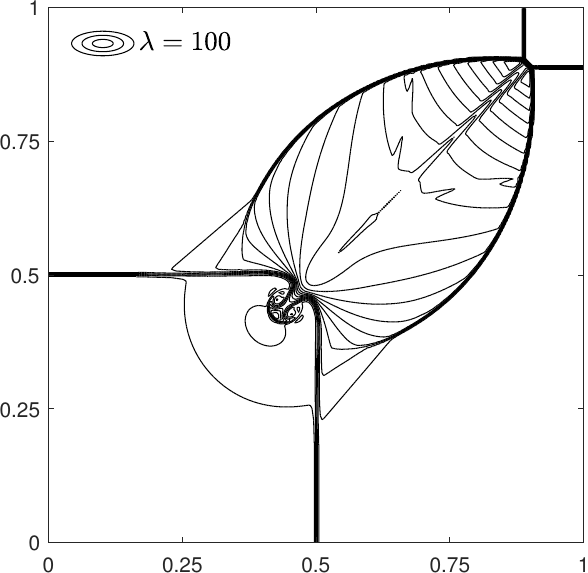}
		\subcaption{Proposed OEDG with scale-invariant damping}
		\vspace{3.6mm}
	\end{subfigure}
	\begin{subfigure}[h]{.99\linewidth}
		\includegraphics[width=0.48\textwidth]{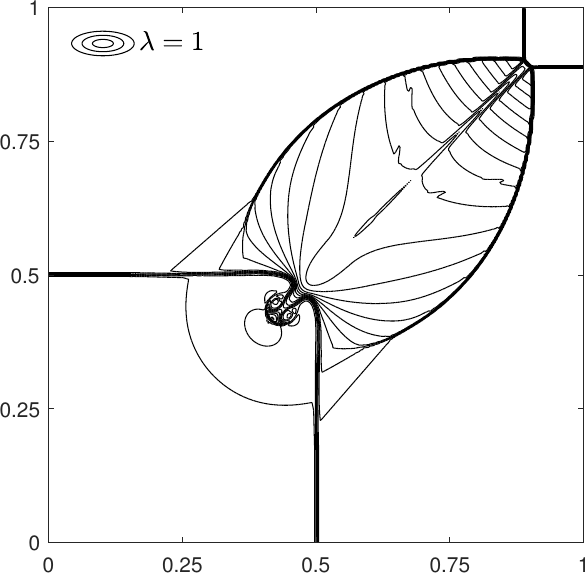} \hfill 
		\includegraphics[width=0.48\textwidth]{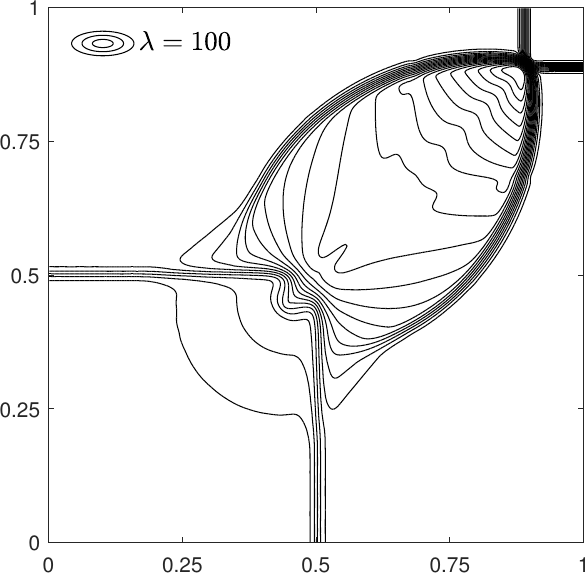}
		\subcaption{OFDG with non-scale-invariant damping}
	\end{subfigure}
	\caption{Contour plots of density at $t = 0.25$ computed by OEDG and OFDG methods. }
	\label{2D_rie1}
\end{figure}
\begin{exmp}[2D Riemann problems] 
	In this example, we investigate two classic 2D Riemann problems of the Euler equations in the domain $[0,1]\times [0,1]$ with outflow boundary conditions. 
	To check the scale-invariant property of the 2D OEDG method, we consider the scaled 
	 initial data ${\bf u}^\lambda(x,0)=\lambda {\bf u}_0(x)$, where the scaling represents the use of different units for ${\bf u}$. 
	For the first Riemann problem, ${\bf u}_0(x)$ is given by  
	$$(\rho_0, {\bf v}_0,p_0) = \begin{cases}
		(0.8,0,0,1), \quad  & x<0.5,y<0.5,\\
		(1, 0.7276, 0, 1), \quad  &x<0.5,y>0.5,\\
		(1, 0, 0.7276, 1), \quad & x>0.5,y<0.5,\\
		(0.5313, 0, 0, 0.4), \quad & x>0.5,y>0.5,\\
	\end{cases}
$$ 
	which involve two stationary contact discontinuities and two shocks. 
	The initial solution ${\bf u}_0(x)$ of the second Riemann problem is defined by 
	$$(\rho_0, {\bf v}_0, p_0) = \begin{cases}
		(0.138, 1.206, 1.206, 0.029), \quad & x<0.5,y<0.5,\\
		(0.5323, 1.206, 0, 0.3),  \quad & x<0.5,y>0.5,\\
		(0.5323, 0, 1.206, 0.3), \quad & x>0.5,y<0.5,\\
		(1.5, 0, 0, 1.5), \quad & x>0.5,y>0.5.\\
	\end{cases}$$ 
	\Cref{2D_rie1} presents the numerical results of the first Riemann problem simulated by the third-order OEDG and OFDG methods, respectively, on the rectangular mesh of $320\times 320$ uniform cells for two distinct cases ($\lambda=1$ and $\lambda=100$).  
	The results for the second Riemann problem are depicted in \Cref{2D_rie2}. 
	For the normal scale with $\lambda =1$, both  OEDG and OFDG methods 
	yield satisfactory results,  capturing the solution accurately without producing excessive dissipation or nonphysical oscillations. 
	Yet, when $\lambda =100$, the damping terms in the OFDG method escalate to $\lambda^2$ times its original magnitudes, causing excessive smearing and smoothing out 
	 the detailed features near discontinuities. 
	 Contrarily, the OEDG method consistently produces satisfactory results agreeing with those in the normal scale case, thanks to its scale-invariant advantage. 
	 It is worth mentioning that the results of OFDG method will be improved if our scale-invariant damping \eqref{eq:delta} is employed.

	\begin{figure}[!htb] 
		\centering
		\begin{subfigure}[h]{.99\linewidth}
			\includegraphics[width=0.48\textwidth]{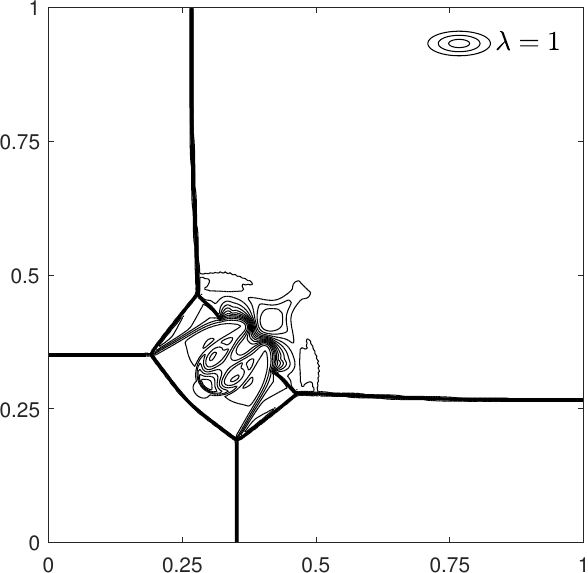} \hfill 
			\includegraphics[width=0.48\textwidth]{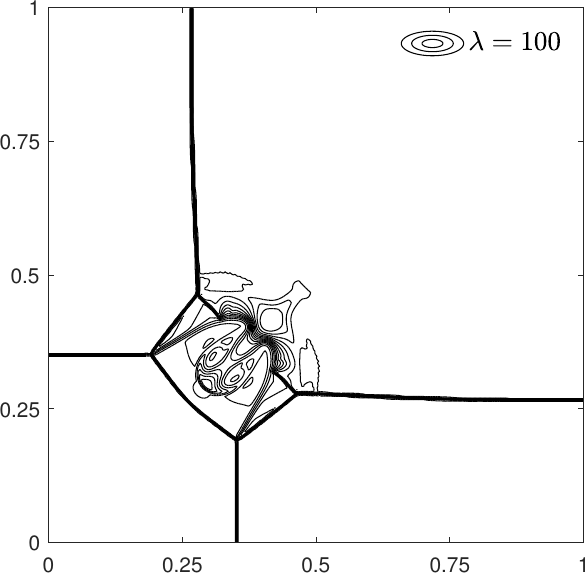}
			\subcaption{Proposed OEDG with scale-invariant damping}
			\vspace{3.6mm}
		\end{subfigure}
%
		\begin{subfigure}[h]{.99\linewidth}
			\includegraphics[width=0.48\textwidth]{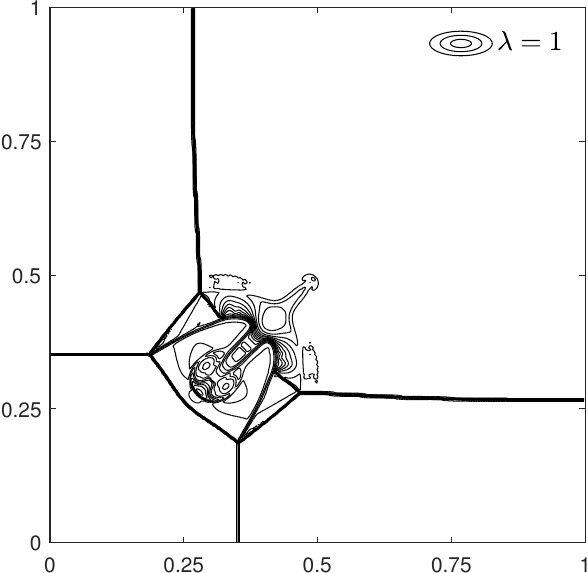} \hfill 
			\includegraphics[width=0.48\textwidth]{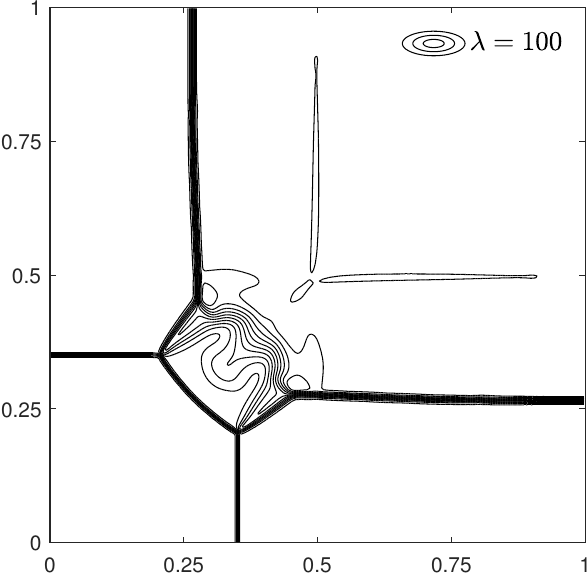}
			\subcaption{OFDG with non-scale-invariant damping}
		\end{subfigure}
		\caption{Contour plots of density at $t = 0.35$ computed by OEDG and OFDG methods. }
		\label{2D_rie2}
	\end{figure}
\end{exmp}

\begin{exmp}[shock-vortex interaction]
	Consider the interaction between a vortex and a Mach 1.1 shock, which is perpendicular to the $x$-axis and located at $x=0.5$. 
	The left state of the shock is $(\rho, {\bf v}, p) = (1,1.1\sqrt{\gamma},0,1)$, while the right state is  derived from the Rankine--Hugoniot condition. 
	Initially, an isentropic vortex, centered at $(x_c,y_c) = (0.25,0.5)$, is superimposed on the mean flow to the left of the shock.
	The velocity, temperature, and entropy perturbations due to the vortex are defined by 
	\begin{align*}
		\delta {\bf v} = \frac{\varepsilon}{r_c} e^{\alpha(1-\eta^2)}(\bar{y}, -\bar{x}), \quad
		\delta T = -\frac{(\gamma-1)\varepsilon^2}{4\alpha\gamma} e^{2\alpha(1-\eta^2)}, \quad
		\delta S =0,
	\end{align*}
	where $r^2= \bar{x}^2 + \bar{y}^2$, $\eta = \frac{r}{r_c}$, and $(\bar{x}, \bar{y}) = (x-x_c, y-y_c)$. 
	The parameters $\varepsilon=0.3$, $\alpha = 0.204$, and $r_c=0.05$ represent the vortex's strength, decay rate, and critical radius, respectively. 
	The computational domain is set as $[0,2]\times[0,1]$ and is divided into $400\times 200$ uniform  rectangular cells. Reflective boundary conditions are applied on both the upper and lower boundaries,  while the left and right boundaries utilize inflow and outflow conditions, respectively.  
	The third-order OEDG method is employed for simulations up to $t=0.8$. 
	Contour plots of pressure at six distinct times are illustrated in Figure \ref{2D_vortex}. 
	These results are in good agreement with those reported in \cite{LiuLuShu_OFDG_system}. 
	The OEDG method effectively captures the shock-vortex interaction dynamics and accurately depicts  intricate wave structures. Notably, the OEDG solution exhibits no nonphysical oscillations.


	\begin{figure}[!htb]
		\centering
		\begin{subfigure}[h]{.32\linewidth}
			\centering 
		\includegraphics[width=0.99\textwidth]{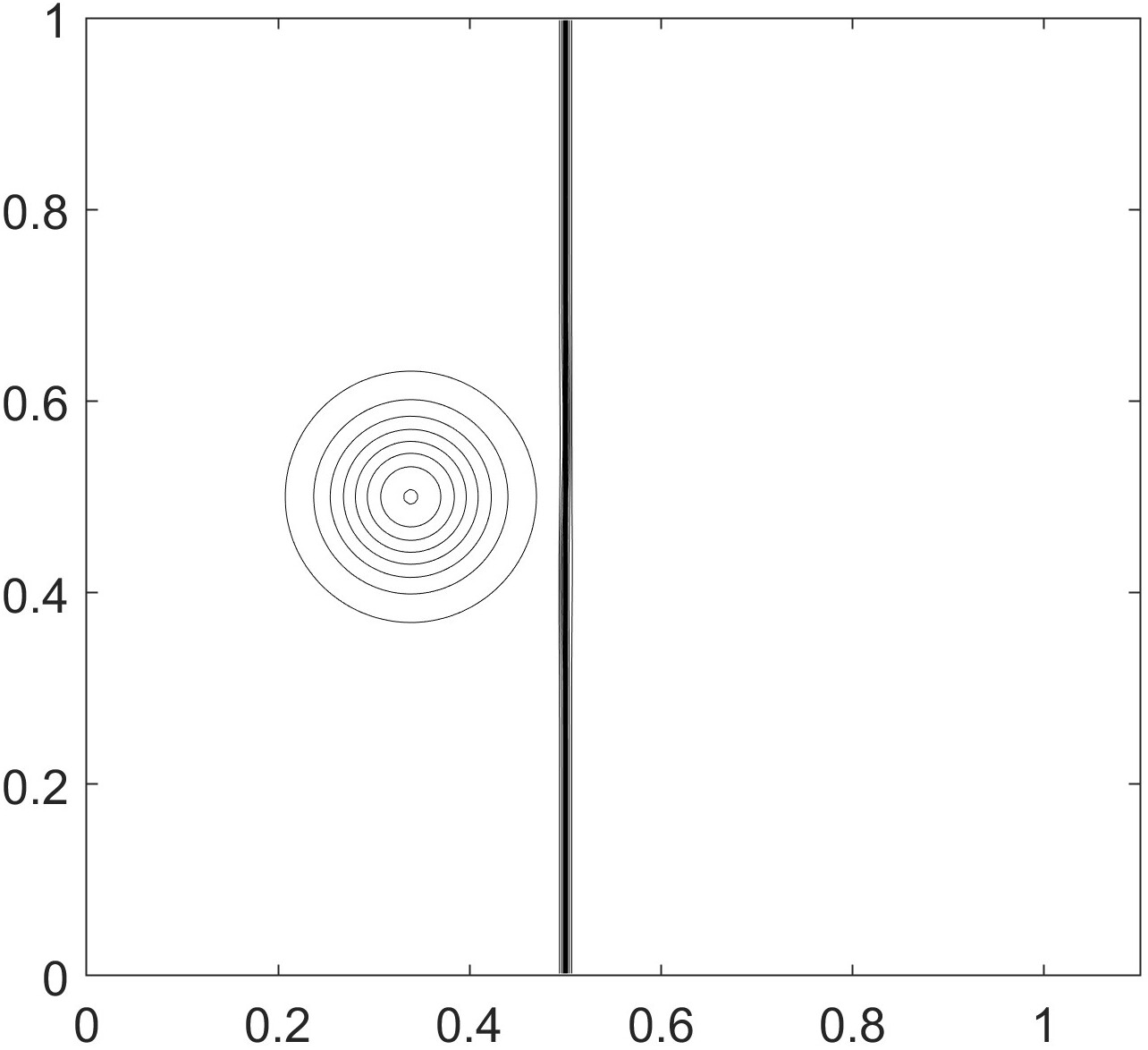}
		\subcaption{$t=0.068$}
	\end{subfigure}
\hfill 
		\begin{subfigure}[h]{.32\linewidth}
	\centering 
	\includegraphics[width=0.99\textwidth]{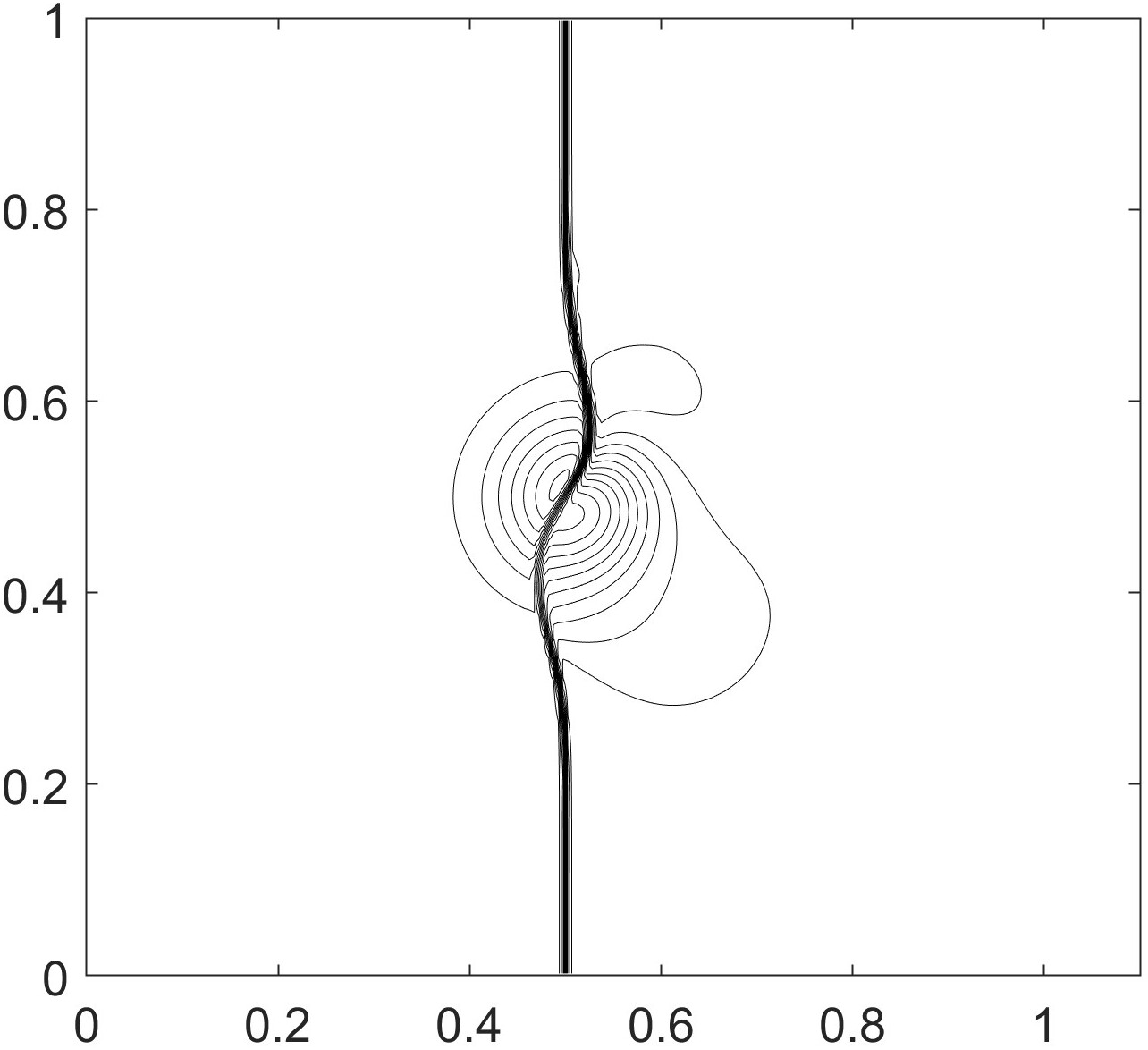}
	\subcaption{$t=0.203$}
\end{subfigure}
\hfill 
		\begin{subfigure}[h]{.32\linewidth}
	\centering 
	\includegraphics[width=0.99\textwidth]{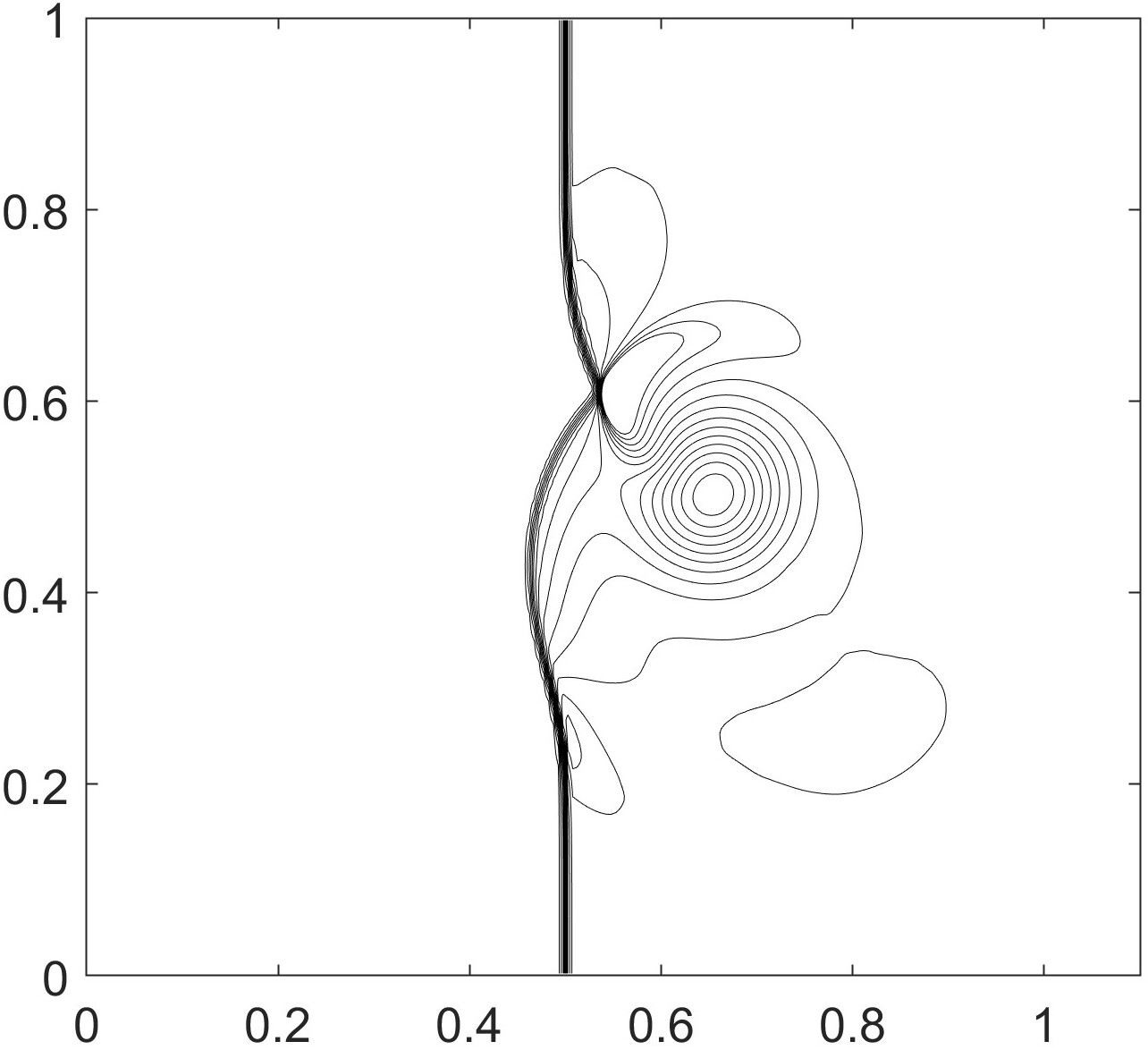}
	\subcaption{$t=0.330$}
	
\end{subfigure}
\vspace{3.6mm}

		\begin{subfigure}[h]{.32\linewidth}
	\centering 
	\includegraphics[width=0.99\textwidth]{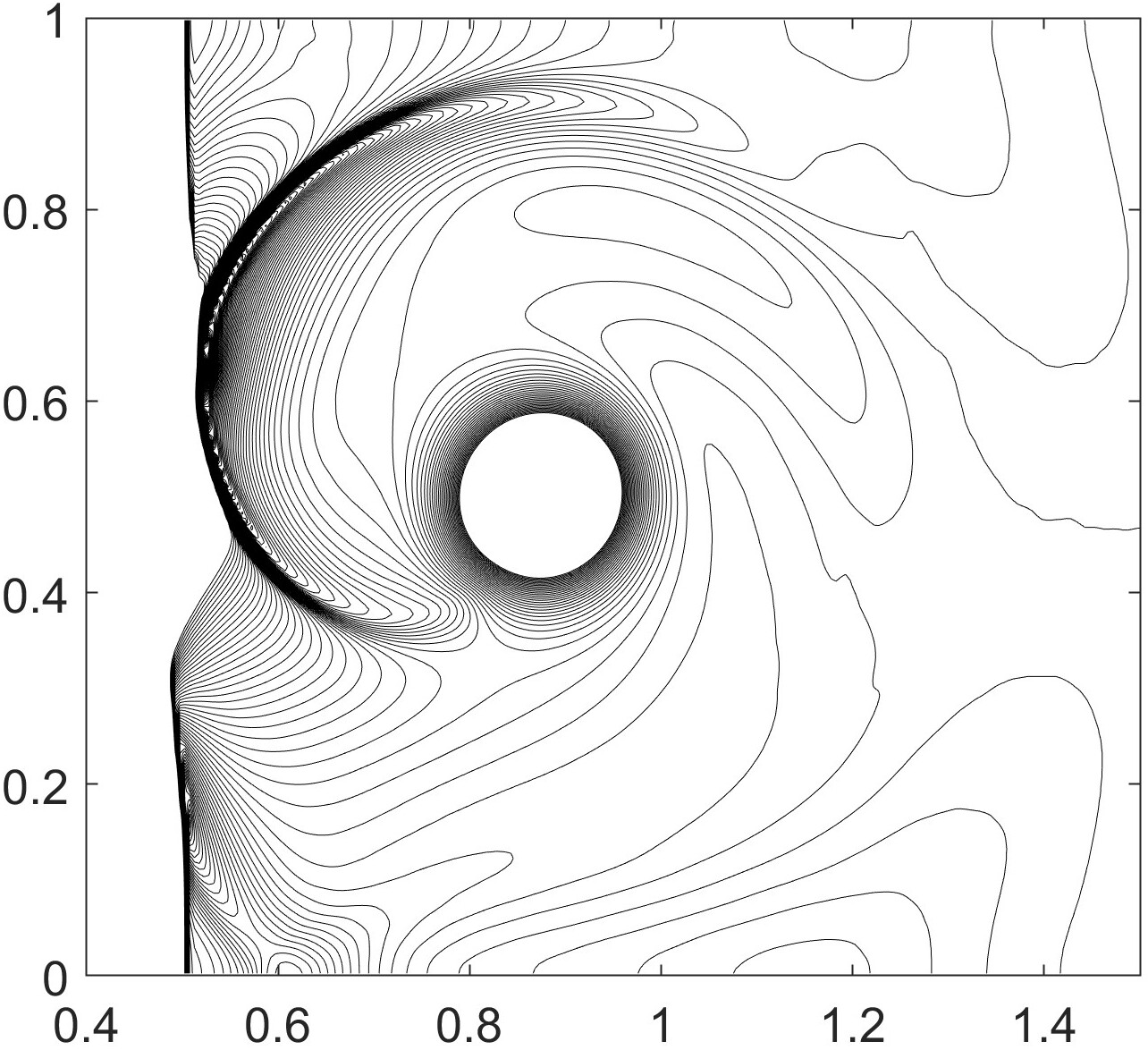}
	\subcaption{$t=0.529$}
\end{subfigure}
\hfill 
\begin{subfigure}[h]{.32\linewidth}
	\centering 
	\includegraphics[width=0.99\textwidth]{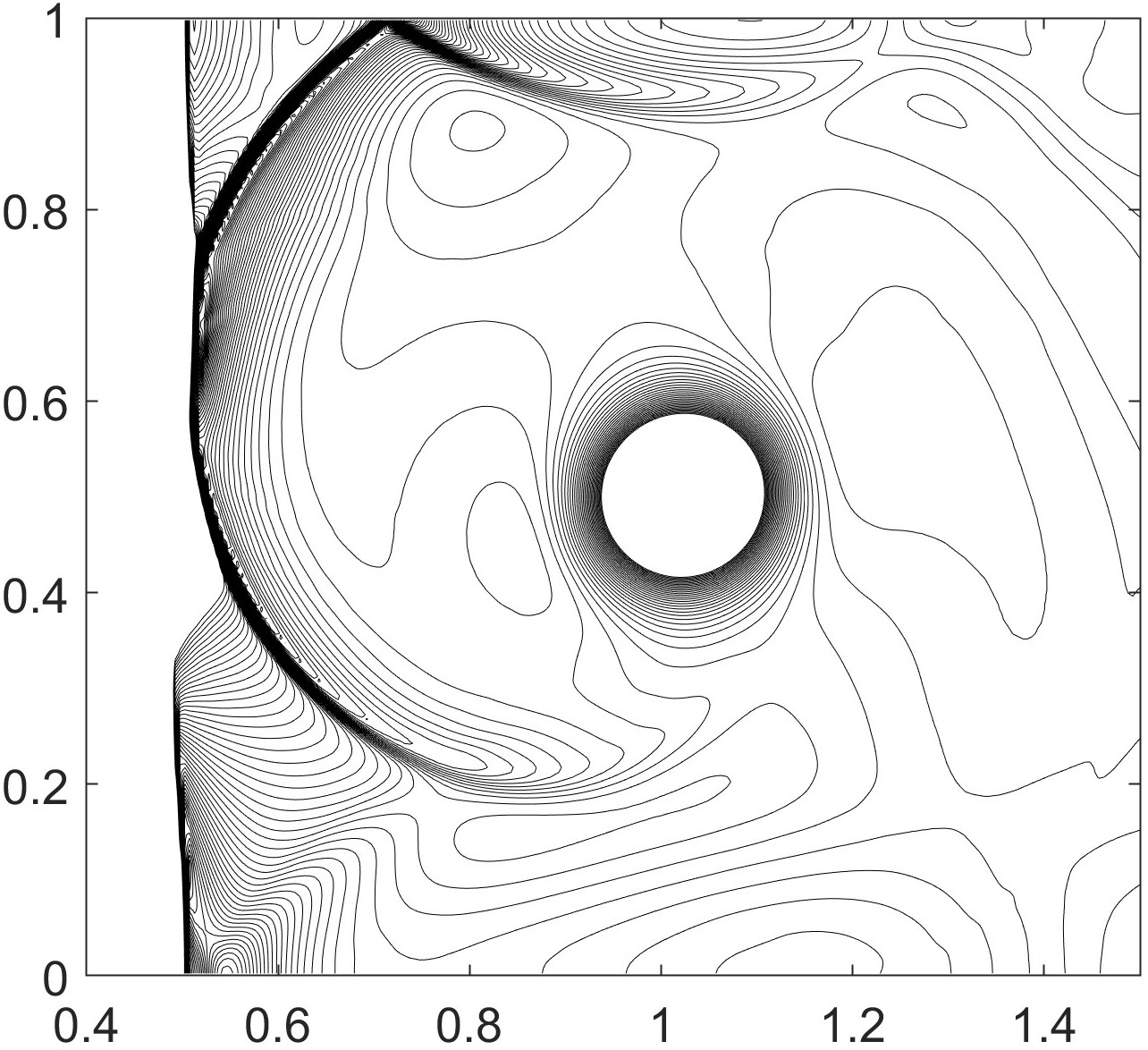}
	\subcaption{$t=0.662$}
\end{subfigure}
\hfill 
\begin{subfigure}[h]{.32\linewidth}
	\centering 
	\includegraphics[width=0.99\textwidth]{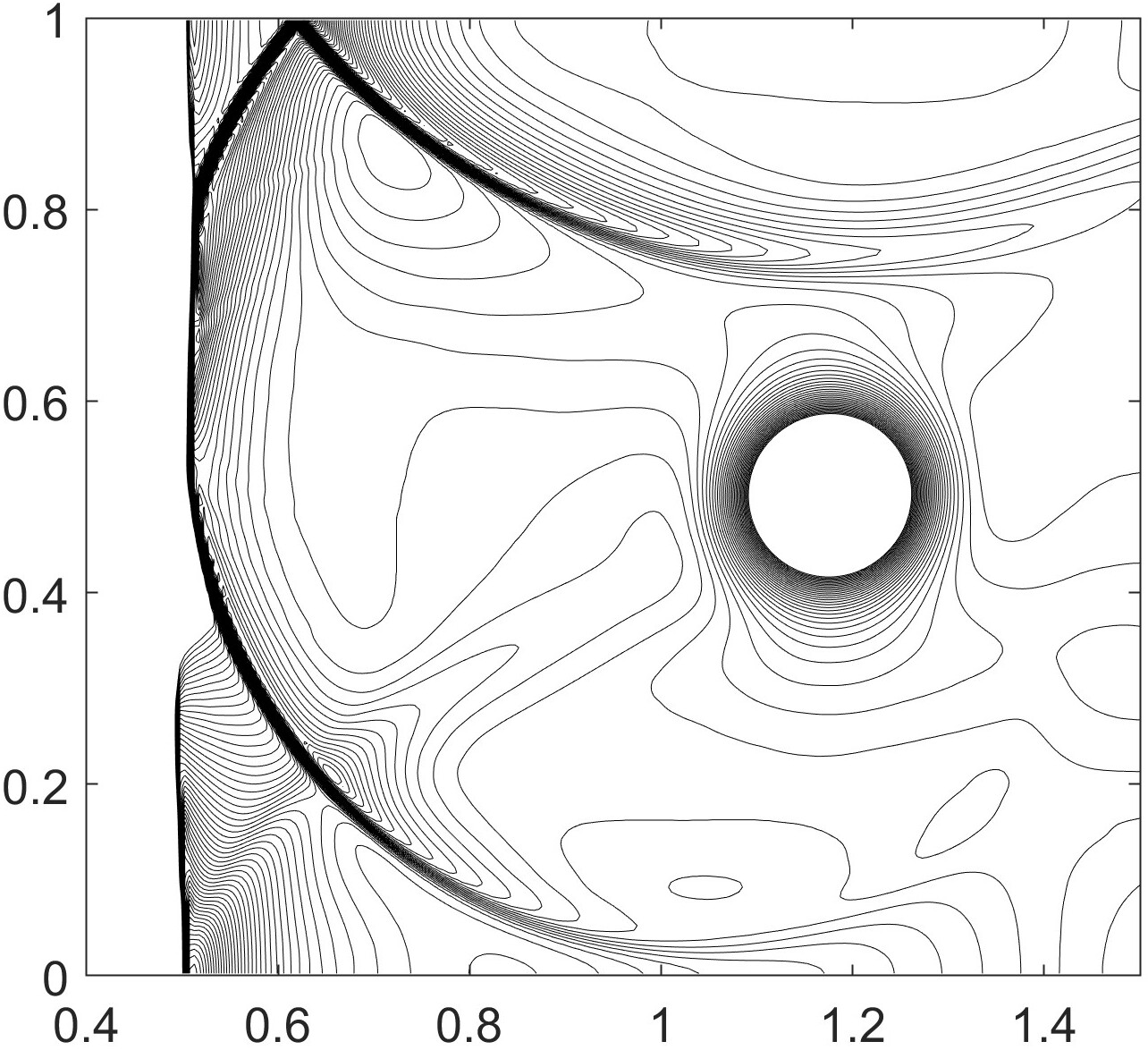}
	\subcaption{$t=0.8$}
\end{subfigure}

		\caption{Contour plots of pressure for shock-vortex interaction:  
		30 contour lines from 0.68 to 1.3 are shown in the top figures, while 90 contour lines from 1.19 to 1.36 are plotted in the bottom figures.}
		\label{2D_vortex}
	\end{figure}
\end{exmp}

\begin{exmp}[double Mach reflection]
	This is a classic test case in computational fluid dynamics for assessing the capabilities of numerical schemes in handling strong shocks and their interactions. 
	The computational domain is a rectangular region $\Omega = [0,4] \times [0,1]$.
	Initially, an oblique shock with a Mach number of 10 propagates to the right. This shock originates at $(x,y)=(1/6,0)$ 
	and forms an angle of $60^{\circ}$ relative to the bottom boundary. 
	 The states to the left and right of this shock are   
	 	$$(\rho_0,u_0,v_0,p_0) = \begin{cases}
	 	(8,8.25\cos(\frac{\pi}{6}),-8.25\sin(\frac{\pi}{6}),116.5),\quad &x<\frac{1}{6}+\frac{y}{\sqrt{3}},\\
	 	(1.4,0,0,1),\quad &x>\frac{1}{6}+\frac{y}{\sqrt{3}}.
	 \end{cases}$$ 
 	The left and right boundaries enforce inflow and outflow conditions, respectively. For the upper boundary, the postshock state is specified in the segment $x=0$ to $x=\frac16+\frac1{\sqrt{3}}(1+20t)$, while the preshock state is maintained for the remaining portion.  
 	On the lower boundary, the segment from 
 	$x=0$ to $x=1/6$ holds the postshock state, and a reflective boundary condition is applied to the remainder. 
 	We employ the proposed 
 	$\mathbb P^3$-based OEDG method for the simulation on the uniform rectangular mesh with $h_x=h_y=1/240$. 
 The resulting density contour plot at $t = 0.2$ is visualized in Figure \ref{2D_doublemach}. 
 The intricate flow characteristics, including the double Mach region, are finely delineated, with no nonphysical oscillations near the discontinuities observed. 
	
%

	\begin{figure}[!htb]
		\begin{subfigure}[h]{.685\linewidth}
			\centering
			\includegraphics[width=.99\textwidth]{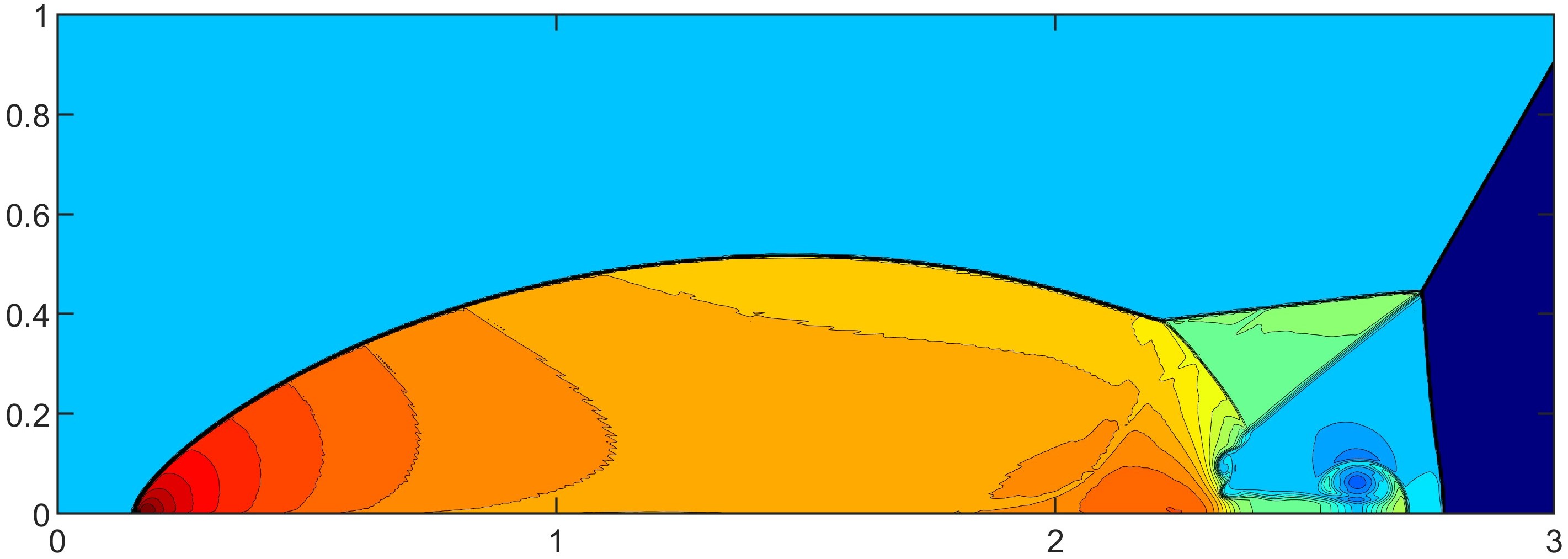}
		\end{subfigure}
	\hfill 
		\begin{subfigure}[h]{.29\linewidth}
			\includegraphics[width=.99\textwidth]{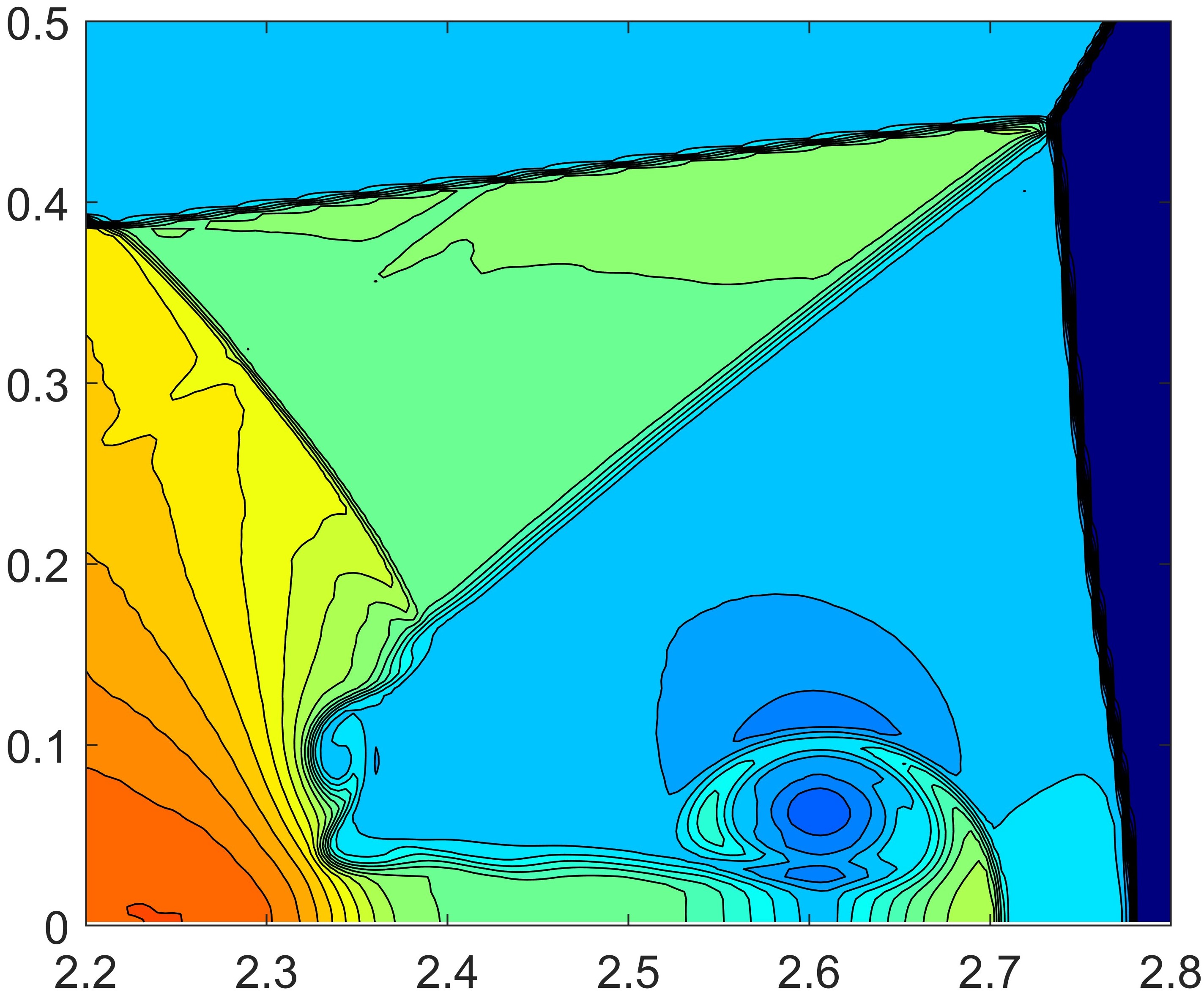}
		\end{subfigure}
	\caption{The contour plot of density (left) and its close-up (right) at $t=0.2$ obtained by the $\mathbb P^3$-based OEDG method with $h_x=h_y=1/240$.} 
		\label{2D_doublemach}
	\end{figure}
\end{exmp}

\begin{exmp}[Mach 2000 jet]
	In the last example, we further examine the robustness of the OEDG method through simulating a challenging jet problem \cite{zhang2010positivity,LiuLuShu_OFDG_system}. The ratio of heat capacity is set to be $\gamma = \frac{5}{3}$. 
	  The domain $\Omega = [0,1]\times[-0.25,0.25]$ is initially populated with a stationary fluid characterized by the state $(\rho,{\bf v},p) = (0.5,0,0,0.4127)$. A high-speed jet 
	state $(\rho,{\bf v},p) = (5,800,0,0.4127)$ is injected into $\Omega$ from the left boundary  within the range $y=-0.05$ to $0.05$. 
	Outflow conditions are applied to all remaining boundaries. 
	\Cref{2D_jet2000} presents the numerical results at $t = 0.001$ obtained by the third-order OEDG method with $320 \times 160$ cells. 
	The intricate structures of the jet flow, including the bow shock and shear layer, are captured with high resolution and agree with those computed in \cite{zhang2010positivity,LiuLuShu_OFDG_system}. 
	The OEDG method exhibits good robustness in this demanding test, and the computed solutions are free of  nonphysical oscillations.


	\begin{figure}[!htb]
		\centering
		\begin{subfigure}[h]{0.49\linewidth}
		\centering
		\includegraphics[width=1\textwidth]{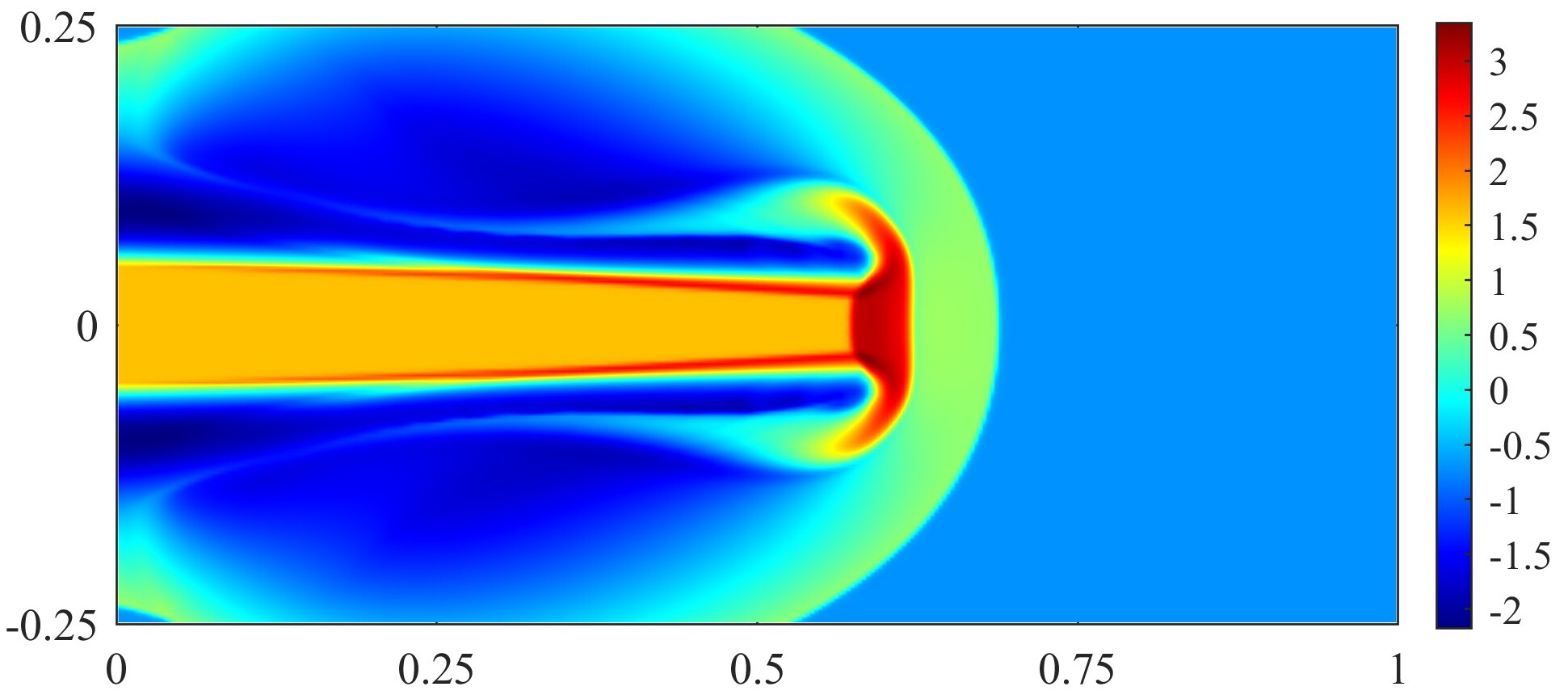}
		\subcaption{\small Density logarithm}
		\vspace{3mm}
		\end{subfigure}
	\hfill 
		\begin{subfigure}[h]{0.49\linewidth}
			\centering
			\includegraphics[width=1\textwidth]{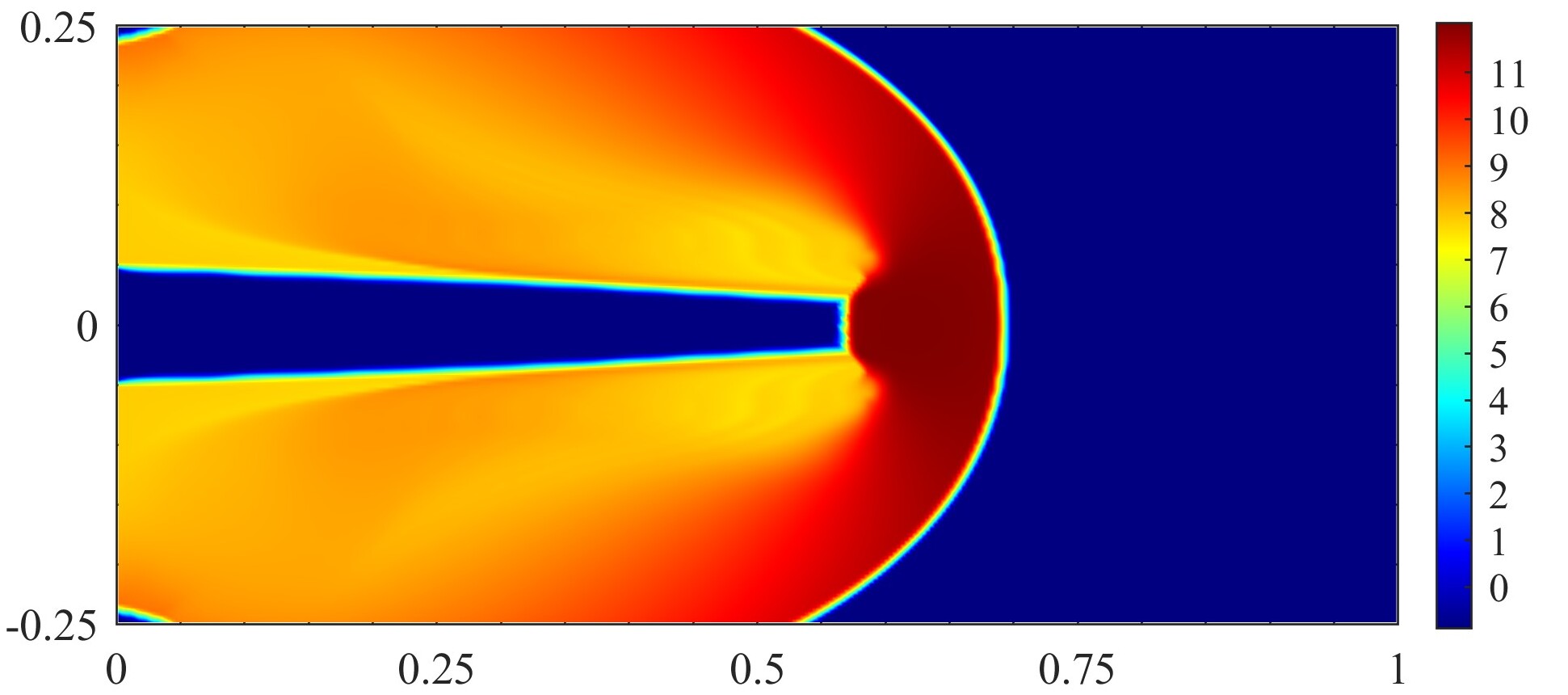}
			\subcaption{\small Pressure logarithm}
			\vspace{3mm}
		\end{subfigure}
	
			\begin{subfigure}[h]{0.49\linewidth}
		\centering
		\includegraphics[width=1\textwidth]{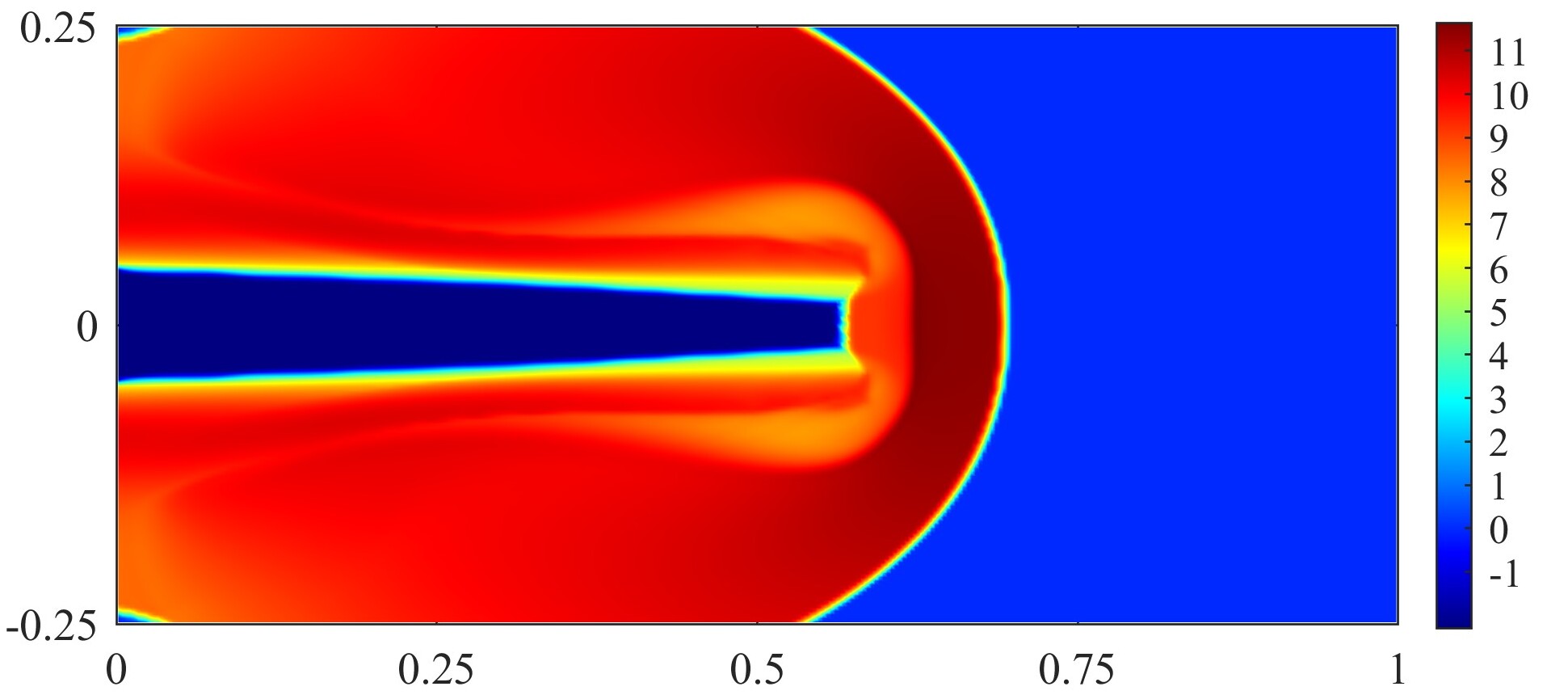}
		\subcaption{\small Temperature logarithm}
	\end{subfigure}
\hfill 
 			\begin{subfigure}[h]{0.49\linewidth}
 	\centering
 	\includegraphics[width=1\textwidth]{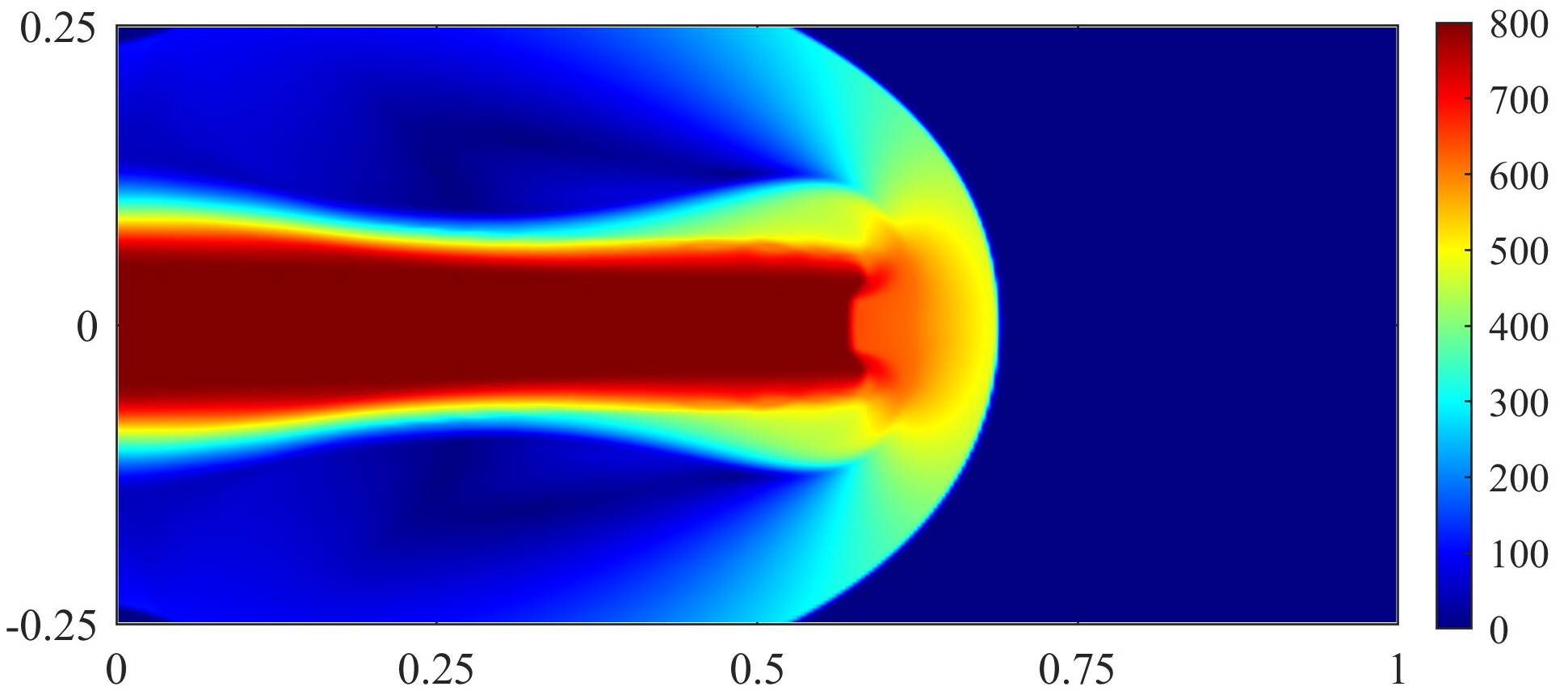}
 	\subcaption{\small Velocity magnitude}
 \end{subfigure}
		\caption{Numerical results at $t=0.001$ for the Mach 2000 jet problem.}
		\label{2D_jet2000}
	\end{figure}
\end{exmp}

\section{Conclusions}\label{sec:conclusion}

This paper has proposed the oscillation-eliminating discontinuous Galerkin (OEDG) method, a novel, robust, and efficient numerical approach for hyperbolic conservation laws. This method is inspired by the damping technique of Lu, Liu, and Shu \cite{lu2021oscillation,LiuLuShu_OFDG_system}. The core principle behind the OEDG approach involves an alternate progression between the conventional DG scheme and a new damping equation. This leads to an OE procedure effectively eliminating spurious oscillations subsequent to each Runge--Kutta stage. 
The new damping equation possesses both scale-invariant and evolution-invariant properties, pivotal in ensuring oscillation-free DG solutions across diverse scales and wave speeds. With the exact solver for the damping equation, the OE procedure's implementation becomes notably efficient, involving only simple  multiplications of modal coefficients by scalars. 
One significant contribution of our work is the rigorous optimal error estimates for the fully discrete OEDG method when applied to linear scalar conservation laws. To our knowledge, this might be the first effort on  fully-discrete error estimates for nonlinear DG methods with an automatic oscillation control mechanism. Furthermore, the OEDG method has provided fresh perspectives on the damping technique for oscillation control. 
It has revealed the role of the damping operator as a modal filter and bridges the damping and spectral viscosity techniques. 

The OEDG method offers several remarkable features. 
With its notable capacity to eliminate undesirable oscillations without necessitating problem-specific parameters, the OEDG method also eliminates the need for characteristic decomposition in hyperbolic systems. Furthermore, it retains many essential attributes of the conventional DG approach, such as conservation, optimal convergence rates, and superconvergence. 
Notably, even when confronted with strong shocks that lead to highly stiff damping terms, the OEDG method remains stable under the normal CFL condition. 
Another benefit of the OE procedure is its non-intrusive nature, allowing for easy integration into existing DG codes as an independent module. 
Our extensive numerical experiments have validated these theoretical findings, underscoring the effectiveness and superiority of the OEDG method on Cartesian meshes. The numerical tests on unstructured meshes will be reported in a separate paper.

It should be emphasized that the scale-invariant damping coefficients presented in \eqref{eq:1Dsigma} are not the sole choice. Indeed, they can be appropriately adjusted to preserve other important structures.  
As pointed out in \Cref{rem:LSI}, the exploration of locally scale-invariant OEDG schemes would be of significant interest. 
We intend to pursue these investigations in our subsequent research efforts.




\bibliography{refs}
\bibliographystyle{siamplain}
\end{document}

%% file: ex_shared.tex
\usepackage{lipsum}
\usepackage{amsfonts}
\usepackage{graphicx}
\usepackage{epstopdf}
\usepackage{algorithmic}
\ifpdf
\DeclareGraphicsExtensions{.eps,.pdf,.png,.jpg}
\else
\DeclareGraphicsExtensions{.eps}
\fi

\usepackage{booktabs}
\usepackage{hyperref}
\usepackage{stmaryrd}
\usepackage{bm}
\usepackage{caption}
\usepackage{subcaption}
\usepackage{multirow}
\usepackage{enumitem}

\usepackage{geometry}

\newsiamthm{assumption}{Assumption}
\newcommand{\nm}[1]{\left\|#1\right\|}

\def\jump#1{\llbracket #1 \rrbracket }
\newcommand{\mykappa}{m}
\newcommand{\ip}[2]{\left(#1,#2\right)}
\newcommand{\ipI}[2]{\left(#1,#2\right)_{I_j}}
\newcommand{\dg}[2]{\mathcal{H}\left(#1,#2\right)}

\newcommand{\sig}{\sigma}

\newcommand{\cF}{\mathcal{F}}
\newcommand{\cZ}{\mathcal{Z}}
\newcommand{\cY}{\mathcal{Y}}

\newcommand{\quand}{\quad\text{and}\quad}
\newcommand{\hf}{\frac{1}{2}}

\newcommand{\dth}{\frac{\mathrm{d}}{\mathrm{d}\hat{t}}}

\newsiamremark{remark}{Remark}
\newsiamremark{hypothesis}{Hypothesis}
\crefname{hypothesis}{Hypothesis}{Hypotheses}
\newsiamthm{claim}{Claim}
\newsiamremark{exmp}{Example}

\headers{}{}

\title{OEDG: Oscillation-eliminating discontinuous Galerkin\\ method for hyperbolic conservation laws
	\thanks{
			M.~Peng and K.~Wu was partially supported by Shenzhen Science and Technology Program (No.~RCJC20221008092757098) and 
			National Natural Science Foundation of China (No.~12171227). 
			Z.~Sun was partially supported by NSF grant DMS-2208391.}}

\author{Manting Peng\thanks{Department of Mathematics, Southern University of Science and Technology, Shenzhen, Guangdong 518055, China.}
	\and Zheng Sun\thanks{Department of Mathematics, The University of Alabama, Tuscaloosa, AL 35487, USA 
		(\email{zsun30@ua.edu}).}
	\and Kailiang Wu\thanks{Corresponding author. Department of Mathematics and SUSTech International Center for Mathematics, Southern University of Science and Technology, and National Center for Applied Mathematics Shenzhen (NCAMS), Shenzhen, Guangdong 518055, China (\email{wukl@sustech.edu.cn}).}} 

\usepackage{amsopn}


%% file: article_10072023.bbl
\begin{thebibliography}{10}

\bibitem{ai20222}
{\sc J.~Ai, Y.~Xu, C.-W. Shu, and Q.~Zhang}, {\em ${L}^{2}$ error estimate to
  smooth solutions of high order {R}unge--{K}utta discontinuous {G}alerkin
  method for scalar nonlinear conservation laws with and without sonic points},
  SIAM Journal on Numerical Analysis, 60 (2022), pp.~1741--1773.

\bibitem{becker2004two}
{\sc R.~Becker and M.~Braack}, {\em A two-level stabilization scheme for the
  {N}avier-{S}tokes equations}, in Numerical Mathematics and Advanced
  Applications: Proceedings of ENUMATH 2003 the 5th European Conference on
  Numerical Mathematics and Advanced Applications Prague, August 2003,
  Springer, 2004, pp.~123--130.

\bibitem{braack2006local}
{\sc M.~Braack and E.~Burman}, {\em Local projection stabilization for the
  {O}seen problem and its interpretation as a variational multiscale method},
  SIAM Journal on Numerical Analysis, 43 (2006), pp.~2544--2566.

\bibitem{chen2022physical}
{\sc Y.~Chen and K.~Wu}, {\em A physical-constraint-preserving finite volume
  {WENO} method for special relativistic hydrodynamics on unstructured meshes},
  Journal of Computational Physics, 466 (2022), p.~111398.

\bibitem{cheng2017application}
{\sc Y.~Cheng, X.~Meng, and Q.~Zhang}, {\em Application of generalized
  {Gauss--Radau} projections for the local discontinuous {G}alerkin method for
  linear convection-diffusion equations}, Mathematics of Computation, 86
  (2017), pp.~1233--1267.

\bibitem{rkdg4}
{\sc B.~Cockburn, S.~Hou, and C.-W. Shu}, {\em The {R}unge--{K}utta local
  projection discontinuous {G}alerkin finite element method for conservation
  laws. {I}{V}. the multidimensional case}, Mathematics of Computation, 54
  (1990), pp.~545--581.

\bibitem{rkdg3}
{\sc B.~Cockburn, S.-Y. Lin, and C.-W. Shu}, {\em {T}{V}{B} {R}unge--{K}utta
  local projection discontinuous {G}alerkin finite element method for
  conservation laws {I}{I}{I}: one-dimensional systems}, Journal of
  Computational Physics, 84 (1989), pp.~90--113.

\bibitem{rkdg2}
{\sc B.~Cockburn and C.-W. Shu}, {\em {T}{V}{B} {R}unge--{K}utta local
  projection discontinuous {G}alerkin finite element method for conservation
  laws. {I}{I}. general framework}, Mathematics of computation, 52 (1989),
  pp.~411--435.

\bibitem{rkdg1}
{\sc B.~Cockburn and C.-W. Shu}, {\em The {R}unge--{K}utta local projection $
  {P}^1$-discontinuous-{G}alerkin finite element method for scalar conservation
  laws}, ESAIM: Mathematical Modelling and Numerical Analysis, 25 (1991),
  pp.~337--361.

\bibitem{rkdg5}
{\sc B.~Cockburn and C.-W. Shu}, {\em The {R}unge--{K}utta discontinuous
  {G}alerkin method for conservation laws {V}: multidimensional systems},
  Journal of Computational Physics, 141 (1998), pp.~199--224.

\bibitem{dolejvsi2002finite}
{\sc V.~Dolej{\v{s}}{\'\i}, M.~Feistauer, and C.~Schwab}, {\em A finite volume
  discontinuous {G}alerkin scheme for nonlinear convection-diffusion problems},
  Calcolo, 39 (2002), pp.~1--40.

\bibitem{don2022novel}
{\sc W.~S. Don, R.~Li, B.-S. Wang, and Y.~Wang}, {\em A novel and robust
  scale-invariant {WENO} scheme for hyperbolic conservation laws}, Journal of
  Computational Physics, 448 (2022), p.~110724.

\bibitem{du2023oscillation}
{\sc J.~Du, Y.~Liu, and Y.~Yang}, {\em An oscillation-free bound-preserving
  discontinuous {G}alerkin method for multi-component chemically reacting
  flows}, Journal of Scientific Computing, 95 (2023), p.~90.

\bibitem{gottlieb2001spectral}
{\sc D.~Gottlieb and J.~S. Hesthaven}, {\em Spectral methods for hyperbolic
  problems}, Journal of Computational and Applied Mathematics, 128 (2001),
  pp.~83--131.

\bibitem{hesthaven2008filtering}
{\sc J.~Hesthaven and R.~Kirby}, {\em Filtering in {L}egendre spectral
  methods}, Mathematics of Computation, 77 (2008), pp.~1425--1452.

\bibitem{hesthaven2007spectral}
{\sc J.~S. Hesthaven, S.~Gottlieb, and D.~Gottlieb}, {\em Spectral methods for
  time-dependent problems}, vol.~21, Cambridge University Press, 2007.

\bibitem{hiltebrand2014entropy}
{\sc A.~Hiltebrand and S.~Mishra}, {\em Entropy stable shock capturing
  space-time discontinuous {G}alerkin schemes for systems of conservation
  laws}, Numerische Mathematik, 126 (2014), pp.~103--151.

\bibitem{huang2020adaptive}
{\sc J.~Huang and Y.~Cheng}, {\em An adaptive multiresolution discontinuous
  {G}alerkin method with artificial viscosity for scalar hyperbolic
  conservation laws in multidimensions}, SIAM Journal on Scientific Computing,
  42 (2020), pp.~A2943--A2973.

\bibitem{huang2017error}
{\sc J.~Huang and C.-W. Shu}, {\em Error estimates to smooth solutions of
  semi-discrete discontinuous {G}alerkin methods with quadrature rules for
  scalar conservation laws}, Numerical Methods for Partial Differential
  Equations, 33 (2017), pp.~467--488.

\bibitem{huang2018bound}
{\sc J.~Huang and C.-W. Shu}, {\em Bound-preserving modified exponential
  {R}unge--{K}utta discontinuous {G}alerkin methods for scalar hyperbolic
  equations with stiff source terms}, Journal of Computational Physics, 361
  (2018), pp.~111--135.

\bibitem{LiuLuShu_OFDG_system}
{\sc Y.~Liu, J.~Lu, and C.-W. Shu}, {\em An essentially oscillation-free
  discontinuous {G}alerkin method for hyperbolic systems}, SIAM Journal on
  Scientific Computing, 44 (2022), pp.~A230--A259.

\bibitem{liu2022oscillation}
{\sc Y.~Liu, J.~Lu, Q.~Tao, and Y.~Xia}, {\em An oscillation-free discontinuous
  {G}alerkin method for shallow water equations}, Journal of Scientific
  Computing, 92 (2022), p.~109.

\bibitem{liu2020optimal}
{\sc Y.~Liu, C.-W. Shu, and M.~Zhang}, {\em Optimal error estimates of the
  semidiscrete discontinuous {G}alerkin methods for two dimensional hyperbolic
  equations on {C}artesian meshes using ${P^k}$ elements}, ESAIM: Mathematical
  Modelling and Numerical Analysis, 54 (2020), pp.~705--726.

\bibitem{lu2021oscillation}
{\sc J.~Lu, Y.~Liu, and C.-W. Shu}, {\em An oscillation-free discontinuous
  {G}alerkin method for scalar hyperbolic conservation laws}, SIAM Journal on
  Numerical Analysis, 59 (2021), pp.~1299--1324.

\bibitem{lu2008explicit}
{\sc Y.~Lu, S.~Wong, M.~Zhang, C.-W. Shu, and W.~Chen}, {\em Explicit
  construction of entropy solutions for the {L}ighthill--{W}hitham--{R}ichards
  traffic flow model with a piecewise quadratic flow--density relationship},
  Transportation Research Part B: Methodological, 42 (2008), pp.~355--372.

\bibitem{meng2016optimal}
{\sc X.~Meng, C.-W. Shu, and B.~Wu}, {\em Optimal error estimates for
  discontinuous {G}alerkin methods based on upwind-biased fluxes for linear
  hyperbolic equations}, Mathematics of Computation, 85 (2016), pp.~1225--1261.

\bibitem{qiu2005runge}
{\sc J.~Qiu and C.-W. Shu}, {\em {Runge--Kutta} discontinuous {G}alerkin method
  using {W}{E}{N}{O} limiters}, SIAM Journal on Scientific Computing, 26
  (2005), pp.~907--929.

\bibitem{reed1973triangular}
{\sc W.~H. Reed and T.~Hill}, {\em Triangular mesh methods for the neutron
  transport equation}, tech. report, Los Alamos Scientific Lab., N. Mex.(USA),
  1973.

\bibitem{shu2009discontinuous}
{\sc C.-W. Shu}, {\em Discontinuous {G}alerkin methods: general approach and
  stability}, Numerical solutions of partial differential equations, 201
  (2009).

\bibitem{sun2017rk4}
{\sc Z.~Sun and C.-W. Shu}, {\em Stability of the fourth order {R}unge--{K}utta
  method for time-dependent partial differential equations}, Annals of
  Mathematical Sciences and Applications, 2 (2017), pp.~255--284.

\bibitem{sun2019strong}
{\sc Z.~Sun and C.-W. Shu}, {\em Strong stability of explicit {R}unge--{K}utta
  time discretizations}, SIAM Journal on Numerical Analysis, 57 (2019),
  pp.~1158--1182.

\bibitem{sun2022energy}
{\sc Z.~Sun, Y.~Wei, and K.~Wu}, {\em On energy laws and stability of
  {Runge--Kutta} methods for linear seminegative problems}, SIAM Journal on
  Numerical Analysis, 60 (2022), pp.~2448--2481.

\bibitem{sun2023generalized}
{\sc Z.~Sun and Y.~Xing}, {\em On generalized {Gauss--Radau} projections and
  optimal error estimates of upwind-biased {DG} methods for the linear
  advection equation on special simplex meshes}, Journal of Scientific
  Computing, 95 (2023), p.~40.

\bibitem{tao2023oscillation}
{\sc Q.~Tao, Y.~Liu, Y.~Jiang, and J.~Lu}, {\em An oscillation free local
  discontinuous {G}alerkin method for nonlinear degenerate parabolic
  equations}, Numerical Methods for Partial Differential Equations, 39 (2023),
  pp.~3145--3169.

\bibitem{xu2020superconvergence}
{\sc Y.~Xu, X.~Meng, C.-W. Shu, and Q.~Zhang}, {\em Superconvergence analysis
  of the {R}unge--{K}utta discontinuous {G}alerkin methods for a linear
  hyperbolic equation}, Journal of Scientific Computing, 84 (2020), pp.~1--40.

\bibitem{xu2020error}
{\sc Y.~Xu, C.-W. Shu, and Q.~Zhang}, {\em Error estimate of the fourth-order
  {Runge--Kutta} discontinuous {G}alerkin methods for linear hyperbolic
  equations}, SIAM Journal on Numerical Analysis, 58 (2020), pp.~2885--2914.

\bibitem{xu2019l2}
{\sc Y.~Xu, Q.~Zhang, C.-W. Shu, and H.~Wang}, {\em The ${L}^{2}$-norm
  stability analysis of {R}unge--{K}utta discontinuous {G}alerkin methods for
  linear hyperbolic equations}, SIAM Journal on Numerical Analysis, 57 (2019),
  pp.~1574--1601.

\bibitem{yu2020study}
{\sc J.~Yu and J.~S. Hesthaven}, {\em A study of several artificial viscosity
  models within the discontinuous {G}alerkin framework}, Communications in
  Computational Physics, 27 (2020), pp.~1309--1343.

\bibitem{zhang2006error}
{\sc Q.~Zhang and C.-W. Shu}, {\em Error estimates to smooth solutions of
  {Runge--Kutta} discontinuous {G}alerkin method for symmetrizable systems of
  conservation laws}, SIAM Journal on Numerical Analysis, 44 (2006),
  pp.~1703--1720.

\bibitem{zhang2010positivity}
{\sc X.~Zhang and C.-W. Shu}, {\em On positivity-preserving high order
  discontinuous {G}alerkin schemes for compressible {E}uler equations on
  rectangular meshes}, Journal of Computational Physics, 229 (2010),
  pp.~8918--8934.

\bibitem{zhong2013simple}
{\sc X.~Zhong and C.-W. Shu}, {\em A simple weighted essentially nonoscillatory
  limiter for {R}unge--{K}utta discontinuous {G}alerkin methods}, Journal of
  Computational Physics, 232 (2013), pp.~397--415.

\bibitem{zhu2017numerical}
{\sc J.~Zhu and C.-W. Shu}, {\em Numerical study on the convergence to steady
  state solutions of a new class of high order {WENO} schemes}, Journal of
  Computational Physics, 349 (2017), pp.~80--96.

\bibitem{zingan2013implementation}
{\sc V.~Zingan, J.-L. Guermond, J.~Morel, and B.~Popov}, {\em Implementation of
  the entropy viscosity method with the discontinuous {G}alerkin method},
  Computer Methods in Applied Mechanics and Engineering, 253 (2013),
  pp.~479--490.

\end{thebibliography}
